\documentclass[reqno]{amsart}
\usepackage{amsfonts}
\usepackage{amssymb}
\usepackage{amsmath}
\usepackage{amsthm}
\usepackage{amscd}
\usepackage{mathrsfs}
\usepackage{graphicx, color}
\usepackage{url}
\usepackage{tabularx}
\usepackage[normalem]{ulem}

\theoremstyle{plain}
\newtheorem{theorem}{Theorem}[section]

\newtheorem{lemma}[theorem]{Lemma}
\newtheorem{corollary}[theorem]{Corollary}
\newtheorem{proposition}[theorem]{Proposition}

\theoremstyle{definition}

\newtheorem{remark}[theorem]{Remark}

\numberwithin{equation}{section}
\numberwithin{figure}{section}
\numberwithin{table}{section}

\allowdisplaybreaks[1]

\begin{document}

\title[Lengths of 3-cocycles of 7-dihedral and octahedral quandles]{The lengths of 3-cocycles of \\ the 7-dihedral and the octahedral quandles}
\author{Ayumu Inoue}
\address{Department of Mathematics, Tsuda University, 2-1-1 Tsuda-machi, Kodaira-shi, Tokyo 187-8577, Japan}
\email{ayminoue@tsuda.ac.jp}


\dedicatory{Dedicated to Professor Tomotada Ohtsuki on the occasion of his 60th birthday}

\begin{abstract}
We determine the lengths of certain 3-cocycles of the 7-dihedral and the octahedral quandles.
As a consequence, we show that both of the 2-twist-spun $5_{2}$-knot and the 4-twist-spun trefoil have the triple point number eight.
\end{abstract}

\maketitle

\section{Introduction}
\label{sec:introduction}

Quandles and their homology are known to have really good chemistry with knot theory.
In particular, the lengths of 3-cocycles of quandles have been investigated by several authors to determine the triple point numbers of surface knots.
Here, a surface knot is a closed surface which is embedded in $\mathbb{R}^{4}$ smoothly.
The triple point number of a surface knot is defined in a similar way to the crossing number of a classical knot.
In general, it is difficult to determine the triple point number of a surface knot, as well as the crossing number of a classical knot.

In a celebrated paper \cite{SS2004}, Satoh and Shima essentially showed that the length of the Mochizuki 3-cocycle of the 3-dihedral quandle is equal to four, investigating ``networks of triple points''.
As a consequence, they proved that the 2-twist-spun trefoil has the triple point number four.
We note that this 2-twist-spun trefoil is the first surface knot whose triple point number is concretely determined except surface knots having the triple point number zero.
After that, Hatakenaka essentially showed in \cite{Hat2004} that the length of the Mochizuki 3-cocycle of the 5-dihedral quandle is at least six in a similar way.

The basic notion of the length of a 3-cocycle of a quandle is introduced by Satoh and Shima in \cite{SS2005}.
They showed that the length of a certain 3-cocycle of the tetrahedral quandle is equal to six, investigating ``graphs of 3-cycles'' of the tetrahedral quandle.
As a consequence, they proved that the 3-twist-spun trefoil has the triple point number six.
Then, Satoh organized arguments on the length, and showed that the length of the Mochizuki 3-cocycle of the 5-dihedral quandle is equal to eight in \cite{Sat2016}.
As a consequence, he proved that both of the 2-twist-spun $4_{1}$-knot and the 2-twist-spun $5_{1}$-knot have the triple point number eight.

In this paper, we show the following theorem in a similar way to \cite{Sat2016}.

\begin{theorem}
\label{thm:R7}
The length of the Mochizuki 3-cocycle $\zeta$ of the 7-dihedral quandle is equal to eight.
\end{theorem}

Moreover, developing Satoh's arguments in \cite{Sat2016}, we show the following theorem.

\begin{theorem}
\label{thm:O6}
The length of the 3-cocycle $\eta$, defined in Section \ref{sec:preliminaries}, of the octahedral quandle is equal to eight.
\end{theorem}

In light of the above theorems, we finally show the following corollaries.

\begin{corollary}
\label{cor:2-twist-spun_5_2}
The 2-twist-spun $5_{2}$-knot has the triple point number eight.
\end{corollary}

\begin{corollary}
\label{cor:4-twist-spun_3_1}
The 4-twist-spun trefoil has the triple point number eight.\footnote{Yashiro already showed that the $2 k$-twist-spun trefoil has the triple point number $4 k$ for each $k \geq 1$ in \cite{Yas2018}. However, there seems to be a slight leap in his proof (see Appendix \ref{app:counter_example}). Therefore, the author decided to provide an alternative proof for $k = 2$.}
\end{corollary}

This paper is organized as follows.
In Section \ref{sec:preliminaries}, we review the notion of the length of a cocycle of a quandle, and the definitions of a dihedral and the octahedral quandles.
Furthermore, we review the definition of the Mochizuki 3-cocycle of a dihedral quandle and introduce a 3-cocycle $\eta$ of the octahedral quandle.
In Section \ref{sec:tools}, we prepare some notions which are useful to discuss the length of a 3-cocycle of a quandle.
In Section \ref{sec:enumeration_of_f-connected_3-terms}, we enumerate 3-chains of a quandle satisfying a certain condition introduced in Section \ref{sec:tools}, for the subsequent arguments.
In Sections \ref{sec:lower_bound_of_zeta} and \ref{sec:lower_bound_of_eta}, we respectively give us lower bounds of the lengths of the Mochizuki 3-cocycle of the 7-dihedral quandle and the 3-cocycle $\eta$ of the octahedral quandle.
Since their arguments are quite complicated, we devote Sections \ref{sec:proof_of_lemma_O6_f2_special_2-5}--\ref{sec:proof_of_proposition_O6_VIII} to prove some claims which we need to achieve the latter lower bound.
Finally in Section \ref{sec:upper_bounds_of_zeta_and_eta}, we prove Theorems \ref{thm:R7} and \ref{thm:O6} in light of the lower bounds, and see that we naturally have Corollaries \ref{cor:2-twist-spun_5_2} and \ref{cor:4-twist-spun_3_1}.

\section{Preliminaries}
\label{sec:preliminaries}

In this section, we first review the notion of the length of a cocycle of a quandle in the style of Satoh's work \cite{Sat2016}.
Then, we introduce concrete quandles and their 3-cocycles on which we focus in this paper.
We refer the reader \cite{Kam2017} for more details on quandle theory.

A \emph{quandle} is a non-empty set $X$ equipped with a binary operation $(a, b) \mapsto a^{b}$ satisfying the following axioms.
\begin{itemize}
\item[(Q1)]
For each $a \in X$, $a^{a} = a$.
\item[(Q2)]
For each $y \in X$, the map $\square^{y} : X \rightarrow X$ ($x \mapsto x^{y}$) is bijective.
\item[(Q3)]
For each $a, b, c \in X$, $(a^{b})^{c} = (a^{c})^{b^{c}}$.
\end{itemize}
In the remaining, we abbreviate $(a^{b})^{c}$ as $a^{bc}$, $((a^{b})^{c})^{d}$ as $a^{bcd}$, and so on.
The notions of homomorphism and isomorphism are appropriately defined for quandles.

Let $X$ be a quandle.
Associated with $X$, consider a group $G(X)$ having the following presentation:
\[
 G(X) = \langle X \mid \{ b^{-1} a b (a^{b})^{-1} \mid a, b \in X \} \rangle.
\]
We call $G(X)$ the \emph{associated group} of $X$.
A set $S$ is said to be an \emph{$X$-set} if $S$ is equipped with a right action $(s, g) \mapsto s^{g}$ of $G(X)$.
We note that $X$ itself is a typical $X$-set, because $G(X)$ naturally acts on $X$ from the right considering each generator of $G(X)$ to be an element of the quandle $X$.

Let $S$ be an $X$-set, $C^{Q}_{m}(X)_{S}$ the free abelian group generated by the set
\[
 U^{Q}_{m}(X)_{S} = \{ (s; a_{1}, a_{2}, \dots, a_{m}) \in S \times X^{m} \mid a_{i} \neq a_{i+1} \ (1 \leq i \leq m - 1) \}
\]
for each $m \geq 1$, and $C^{Q}_{0}(X)_{S}$ the free abelian group generated by $S$.
We refer to an element of $C^{Q}_{m}(X)_{S}$ of the form $\pm u$ with some $u \in U^{Q}_{m}(X)_{S}$ as an \emph{$m$-term} of $(X, S)$.
Some $m$-terms $\gamma_{1}, \gamma_{2}, \dots, \gamma_{k}$ of $(X, S)$ is said to be \emph{efficient} if $\gamma_{i} \neq - \gamma_{j}$ for any $i \neq j$.
Obviously, each non-zero $m$-chain $\gamma$ of $(X, S)$ can be uniquely written as $\gamma = \sum_{i = 1}^{l} \gamma_{i}$, up to order of summation, with some efficient $m$-terms $\gamma_{1}, \gamma_{2}, \dots, \gamma_{l}$ of $(X, S)$.
We call the number $l$ the \emph{length} of $\gamma$.
We let $l(\gamma)$ denote the length of $\gamma$.
We further refer to $\sum_{i = 1}^{l} \gamma_{i}$ as a \emph{reduced form} of $\gamma$.

For each $m \geq 1$, define homomorphisms $f$ and $g$ from $C^{Q}_{m}(X)_{S}$ to $C^{Q}_{m - 1}(X)_{S}$ by
\begin{align*}
 f(s; a_{1}, a_{2}, \dots, a_{m})
 & = \sum_{i = 1}^{m} (-1)^{i} (s; a_{1}, \dots, a_{i-1}, a_{i+1}, \dots, a_{m}), \\
 g(s; a_{1}, a_{2}, \dots, a_{m})
 & = \sum_{i = 1}^{m} (-1)^{i+1} (s^{a_{i}}; a_{1}^{a_{i}}, \dots, a_{i-1}^{a_{i}}, a_{i+1}, \dots, a_{m}),
\end{align*}
removing all terms $\pm (t; b_{1}, b_{2}, \dots, b_{m - 1})$ in the right-hand sides of the formulae which satisfy $b_{j} = b_{j + 1}$ for some $j$ ($1 \leq j \leq m - 2$).
It is routine to check that the homomorphism $\partial = f + g$ from $C^{Q}_{m}(X)_{S}$ to $C^{Q}_{m - 1}(X)_{S}$ satisfies the condition $\partial \circ \partial = 0$.
Therefore, we have the chain complex $(C^{Q}_{\ast}(X)_{S}, \partial)$, and thus its homology or cohomology groups with coefficient in an abelian group.

Let $O$ be a set consisting of a single element $o$.
Then, $O$ is obviously equipped with the trivial right action of $G(X)$, which is given by $o^{g} = o$.
In the remaining, we abbreviate $(X, O)$ as $X$, $C^{Q}_{m}(X)_{O}$ as $C^{Q}_{m}(X)$ and $(o; a_{1}, a_{2}, \dots, a_{m}) \in C^{Q}_{m}(X)$ as $(a_{1}, a_{2}, \dots, a_{m})$ by tradition.
Obviously, we have a chain map $\pi : C^{Q}_{m}(X)_{S} \to C^{Q}_{m}(X)$ given by
\[
 \pi(s; a_{1}, a_{2}, \dots, a_{m}) = (a_{1}, a_{2}, \dots, a_{m}).
\]
For an $m$-cocycle $\theta$ of $X$, define the number $l(\theta, S)$ by
\[
 l(\theta, S)
 = \min \{ l(\gamma) \mid \text{$\gamma$ is an $m$-cycle of $(X, S)$ satisfying $\theta(\pi(\gamma)) \neq 0$} \},
\]
where we let $l(\theta, S) = 0$ if the set appearing in the right-hand side of the formula is empty.
Furthermore, define the number $l(\theta)$ by
\[
 l(\theta) = \max \{ l(\theta, S) \mid \text{$S$ is an $X$-set} \},
\]
where we let $l(\theta) = \infty$ if the set appearing in the right-hand side of the formula is unbounded as a subset of $\mathbb{R}^{1}$.
We call $l(\theta)$ the \emph{length} of $\theta$.

It is easy to see that, for each integer $n \geq 3$, the cyclic group $R_{n}$ of order $n$ is a quandle with the binary operation $a^{b} = 2 b - a$.
We call this quandle $R_{n}$ the \emph{$n$-dihedral quandle}.
It is originally shown by Mochizuki \cite{Moc2003} and explicitly claimed by Asami and Satoh \cite{AS2005} that a homomorphism $\zeta_{n} : C^{Q}_{3}(R_{n}) \to \mathbb{Z} / n \mathbb{Z}$ given by
\[
 \zeta_{n}(a, b, c)
 = (a - b) \dfrac{b^{n} + (2 c - b)^{n} - 2 c^{n}}{n}
\]
is a 3-cocycle of $R_{n}$ with coefficient in $\mathbb{Z} / n \mathbb{Z}$, if $n$ is odd prime.
We call this $\zeta_{n}$ the \emph{Mochizuki 3-cocycle}.
We let $\zeta$ denote $\zeta_{7}$ throughout this paper.

Consider a regular octahedron whose vertices are indexed by numbers as depicted in the left-hand side of Figure \ref{fig:octahedral_quandle}.
Let $O_{6} = \{ 0, 1, 2, 3, 4, 5 \}$ be the set consisting of vertices of the octahedron.
For each $b \in O_{6}$, let $l_{b}$ be the line passing through the center of the octahedron and $b$.
Define a binary operation on $O_{6}$ so that $a^{b}$ is the image of $a$ by the $\pi / 2$-rotation about $l_{b}$ (counterclockwise when we look the center from $b$).
For example, we have $1^{0} = 2$ by definition.
It is routine to check that this binary operation satisfies the axioms of a quandle.
We call this quandle $O_{6}$ the \emph{octahedral quandle}.
For each $a \in O_{6}$, let $[a + 3]$ denote the element of $O_{6}$ whose value as an integer is equivalent to $a + 3$ modulo 6.
We note that we have $a^{b} = a$ if and only if $b = a$ or $[a + 3]$.
Let $H_{0}$ be the subgroup of $G(O_{6})$ generated by $0$, which is cyclic of order four.
Define a homomorphism $\eta : C^{Q}_{3}(O_{6}) \to \mathbb{Z} / 3 \mathbb{Z}$ by
\begin{align*}
 \eta(a, b, c) =
 \begin{cases}
  1 & \text{if $(a, b, c) = (0^{h}, 1^{h}, 2^{h}), \, (0^{h}, 3^{h}, 1^{h}), \, (1^{h} ,2^{h}, 0^{h}), \, (1^{h}, 4^{h}, 2^{h}),$} \\
    & \text{\phantom{if $(a, b, c) = $} $(3^{h}, 1^{h}, 5^{h})$ \enskip ($h \in H_{0}$)}, \\
  2 & \text{if $(a, b, c) = (0^{h}, 1^{h}, 5^{h}), \, (1^{h}, 0^{h}, 1^{h}), \, (1^{h}, 0^{h}, 5^{h}), \, (1^{h}, 2^{h}, 1^{h}),$} \\
    & \text{\phantom{if $(a, b, c) = $} $(1^{h}, 3^{h}, 1^{h}), \, (1^{h}, 3^{h}, 2^{h}), \, (1^{h}, 4^{h}, 5^{h}), \, (1^{h}, 5^{h}, 1^{h}),$} \\
    & \text{\phantom{if $(a, b, c) = $} $(3^{h}, 0^{h}, 1^{h}), \, (3^{h}, 1^{h}, 0^{h}), \, (3^{h}, 1^{h},2^{h})$ \enskip ($h \in H_{0}$)}, \\
  0 & \text{otherwise}.
 \end{cases}
\end{align*}
Then, it is routine to check that $\eta$ is a 3-cocycle of $O_{6}$ with coefficient in $\mathbb{Z} / 3 \mathbb{Z}$.
\begin{figure}[htbp]
 \centering
 \includegraphics[scale=0.25]{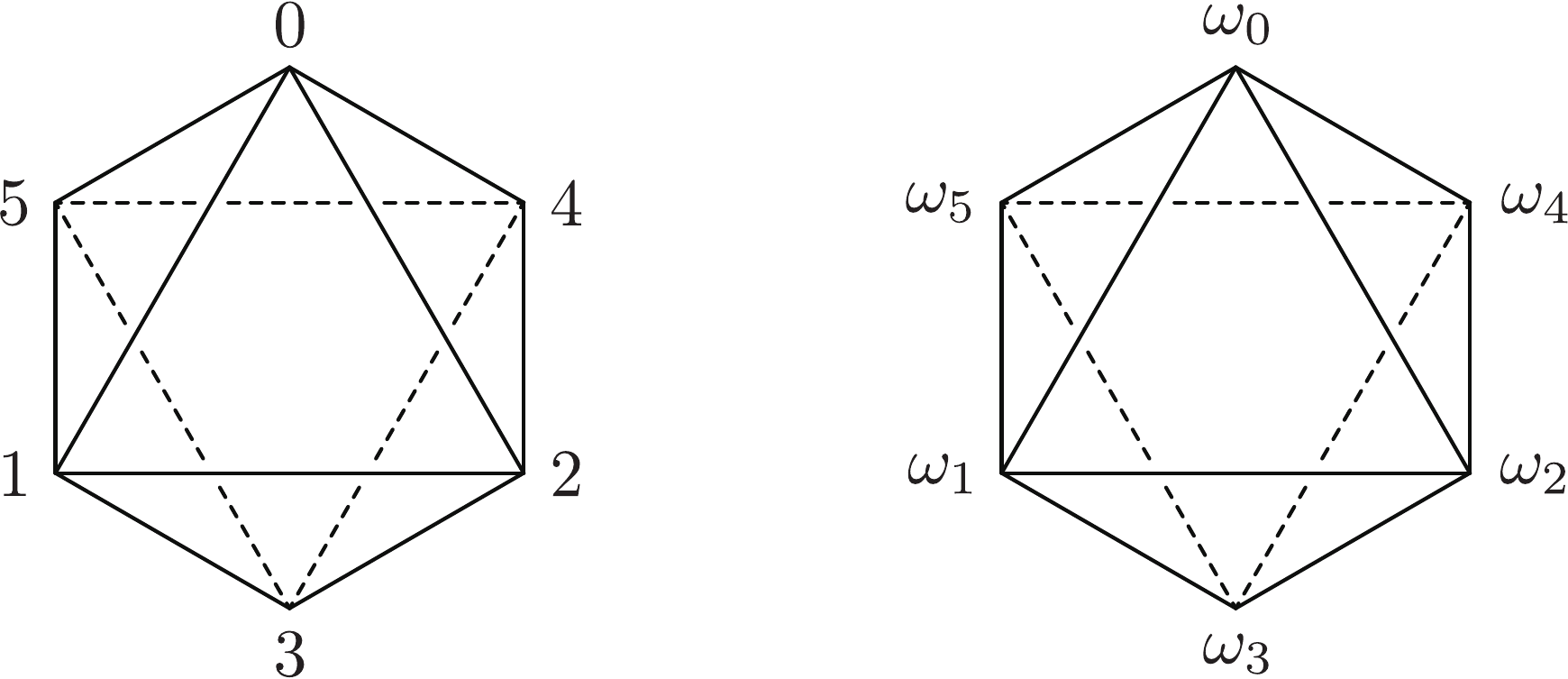}
 \caption{A regular octahedron (left) and its image by rotations (right).}
 \label{fig:octahedral_quandle}
\end{figure}

\section{Tools}
\label{sec:tools}

In this section, we introduce some notions and lemmas all of which are essentially given by Satoh in \cite{Sat2016} to discuss the length of a 3-cocycle of a quandle.
Although Satoh only focused on a dihedral quandle in \cite{Sat2016}, we do not only on a dihedral quandle but on the octahedral quandle.
Thus, we suitably develop arguments in \cite{Sat2016} to fit for any quandles.
We will summarize the modifications at the end of this section as Remarks \ref{rem:type0} and \ref{rem:reverse}.

Let $X$ be a quandle, $S$ an $X$-set, and $\gamma = \pm (s; a, b, c)$ a 3-term of $(X, S)$.
We define the \emph{type} of $\gamma$ as follows.
\begin{itemize}
\item
$\gamma$ is of type 0 if $a = c$ and $a^{b} = c$.
\item
$\gamma$ is of type 1 if $a = c$ and $a^{b} \neq c$.
\item
$\gamma$ is of type 2 if $a \neq c$ and $a^{b} = c$.
\item
$\gamma$ is of type 3 if $a \neq c$ and $a^{b} \neq c$.
\end{itemize}
We note that, in formulae
\begin{align*}
 f(\gamma) & = \mp \> (s; b, c) \> \uline{\pm \> (s; a, c)} \mp (s; a, b), \\
 g(\gamma) & = \pm \> (s^{a}; b, c) \> \uuline{\mp \> (s^{b}; a^{b}, c)} \pm (s^{c}, a^{c}, b^{c}),
\end{align*}
the single- (or double-) underlined 2-term vanishes if and only if $\gamma$ is of type 0 or 1 (or of type 0 or 2).
Further, the other 2-terms never vanish, because we have $b \neq c$, $a \neq b$ and $a^{c} \neq b^{c}$.

In the remaining of this section, we suppose that $S = \mathbb{Z} \times X$ on which $G(X)$ acts from the right as $(n, u)^{a} = (n + 1, u^{a})$ for each generator $a$ of $G(X)$.
We abbreviate an $m$-term $\gamma = \pm ((n, u); a_{1}, a_{2}, \dots, a_{m})$ of $(X, S)$ as $\pm (n, u; a_{1}, a_{2}, \dots, a_{m})$.
We call $\pm$, $n$, $u$ and $(a_{1}, a_{2}, \dots, a_{m})$ the \emph{sign}, \emph{degree}, \emph{index} and \emph{color} of $\gamma$, respectively.
Throughout this paper, we often abbreviate degrees or indices of $m$-terms if they are clear.
Furthermore, we let $\pm \langle n, u; a, b \rangle$ denote the 3-term $\pm (n, u; a, b, a)$ or $\pm (n, u; b, a, b)$.
We note that $g$ increases degrees by one, while $f$ preserves.

Let $\gamma$ be a 3-chain of $(X, S)$ and $\sum_{i = 1}^{l(\gamma)} \gamma_{i}$ a reduced form of $\gamma$.
We respectively call the minimal and maximal degrees of 3-terms $\gamma_{i}$ the \emph{minimal} and \emph{maximal degrees} of $\gamma$.
Furthermore, for each integer $k$, we let $T_{k}(\gamma)$ denote the set consisting of 3-terms $\gamma_{i}$ of degree $k$.
Then, we have the following lemma.

\begin{lemma}[cf.\ Lemma 3.1 of \cite{Sat2016}]
\label{lem:3-cycle_condition}
Let $\gamma$ be a 3-chain of $(X, S)$ and $\sum_{i = 1}^{l(\gamma)} \gamma_{i}$ a reduced form of $\gamma$.
Then, we have the following claim.
\begin{itemize}
\item[(i)]
$\gamma$ is a 3-cycle if and only if $\sum_{\gamma_{i} \in T_{k-1}(\gamma)} g(\gamma_{i}) + \sum_{\gamma_{i} \in T_{k}(\gamma)} f(\gamma_{i}) = 0$ for each $k \in \mathbb{Z}$.
\end{itemize}
Furthermore, if $\gamma$ is a 3-cycle, we have the following claims.
\begin{itemize}
\item[(ii)]
$\sum_{\gamma_{i} \in T_{k}(\gamma)} f(\gamma_{i}) = 0$ if $k$ is the minimal degree of $\gamma$.
\item[(iii)]
$\sum_{\gamma_{i} \in T_{k}(\gamma)} g(\gamma_{i}) = 0$ if $k$ is the maximal degree of $\gamma$.
\end{itemize}
\end{lemma}

\begin{proof}
We immediately have the claims by definition.
\end{proof}

Define a new binary operation $(a, b) \mapsto a^{\overline{b}}$ on $X$ by $a^{\overline{b}} = (\square^{b})^{-1}(a)$.
Then, it is routine to see that this new binary operation also satisfies the axioms of a quandle.
We call the quandle $X$ equipped with this new binary operation the \emph{dual} of the original quandle $X$.
We let $\overline{X}$ denote the dual of $X$.
Obviously, we have $\overline{\overline{X}} = X$ by definition.

For each $m$-term $\gamma = \pm (n, u; a_{1}, a_{2}, \dots, a_{m})$ of $(X, S)$, we define the \emph{reverse} of $\gamma$, which is an $m$-term of $\left(\overline{X}, \mathbb{Z} \times \overline{X} \right)$, by
\[
 \overline{\gamma} = \pm (- n, u^{a_{1} a_{2} \dots a_{m}}; a_{1}^{a_{2} a_{3} \dots a_{m}}, a_{2}^{a_{3} a_{4} \dots a_{m}}, \dots, a_{m-1}^{a_{m}}, a_{m}).
\]
We further extend it to the reverse of an $m$-chain of $(X, S)$ in a natural way.
Let $\sigma$ be an automorphism of $C^{Q}_{m}(X)_{S}$ given by
\[
 \sigma (n, u; a_{1}, a_{2}, \dots, a_{m}) = (n + 1, u; a_{1}, a_{2}, \dots, a_{m}).
\]
Then, we have the following lemma.

\begin{lemma}[cf.\ Lemma 4.1 of \cite{Sat2016}]
\label{lem:reverse}
For each $m$-chain $\gamma$ of $(X, S)$, we have the following claims.
\begin{itemize}
\item[(i)]
$\overline{\overline{\gamma}} = \gamma$.
\item[(ii)]
Assume that $\gamma$ is a 3-term.
Then, $\gamma$ is of type 0, 1, 2 or 3 if and only if $\overline{\gamma}$ is of type 0, 2, 1 or 3, respectively.
\item[(iii)]
$f \left( \overline{\gamma} \right) = - \> \overline{g(\sigma(\gamma))}$,
$g \left( \overline{\gamma} \right) = - \> \overline{f(\sigma(\gamma))}$, and thus
$\partial \left( \overline{\gamma} \right) = - \> \overline{\partial(\sigma(\gamma))}$.
\item[(iv)]
$\gamma$ is an $m$-cycle or boundary if and only if $\overline{\gamma}$ is an $m$-cycle or boundary, respectively.
\end{itemize}
\end{lemma}

\begin{proof}
The proof is straightforward.
\end{proof}

Let $\gamma_{1}, \gamma_{2}, \dots, \gamma_{k}$ be efficient 3-terms of $(X, S)$ having the same degree.
They are said to be \emph{$f$-connected} (or \emph{$g$-connected}) if $\sum_{i = 1}^{k} f(\gamma_{i}) = 0$ (or $\sum_{i = 1}^{k} g(\gamma_{i}) = 0$) and there are no non-empty proper subsets $I$ of $\{ 1, 2, \dots, k \}$ satisfying $\sum_{i \in I} f(\gamma_{i}) = 0$ (or $\sum_{i \in I} g(\gamma_{i}) = 0$).
We also refer to a set $\{ \gamma_{1}, \gamma_{2}, \dots, \gamma_{k} \}$ as being $f$- or $g$-connected if $\gamma_{1}, \gamma_{2}, \dots, \gamma_{k}$ are $f$- or $g$-connected, respectively.
With respect to $f$- and $g$-connectedness, we have the following lemma.

\begin{lemma}[cf.\ Lemma 4.3 of \cite{Sat2016}]
\label{lem:f-/g-connectedness}
For some integer $n$ and positive integer $k$, let $\gamma_{i} = \varepsilon_{i} (n, u_{i}; a_{i}, b_{i}, c_{i})$ be 3-terms of $(X, S)$ {\upshape (}$1 \leq i \leq k${\upshape )}.
Then, we have the following claims.
\begin{itemize}
\item[(i)]
$\gamma_{1}, \gamma_{2}, \dots, \gamma_{k}$ are $g$-connected if and only if $\overline{\gamma_{1}}, \overline{\gamma_{2}}, \dots, \overline{\gamma_{k}}$ are $f$-connected.
\item[(ii)]
If $\gamma_{1}, \gamma_{2}, \dots, \gamma_{k}$ are $f$-connected, then $u_{1} = u_{2} = \dots = u_{k}$.
\item[(iii)]
If $\gamma_{1}, \gamma_{2}, \dots, \gamma_{k}$ are $g$-connected, then $u_{1}^{a_{1} b_{1} c_{1}} = u_{2}^{a_{2} b_{2} c_{2}} = \dots = u_{k}^{a_{k} b_{k} c_{k}}$.
\end{itemize}
\end{lemma}

\begin{proof}
In light of Lemma \ref{lem:reverse} (i) and (iii), we immediately have the claim (i).
We obviously have the claim (ii) by the definition of $f$.
In light of (i) and (ii), we immediately have the claim (iii).
\end{proof}

\begin{remark}
\label{rem:type0}
There are no 3-terms of $(R_{n}, \mathbb{Z} \times R_{n})$ which are of type 0 if $n$ is odd, because $a = c$ and $a^{b} = c$ yield $a = b$.
Therefore, Satoh did not define type 0 in \cite{Sat2016}.
On the other hand, there are 3-terms of $(O_{6}, \mathbb{Z} \times O_{6})$ which are of type 0.
We note that a 3-term $\gamma$ of $(O_{6}, \mathbb{Z} \times O_{6})$ is of type 0 if and only if $\gamma = \pm(n, u; a, [a + 3], a)$ with some $n \in \mathbb{Z}$ and $u, a \in O_{6}$.
\end{remark}

\begin{remark}
\label{rem:reverse}
Since we have $R_{n} = \overline{R_{n}}$ for each $n \geq 3$, Satoh defined the reverses of 2- and 3-terms of $(R_{n}, \mathbb{Z} \times R_{n})$ as those of $(R_{n}, \mathbb{Z} \times R_{n})$ in \cite{Sat2016}.
On the other hand, we have $O_{6} \neq \overline{O_{6}}$.
We thus define the reverses as above in this paper.
We note that, since we have an isomorphism $\tau : O_{6} \to \overline{O_{6}}$ given by
\begin{align*}
 \tau(0) & = 0, &
 \tau(1) & = 1, &
 \tau(2) & = 5, &
 \tau(3) & = 3, &
 \tau(4) & = 4, &
 \tau(5) & = 2,
\end{align*}
we may regard the reverse of each $m$-term of $(O_{6}, \mathbb{Z} \times O_{6})$ as also an $m$-term of $(O_{6}, \mathbb{Z} \times O_{6})$ mapping it by $\tau^{-1}$.
\end{remark}

\section{Enumeration of $f$-connected 3-terms}
\label{sec:enumeration_of_f-connected_3-terms}

The aim of this section is to enumerate all $f$-connected 3-terms $\gamma_{1}, \gamma_{2}, \dots, \gamma_{k}$ for $1 \leq k \leq 5$.
We note that the enumeration does not depend on the choice of a quandle, because we can compute each $f(\gamma_{i})$ without using the binary operation of the quandle.

Let $X$ be a quandle and $S = \mathbb{Z} \times X$ our $X$-set.
Choose and fix $n \in \mathbb{Z}$ and $u \in X$.
Associated with a 3-term $\gamma$ of $(X, S)$ whose degree is $n$ and index $u$, we consider an oriented bigon or triangle whose vertices and edges are respectively labeled by elements of $X$ and oriented as depicted in Figure \ref{fig:pieces}.
We can obtain an ordered pair of elements of $X$ from an edge of the polygon by writing down the labels of the initial and end points of the edge in this order.
We call this ordered pair the \emph{label} of the edge.
Labels of the edges are also illustrated in Figure \ref{fig:pieces}.
We note that each label of the edges of the polygon corresponds to each color of the 2-terms of $f(\gamma)$, and each orientation of the edges coincides with the orientation of the polygon if and only if the sign of the corresponding 2-term is positive.
Furthermore, we remark that we have the same bigon associated with 3-terms $\pm(n, u; a, b, a)$ and $\pm(n, u; b, a, b)$ ($a \neq b$).
We allow two polygons $P$ and $P^{\prime}$ associated with 3-terms to be glued along an edge $e$ of $P$ and $e^{\prime}$ of $P^{\prime}$ if the 2-terms corresponding to $e$ and $e^{\prime}$ can be cancelled out, that is, if
\begin{itemize}
\item
$e$ and $e^{\prime}$ have the same label,
\item
the orientation of $e$ coincides with the orientation of $P$, and
\item
the orientation of $e^{\prime}$ does not coincide with the orientation of $P^{\prime}$.
\end{itemize}
Then, by definition, we can obtain a cell decomposition of an oriented close surface (which should not always be connected) from 3-terms $\gamma_{1}, \gamma_{2}, \dots, \gamma_{k}$ of $(X, S)$ having the same degree $n$ and the same index $u$, and satisfying $\sum_{i = 1}^{k} f(\gamma_{i}) = 0$.
\begin{figure}[htbp]
 \centering
 \includegraphics[scale=0.25]{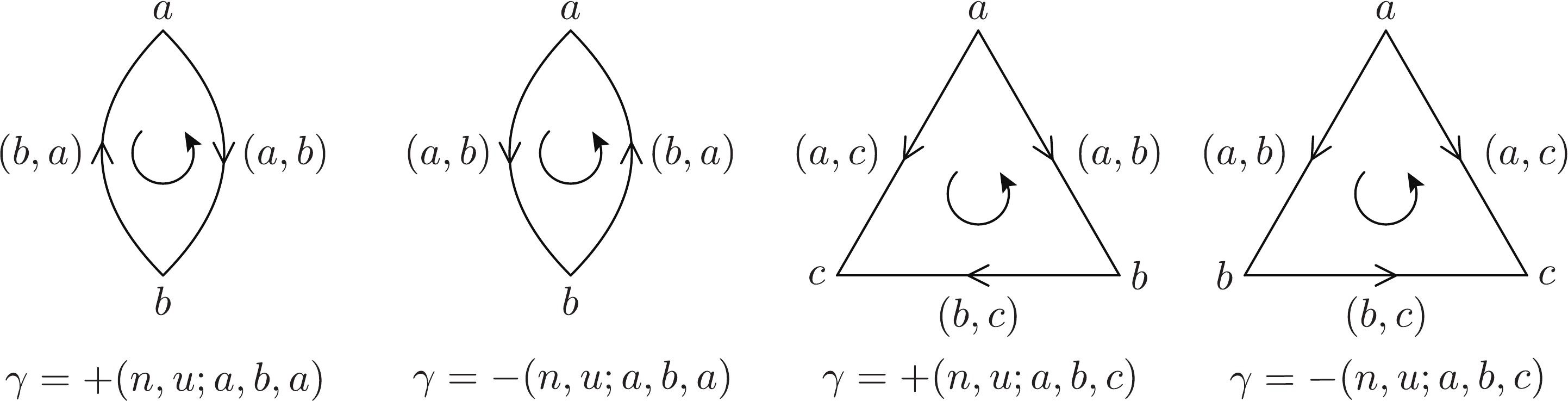}
 \caption{Bigons and triangles associated with 3-terms $\gamma$ of $(X, S)$ whose degree is $n$ and index $u$, where $a$, $b$ and $c$ are mutually distinct elements of $X$.}
 \label{fig:pieces}
\end{figure}

Conversely, consider a cell decomposition of an oriented closed surface which consists of bigons and triangles.
Assume that the 0- and 1-cells of the complex are respectively labeled by elements of $X$ and oriented so that each 2-cell looks like one of bigons or triangles depicted in Figure \ref{fig:pieces}.
Then, we have 3-terms $\gamma_{i} = \varepsilon_{i} (n, u; a_{i}, b_{i}, c_{i})$ ($1 \leq i \leq k$) of $(X, S)$ satisfying $\sum_{i = 1}^{k} f(\gamma_{i}) = 0$ so that the cell decomposition is obtained from $\gamma_{1}, \gamma_{2}, \dots, \gamma_{k}$.
Here, $k$ denotes the number of 2-cells.
It means that we can obtain all $f$-connected 3-terms by enumerating the cell decompositions of ``connected'' orientable closed surfaces consisting of bigons and triangles.
Therefore, for $1 \leq k \leq 5$, let us enumerate such cell decompositions concretely up to homeomorphism.

We start with enumerating the cell decompositions of connected orientable closed surfaces which consist of only triangles.
Since the number of 1-cells of such a cell decomposition is equal to $\frac{3 k}{2}$ and an integer, it is sufficient to consider the case $k = 2$ or $4$.
Furthermore, since we would like to label the vertices of each triangle by mutually distinct elements of $X$, we can exclude the following case from our enumeration: the number of 0-cells is less than three, or some triangle has vertices sharing the same 0-cell.
Then, since the number of 0-cells is at least three, the Euler characteristics of acceptable cell decompositions are at least one.
Under the constraints, we exactly have the cell decompositions (2b), (4b), (4c) and (4d) of the 2-sphere $S^{2}$ depicted in Figure \ref{fig:the_cell_decompositions}.
\begin{figure}[htbp]
 \centering
 \includegraphics[scale=0.25]{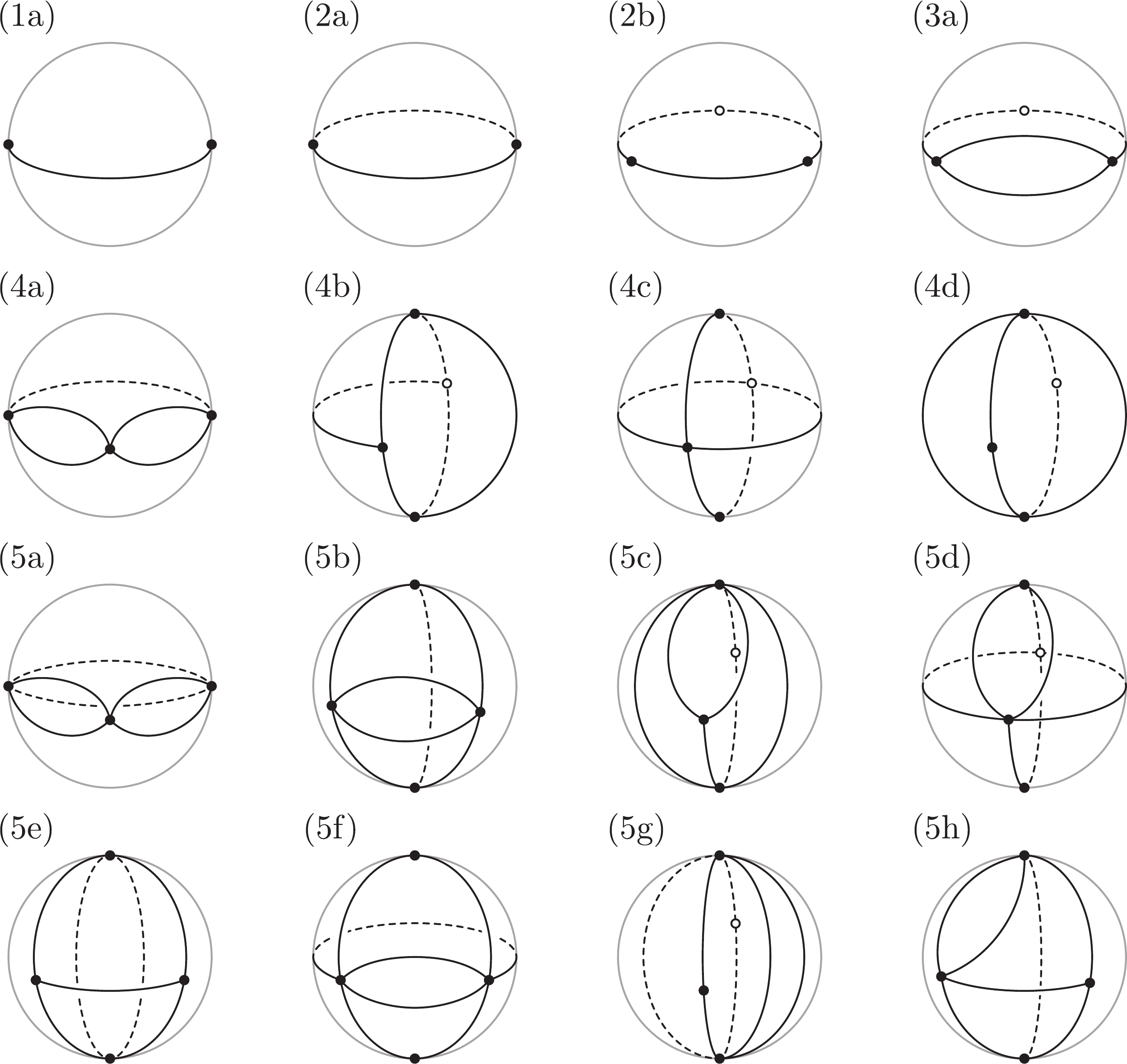}
 \caption{The cell decompositions of $S^{2}$ consisting of at most five bigons and triangles, which might yield $f$-connected 3-terms.}
 \label{fig:the_cell_decompositions}
\end{figure}

We note that each cell decomposition of a surface having at least one bigon as its 2-cell is recursively obtained from a cell decomposition of the surface by replacing some 1-cell with a bigon, except (1a) depicted in Figure \ref{fig:the_cell_decompositions}, if the vertices of each bigon do not share the same 0-cell as we desired.
It is easy to see that two adjacent bigons in a cell decomposition are obtained from $f$-connected 3-terms $+ (n, u; a, b, a)$ and $- (n, u; b, a, b)$, or inefficient 3-terms $+ (n, u; a, b, a)$ and $- (n, u; a, b, a)$ with some distinct $a, b \in X$.
Therefore, we can also exclude the following case from our enumeration except (2a) depicted in Figure \ref{fig:the_cell_decompositions}: there are adjacent bigons.
As a consequence, we exactly have the cell decompositions of $S^{2}$ depicted in Figure \ref{fig:the_cell_decompositions}.
In the figure, the number in the name assigned to each cell decomposition represents $k$.
In light of this enumeration, we have the following lemmas.

\begin{lemma}[cf.\ Lemma 5.1 of \cite{Sat2016}]
\label{lem:f-connected1}
For each 3-term $\gamma$ of $(X, S)$, we have $f(\gamma) \neq 0$.
\end{lemma}

\begin{proof}
Since there are no ways to orient the 1-cell of (1a) so that the 2-cell looks like one of bigons depicted in Figure \ref{fig:pieces}, we have the claim.
\end{proof}

\begin{lemma}[cf.\ Lemma 5.2 of \cite{Sat2016}]
\label{lem:f-connected2}
Let $\gamma_{1}$ and $\gamma_{2}$ be $f$-connected 3-terms of $(X, S)$ having the same degree $n$ and the same index $u$, and $\gamma = \gamma_{1} + \gamma_{2}$.
Then, we have
\[
 \gamma = + \> (a, b, a) - (b, a, b)
\]
with some distinct $a, b \in X$.
\end{lemma}

\begin{proof}
There is essentially one way to orient the 1-cells of (2a) or (2b) so that each 2-cell looks like one of bigons or triangles depicted in Figure \ref{fig:pieces}.
Since (2a) is obtained from $f$-connected 3-terms $+ (a, b, a)$ and $- (b, a, b)$ or inefficient 3-terms $+ (a, b, a)$ and $- (a, b, a)$ with some distinct $a, b \in X$, and (2b) from inefficient 3-terms $+ (a, b, c)$ and $- (a, b, c)$ with some mutually distinct $a, b, c \in X$, we have the claim.
\end{proof}

\begin{lemma}[cf.\ Lemma 5.4 of \cite{Sat2016}]
\label{lem:f-connected3}
Let $\gamma_{1}$, $\gamma_{2}$ and $\gamma_{3}$ be $f$-connected 3-terms of $(X, S)$ having the same degree $n$ and the same index $u$, and $\gamma = \gamma_{1} + \gamma_{2} + \gamma_{3}$.
Then, we have the following two cases up to sign.
\begin{itemize}
\item[(i)]
$\gamma = + \> (a, b, a) - (c, a, b) - (c, b, a)$, where $a$, $b$ and $c$ are mutually different.
\item[(ii)]
$\gamma = + \> (a, b, a) - (a, b, c) - (b, a, c)$, where $a$, $b$ and $c$ are mutually different.
\end{itemize}
\end{lemma}

\begin{proof}
There are essentially two ways, depicted in Figure \ref{fig:f-connected3}, to orient the 1-cells of (3a) so that each 2-cell looks like one of bigons or triangles depicted in Figure \ref{fig:pieces}.
We thus have the claim.
\begin{figure}[htbp]
 \centering
 \includegraphics[scale=0.25]{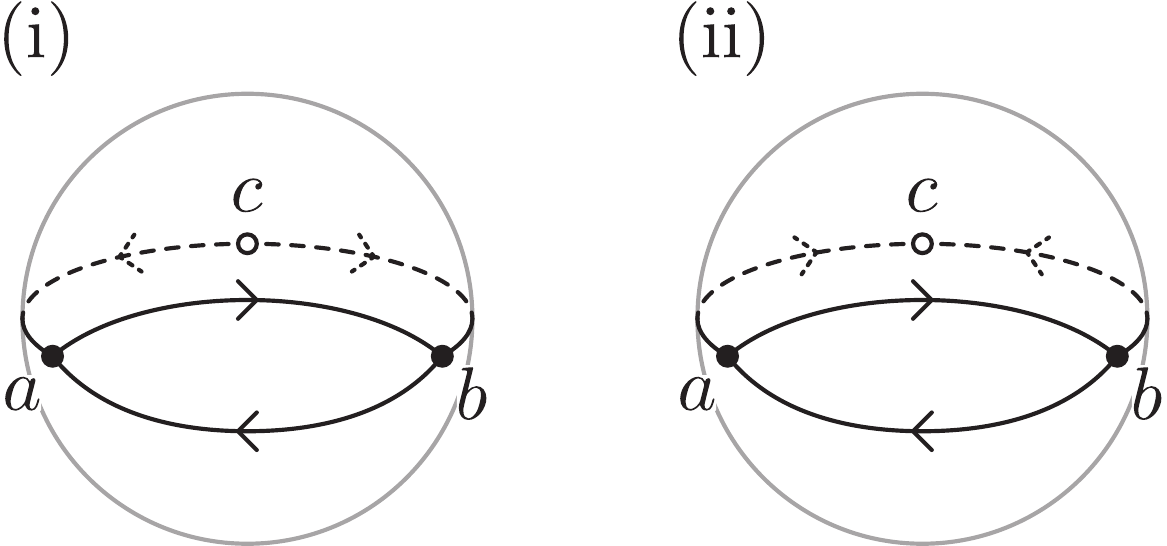}
 \caption{The appropriate ways to orient the 1-cells of (3a).}
 \label{fig:f-connected3}
\end{figure}
\end{proof}

\begin{lemma}
\label{lem:f-connected4}
Let $\gamma_{1}$, $\gamma_{2}$, $\gamma_{3}$ and $\gamma_{4}$ be $f$-connected 3-terms of $(X, S)$ having the same degree $n$ and the same index $u$, and $\gamma = \gamma_{1} + \gamma_{2} + \gamma_{3} + \gamma_{4}$.
Then, we have the following five cases up to sign.
\begin{itemize}
\item[(i)]
$\gamma = + \> (a, b, c) + \langle a, c \rangle - (c, a, b) - \langle b, c \rangle$, \\
where $a$, $b$ and $c$ are mutually different.
\item[(ii)]
$\gamma = + \> (a, b, c) + (a, c, d) - (a, b, d) - (b, c, d)$, \\
where $a$, $b$, $c$ and $d$ are mutually different.
\item[(iii)]
$\gamma = + \> (a, b, c) + (b, a, c) - (a, b, d) - (b, a, d)$, \\
where $a$, $b$, $c$ and $d$ are mutually different.
\item[(iv)]
$\gamma = + \> (c, a, b) + (c, b, a) - (d, a, b) - (d, b, a)$, \\
where $a$, $b$, $c$ and $d$ are mutually different.
\item[(v)]
$\gamma = + \> (c, a, b) + (c, b, a) - (a, b, d) - (b, a, d)$, \\
where $a$, $b$, $c$ are mutually different and $d \neq a, b$.
\end{itemize}
\end{lemma}

\begin{proof}
It is routine to check that there are no ways to orient the 1-cells of (4d) appropriately (i.e., so that each 2-cell looks like one of bigons or triangles depicted in Figure \ref{fig:pieces} and it is obtained from efficient 3-terms).
Furthermore, it is also routine to see that there are essentially five ways, depicted in Figure \ref{fig:f-connected4}, to orient the 1-cells of (4a)--(4c) appropriately.
We thus have the claim.
\begin{figure}[htbp]
 \centering
 \includegraphics[scale=0.25]{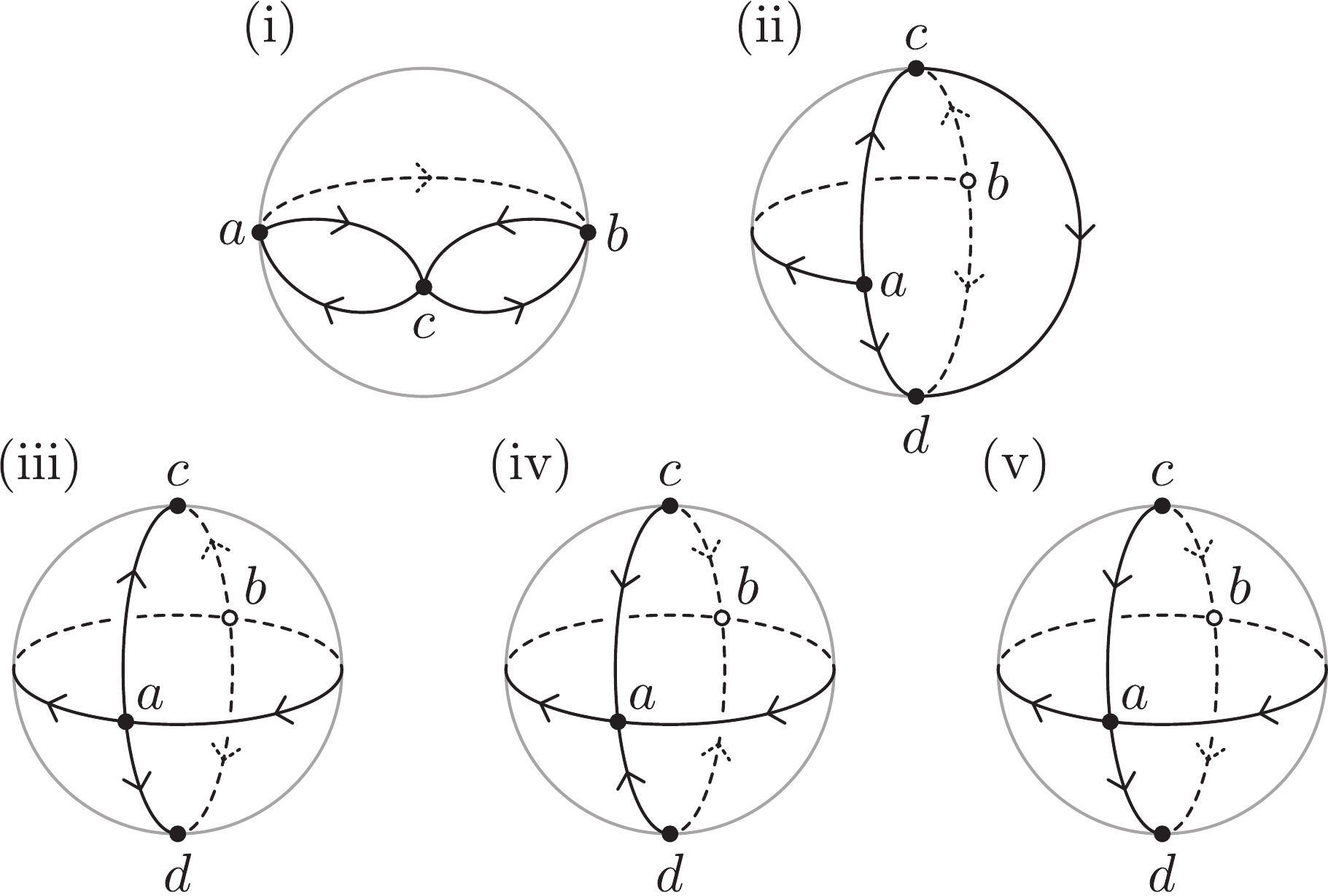}
 \caption{The appropriate ways to orient the 1-cells of (4a)--(4c).}
 \label{fig:f-connected4}
\end{figure}
\end{proof}

\begin{lemma}
\label{lem:f-connected5}
Let $\gamma_{1}$, $\gamma_{2}$, $\gamma_{3}$, $\gamma_{4}$ and $\gamma_{5}$ be $f$-connected 3-terms of $(X, S)$ having the same degree $n$ and the same index $u$, and $\gamma = \gamma_{1} + \gamma_{2} + \gamma_{3} + \gamma_{4} + \gamma_{5}$.
Then, we have the following ten cases up to sign.
\begin{itemize}
\item[(i)]
$\gamma = + \> \langle a, c \rangle + \langle b, c \rangle - (a, b, a) - (a, c, b) - (b, c, a)$, \\
where $a$, $b$ and $c$ are mutually different.
\item[(ii)]
$\gamma = + \> (a, b, c) + \langle b, d \rangle - (a, b, d) - (a, d, c) - (d, b, c)$, \\
where $a$, $b$, $c$ and $d$ are mutually different.
\item[(iii)]
$\gamma = + \> (a, b, c) + \langle a, d \rangle - (a, d, c) - (d, a, b) - (d, b, c)$, \\
where $a$, $b$, $c$ and $d$ are mutually different.
\item[(iv)]
$\gamma = + \> (a, b, c) + \langle c, d \rangle - (a, d, c) - (a, b, d) - (b, c, d)$, \\
where $a$, $b$, $c$ and $d$ are mutually different.
\item[(v)]
$\gamma = + \> (a, b, c) + \langle a, c \rangle - (d, a, b) - (d, b, c) - (d, c, a)$, \\
where $a$, $b$, $c$ and $d$ are mutually different.
\item[(vi)]
$\gamma = + \> (a, b, c) + \langle a, c \rangle - (a, b, d) - (b, c, d) - (c, a, d)$, \\
where $a$, $b$, $c$ and $d$ are mutually different.
\item[(vii)]
$\gamma = + \> (a, b, c) + \langle a, c \rangle - (c, a, b) - (d, b, c) - (d, c, b)$, \\
where $a$, $b$, $c$ and $d$ are mutually different.
\item[(viii)]
$\gamma = + \> (a, b, c) + \langle a, c \rangle - (a, b, d) - (b, a, d) - (b, c, a)$, \\
where $a$, $b$, $c$ and $d$ are mutually different.
\item[(ix)]
$\gamma = + \> (a, b, c) + \langle a, c \rangle - (b, c, a) - (d, a, b) - (d, b, a)$, \\
where $a$, $b$, $c$ are mutually different and $d \neq a, b$.
\item[(x)]
$\gamma = + \> (c, b, a) + \langle a, c \rangle - (a, b, d) - (a, c, b) - (b, a, d)$, \\
where $a$, $b$, $c$ are mutually different and $d \neq a, b$.
\end{itemize}
\end{lemma}

\begin{proof}
It is routine to check that there are no ways to orient 1-cells of (5d)--(5h) appropriately, and essentially ten ways, depicted in Figure \ref{fig:f-connected5}, to orient the 1-cells of (5a)--(5c) appropriately.
We thus have the claim.
\begin{figure}[htbp]
 \centering
 \includegraphics[scale=0.25]{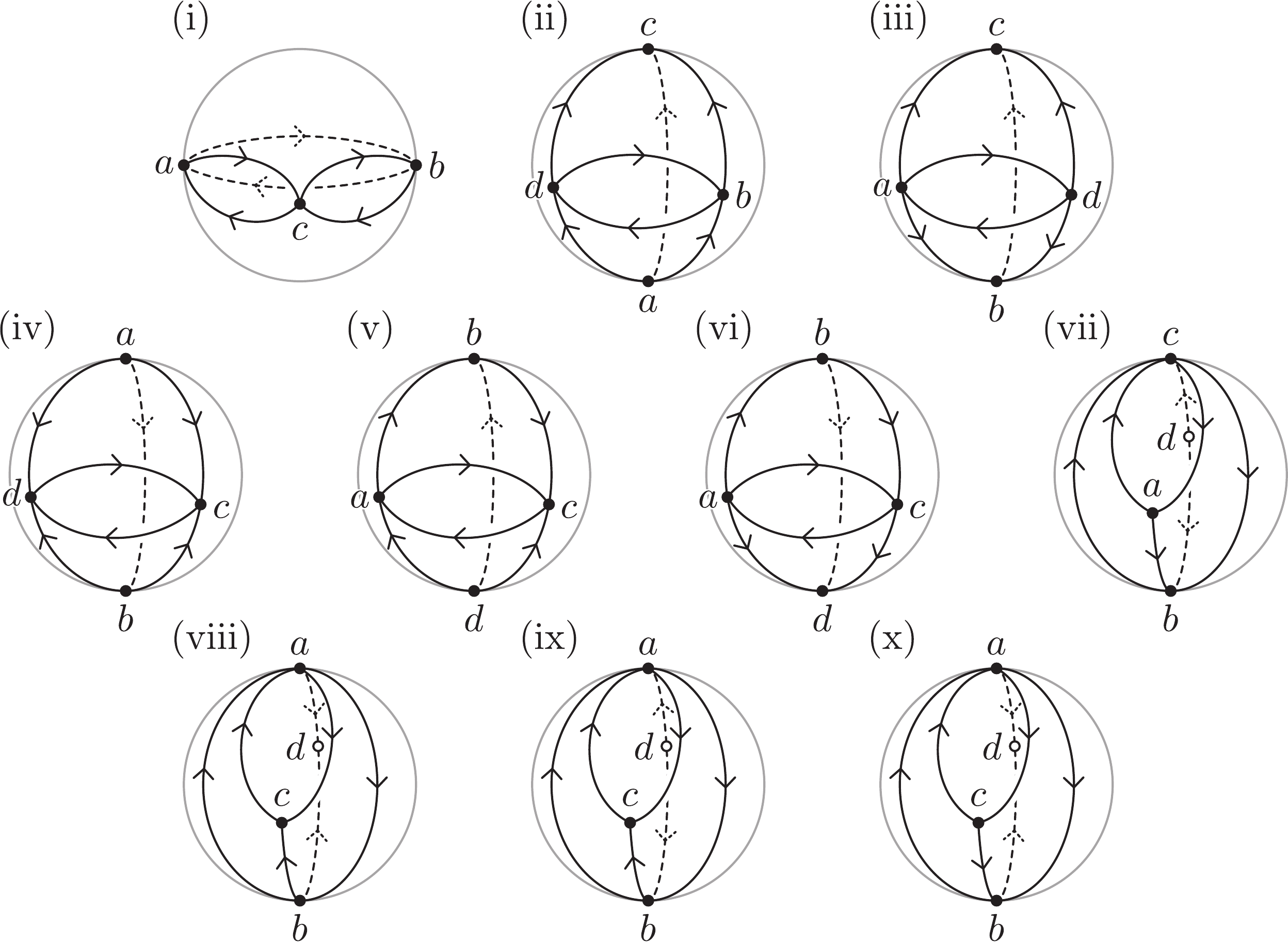}
 \caption{The appropriate ways to orient the 1-cells of (5a)--(5c).}
 \label{fig:f-connected5}
\end{figure}
\end{proof}

\section{A lower bound of the length of the Mochizuki 3-cocycle $\zeta$ of $R_{7}$}
\label{sec:lower_bound_of_zeta}

The aim of this section is to show the following key theorem in a similar way to Satoh's work \cite{Sat2016}.

\begin{theorem}
\label{thm:lower_bound_R7}
For the Mochizuki 3-cocycle $\zeta$ of the 7-dihedral quandle $R_{7}$, we have $l(\zeta, \mathbb{Z} \times R_{7}) \geq 8$.
\end{theorem}

We start with discussing the following things in the style of \cite{Sat2016}.
Let $a_{0}$ and $a_{1}$ be distinct elements of $R_{7}$ and $a_{i} = a_{0} + i (a_{1} - a_{0})$ for each $i$ ($2 \leq i \leq 6$).
Then, it is easy to see that the set consisting of $a_{0}, a_{1}, \dots, a_{6}$ coincides with $R_{7}$ and we have $a_{i}^{a_{j}} = a_{2 j - i}$ ($= a_{i^{j}}$).

Let $S$ denote our $R_{7}$-set $\mathbb{Z} \times R_{7}$.
For each $m$-term $\gamma = \pm (n, w; x_{1}, x_{2}, \dots, x_{m})$ of $(R_{7}, S)$, we define the \emph{reflection} of $\gamma$, which is an $m$-term of $(R_{7}, S)$, by
\[
 \gamma^{\ast} = \pm (n, (-1)^{n+1} w; (-1)^{n} (x_{m} - w), (-1)^{n} (x_{m-1} - w), \dots, (-1)^{n} (x_{1} - w)).
\]
We further extend it to the reflection of an $m$-chain of $(X, S)$ in a natural way.
Then, we have the following lemma.

\begin{lemma}[cf.\ Lemmas 4.2 and 4.3 of \cite{Sat2016}]
\label{lem:reflection}
For each $m$-chain $\gamma$ of $(R_{7}, S)$, we have the following claims.
\begin{itemize}
\item[(i)]
$\gamma^{\ast \ast} = \gamma$.
\item[(ii)]
Assume that $\gamma$ is a 3-term.
Then, $\gamma$ and $\gamma^{\ast}$ are the same type.
\item[(iii)]
$f(\gamma^{\ast}) = f(\gamma)^{\ast}$,
$g(\gamma^{\ast}) = g(\gamma)^{\ast}$, and thus
$\partial (\gamma^{\ast}) = \partial(\gamma)^{\ast}$.
\item[(iv)]
$\gamma$ is an $m$-cycle or boundary if and only if $\gamma^{\ast}$ is an $m$-cycle or boundary, respectively.
\end{itemize}
Furthermore, let $\gamma_{1}, \gamma_{2}, \dots, \gamma_{k}$ be 3-terms of $(R_{7}, S)$.
Then, we have the following claim.
\begin{itemize}
\item[(v)]
$\gamma_{1}, \gamma_{2}, \dots, \gamma_{k}$ are $f$-connected if and only if $\gamma_{1}^{\ast}, \gamma_{2}^{\ast}, \dots, \gamma_{k}^{\ast}$ are $f$-connected.
\end{itemize}
\end{lemma}

\begin{proof}
The proof is straightforward.
\end{proof}

To prove Theorem \ref{thm:lower_bound_R7}, we show the following Lemmas and Propositions.
We note that the reflection of $\gamma$ in Lemma \ref{lem:f-connected3} (ii) is nothing less than that in Lemma \ref{lem:f-connected3} (i) under a suitable transformation of variables (see the proof of Lemma 5.4 of \cite{Sat2016} for more detail).
In the remaining of this section, we select $w$ instead of $u$ in Lemmas \ref{lem:f-connected2} and \ref{lem:f-connected3} according to the notation in \cite{Sat2016}.

\begin{lemma}[cf.\ Lemma 5.3 of \cite{Sat2016}]
\label{lem:R7_f2}
Let $\gamma$ be the 3-chain in Lemma \ref{lem:f-connected2} with $X = R_{7}$.
Then, we have $g(\gamma) \neq 0$.
Furthermore, let $\gamma_{3}, \gamma_{4}, \dots, \gamma_{k + 2}$ be 3-terms of $(R_{7}, S)$ whose degrees are $n + 1$.
If $g(\gamma) + f(\sum_{i = 3}^{k + 2} \gamma_{i}) = 0$, then we have $k \geq 4$.
\end{lemma}

\begin{proof}
The argument in the proof of Lemma 5.3 of \cite{Sat2016} works well even if we select $R_{7}$ instead of $R_{5}$, because we have $w^{a} \neq w^{b}$ for each $w \in R_{7}$ and distinct $a, b \in R_{7}$.
We thus have the claim.
\end{proof}

\begin{lemma}[cf.\ Lemma 5.5 of \cite{Sat2016}]
\label{lem:R7_f3}
Let $\gamma$ be one of the 3-chains in Lemma \ref{lem:f-connected3} with $X = R_{7}$.
Then, we have $g(\gamma) \neq 0$.
Furthermore, let $\gamma_{4}, \gamma_{5}, \dots, \gamma_{k + 3}$ be 3-terms of $(R_{7}, S)$ whose degrees are $n + 1$.
If $g(\gamma) + f(\sum_{i = 4}^{k + 3} \gamma_{i}) = 0$, then we have $k \geq 3$.
\end{lemma}

\begin{proof}
The argument in the proof of Lemma 5.5 of \cite{Sat2016} works well even if we select $R_{7}$ instead of $R_{5}$, because $w^{a}$, $w^{b}$ and $w^{c}$ are mutually different for each $w \in R_{7}$ and mutually distinct $a, b, c \in R_{7}$.
We thus have the claim.
\end{proof}

\begin{proposition}[cf.\ Proposition 5.6 of \cite{Sat2016}]
\label{prop:R7_list}
Let $\gamma$ be a 3-cycle of $(R_{7}, S)$ whose length is $l$ {\upshape (}$1 \leq l \leq 7${\upshape )}, and $n$ the minimal degree of $\gamma$.
Then, we have the following cases up to reverse.
\begin{itemize}
\item[(A)]
$4 \leq l \leq 7$ and $|T_{n}(\gamma)| = l$.
\item[(B)]
$6 \leq l \leq 7$, $|T_{n}(\gamma)| = 2$ and $|T_{n + 1}(\gamma)| = l - 2$.
\item[(C)]
$6 \leq l \leq 7$, $|T_{n}(\gamma)| = 3$ and $|T_{n + 1}(\gamma)| = l - 3$.
\end{itemize}
\end{proposition}

\begin{proof}
In light of Lemmas \ref{lem:3-cycle_condition}--\ref{lem:f-/g-connectedness}, \ref{lem:f-connected1}, \ref{lem:R7_f2} and \ref{lem:R7_f3}, we have the claim in a similar way to the proof of Lemma 5.6 of \cite{Sat2016}.
\end{proof}

\begin{lemma}[cf.\ Lemma 6.1 of \cite{Sat2016}]
\label{lem:R7_f_and_g}
Let $\gamma_{1}, \gamma_{2}, \dots, \gamma_{l}$ be $f$- and $g$-connected 3-terms of $(R_{7}, S)$ having the same degree $n$.
Then, we have $l \geq 16$.
\end{lemma}

\begin{proof}
In light of Lemma \ref{lem:f-/g-connectedness} (ii), $\gamma_{1}, \gamma_{2}, \dots, \gamma_{l}$ have the same index $w$.
Let $\gamma = \sum_{i = 1}^{l} \gamma_{i}$.
Then, in a similar way to the proof of Lemma 6.1 of \cite{Sat2016}, we have
\begin{align*}
 \gamma
 & = \alpha ( (a_{1}, a_{3}, a_{2}) + (a_{1}, a_{6}, a_{5}) - (a_{2}, a_{1}, a_{6}) - (a_{2}, a_{3}, a_{1}) \\
 & \qquad \qquad \qquad - (a_{5}, a_{4}, a_{6}) - (a_{5}, a_{6}, a_{1}) + (a_{6}, a_{1}, a_{2}) + (a_{6}, a_{4}, a_{5}) ) \\
 & \phantom{=} + \beta ( (a_{2}, a_{5}, a_{3}) + (a_{2}, a_{6}, a_{4}) - (a_{3}, a_{1}, a_{5}) - (a_{3}, a_{5}, a_{2}) \\
 & \qquad \qquad \qquad - (a_{4}, a_{2}, a_{5}) - (a_{4}, a_{6}, a_{2}) + (a_{5}, a_{1}, a_{3}) + (a_{5}, a_{2}, a_{4}) ) \\
 &  \phantom{=} + (\alpha + \beta) ( (a_{1}, a_{4}, a_{3}) + (a_{1}, a_{5}, a_{4}) - (a_{3}, a_{2}, a_{6}) - (a_{3}, a_{4}, a_{1}) \\
 & \qquad \qquad \qquad - (a_{4}, a_{3}, a_{6}) - (a_{4}, a_{5}, a_{1}) + (a_{6}, a_{2}, a_{3}) + (a_{6}, a_{3}, a_{4}) )
\end{align*}
for some distinct $a_{0}, a_{1} \in R_{7}$ and $\alpha, \beta \in \mathbb{Z}$.
Since no two of $\alpha$, $\beta$ and $\alpha + \beta$ become zero simultaneously, we have the claim.
\end{proof}

\begin{proposition}[cf.\ Proposition 6.2 of \cite{Sat2016}]
\label{prop:R7_I}
There are no 3-cycles $\gamma$ of $(R_{7}, S)$ satisfying $l(\gamma) = |T_{n}(\gamma)| = 4$.
\end{proposition}

\begin{proof}
The argument in the proof of Proposition 6.2 of \cite{Sat2016} works well even if we select $R_{7}$ instead of $R_{5}$, because we have the above lemmas and $b^{a} \neq a^{b}$ for any distinct $a, b \in R_{7}$.
We thus have the claim.
\end{proof}

\begin{proposition}[cf.\ Proposition 6.3 of \cite{Sat2016}]
\label{prop:R7_II}
There are no 3-cycles $\gamma$ of $(R_{7}, S)$ satisfying $l(\gamma) = |T_{n}(\gamma)| = 5$.
\end{proposition}

\begin{proof}
The argument in the proof of Proposition 6.3 of \cite{Sat2016} works well even if we select $R_{7}$ instead of $R_{5}$, because we have the above lemmas.
Therefore, we have the claim.
\end{proof}

\begin{proposition}[cf.\ Proposition 6.4 of \cite{Sat2016}]
\label{prop:R7_III}
There are no 3-cycles $\gamma$ of $(R_{7}, S)$ satisfying $l(\gamma) = |T_{n}(\gamma)| = 6$.
\end{proposition}

\begin{proof}
The argument in the proof of Proposition 6.4 of \cite{Sat2016} works well even if we select $R_{7}$ instead of $R_{5}$ except Case 2, because we have the above lemmas, and $w^{aba} \neq w^{bab}$ and $w^{aba} \neq w^{cba}$ for any $w \in R_{7}$ and mutually distinct $a, b, c \in R_{7}$.
Furthermore, replacing $a_{3}$ and $a_{4}$ with $a_{5}$ and $a_{6}$ respectively, the argument in Case 2 also works well, because the conditions $c^{a} = b$ and $c^{b} = a$ are never compatible for any mutually distinct $a, b, c \in R_{7}$.
We thus have the claim.
\end{proof}

\begin{proposition}[cf.\ Proposition 6.5 of \cite{Sat2016}]
\label{prop:R7_IV}
There are no 3-cycles $\gamma$ of $(R_{7}, S)$ satisfying $l(\gamma) = |T_{n}(\gamma)| = 7$.
\end{proposition}

\begin{proof}
The argument in the proof of Proposition 6.5 of \cite{Sat2016} works well even if we select $R_{7}$ instead of $R_{5}$ except Cases 3 and 4, because we have the above lemmas.
We thus only prove that Cases 3 and 4 do not happen as well as Cases 1 and 2, in this paper.

Let $\gamma = \sum_{i = 1}^{7} \gamma_{i}$ be a 3-cycle of $(R_{7}, S)$ satisfying $l(\gamma) = |T_{n}(\gamma)| = 7$.
Assume that $\gamma$ satisfies the condition (c) introduced in the proof of Proposition 6.5 of \cite{Sat2016}.
Then, we may assume that
\begin{itemize}
\item
$\gamma_{1}, \gamma_{2}$ and $\gamma_{3}$ are $f$-connected, and
\item
$\gamma_{4}, \gamma_{5}, \gamma_{6}$ and $\gamma_{7}$ are $f$-connected.
\end{itemize}
Furthermore, we may assume that $\gamma_{1}$ is of type 1.

If one of $\gamma_{2}$ and $\gamma_{3}$ is of type 2, it is routine to see that the indices of $\overline{\gamma_{1}}, \overline{\gamma_{2}}$ and $\overline{\gamma_{3}}$ are mutually different.
Therefore, in light of Lemma \ref{lem:f-/g-connectedness} (ii), $\overline{\gamma}$ satisfies the condition (a) introduced in the proof of Proposition 6.5 of \cite{Sat2016}.
Since we have already denied such a case as Case 1, we have no candidates of $\gamma$.
We thus assume that both of $\gamma_{2}$ and $\gamma_{3}$ are of type 3.
Then, we may assume that
\begin{align*}
 \gamma_{1} & = + (n, w; a_{0}, a_{1}, a_{0}), &
 \gamma_{2} & = - (n, w; a_{p}, a_{0}, a_{1}), &
 \gamma_{3} & = - (n, w; a_{p}, a_{1}, a_{0})
\end{align*}
up to sign and reflection, with some distinct $a_{0}, a_{1} \in R_{7}$ and $p \in \{ 3, 4, 5 \}$.
Since $w^{a_{p} a_{0} a_{1}} = w^{a_{p + 1}}$, $w^{a_{p} a_{1} a_{0}} = w^{a_{p - 1}}$, and $w^{a_{p + 1}} \neq w^{a_{p - 1}}$ by $a_{p + 1} \neq a_{p - 1}$, indices of $\overline{\gamma_{2}}$ and $\overline{\gamma_{3}}$ are different.
Furthermore, since both of $\overline{\gamma_{2}}$ and $\overline{\gamma_{3}}$ are not of type (0 or) 1, $\overline{\gamma}$ could not satisfy the condition (b) introduced in the proof of Proposition 6.5 of \cite{Sat2016}.

Assume that $\overline{\gamma}$ satisfies the condition (c) introduced in the proof of Proposition 6.5 of \cite{Sat2016}.
Then, since the index of $\overline{\gamma_{1}}$ is $w^{a_{0} a_{1} a_{0}} = w^{a_{6}}$, one of $a_{p - 1}$ and $a_{p + 1}$ is equal to $a_{6}$.
On the other hand, the case $a_{p - 1} = a_{6}$ is not occurred, since it yields $p = 0$.
We thus have $p = 5$.
Since
\begin{align*}
 \overline{\gamma_{1}} & = + (- n, w^{a_{6}}; a_{2}, a_{6}, a_{0}), &
 \overline{\gamma_{2}} & = + (- n, w^{a_{6}}; a_{0}, a_{2}, a_{1}),
\end{align*}
in light of Lemma \ref{lem:f-connected3}, $\overline{\gamma_{1}}, \overline{\gamma_{2}}$ and $\overline{\gamma_{i}}$ are not $f$-connected for any $i$ ($3 \leq i \leq 7$).
Therefore, in light of Lemma \ref{lem:f-/g-connectedness} (i), we may assume that
\begin{itemize}
\item
$\gamma_{3}, \gamma_{4}$ and $\gamma_{5}$ are $g$-connected, and
\item
$\gamma_{1}, \gamma_{2}, \gamma_{6}$ and $\gamma_{7}$ are $g$-connected.
\end{itemize}
Furthermore, we may assume that $\gamma_{5}$ is of type 2.
Since $\overline{\gamma_{3}} = - (- n, w^{a_{4}}; a_{3}, a_{6}, a_{0})$, we have $\overline{\gamma_{4}} = - (- n, w^{a_{4}}; a_{6}, a_{3}, a_{0})$ or $- (- n, w^{a_{4}}; a_{3}, a_{0}, a_{6})$ by Lemma \ref{lem:f-connected3}.

If $\overline{\gamma_{4}} = - (- n, w^{a_{4}}; a_{6}, a_{3}, a_{0})$, we have
\[
 \overline{\gamma_{5}} = + \> (- n, w^{a_{4}}; a_{6}, a_{3}, a_{6}) \ \text{or} \, + (- n, w^{a_{4}}; a_{3}, a_{6}, a_{3}).
\]
On the other hand, the index of $\gamma_{4} = \overline{\overline{\gamma_{4}}}$ is $w + 5 (a_{1} - a_{0})$, even though the index of $\gamma_{5}$ is respectively $w + 3 (a_{1} - a_{0})$ or $w + 6 (a_{1} - a_{0})$.
It contradicts to that $\gamma_{4}, \gamma_{5}, \gamma_{6}$ and $\gamma_{7}$ are $f$-connected by Lemma \ref{lem:f-/g-connectedness} (ii).

If $\overline{\gamma_{4}} = - (- n, w^{a_{4}}; a_{3}, a_{0}, a_{6})$, we have
\[
 \overline{\gamma_{5}} = + \> (- n, w^{a_{4}}; a_{6}, a_{0}, a_{6}) \ \text{or} \, + (- n, w^{a_{4}}; a_{0}, a_{6}, a_{0}).
\]
On the other hand, the index of $\gamma_{4}$ is $w + 3 (a_{1} - a_{0})$, even though the index of $\gamma_{5}$ is respectively $w + 2 (a_{1} - a_{0})$ or $w + (a_{1} - a_{0})$.
It also contradicts to that $\gamma_{4}, \gamma_{5}, \gamma_{6}$ and $\gamma_{7}$ are $f$-connected.

In conclusion, $\overline{\gamma}$ could not satisfy the condition (b) or (c) introduced in the proof of Proposition 6.5 of \cite{Sat2016}, as long as $\gamma$ satisfies the condition (c) introduced in the proof of Proposition 6.5 of \cite{Sat2016}.
Therefore, Cases 3 and 4 do not happen.
\end{proof}

\begin{proposition}[cf.\ Proposition 7.1 of \cite{Sat2016}]
\label{prop:R7_V}
If $\gamma$ is a 3-cycle of $(R_{7}, S)$ satisfying $l(\gamma) = 6$, $|T_{n}(\gamma)| = 2$ and $|T_{n + 1}(\gamma)| = 4$, then $\gamma$ is a 3-boundary.
\end{proposition}

\begin{proof}
Replacing $a_{4}$ with $a_{6}$, the argument in the proof of Proposition 7.1 of \cite{Sat2016} works well even if we select $R_{7}$ instead of $R_{5}$, because we have the above lemmas and that $w^{a} = w^{b}$ if and only if $a = b$ ($w, a, b \in R_{7}$).
Therefore, if $\gamma$ is a 3-cycle of $(R_{7}, S)$ satisfying $l(\gamma) = 6$, $|T_{n}(\gamma)| = 2$ and $|T_{n + 1}(\gamma)| = 4$, we have
\begin{align*}
 \gamma
 & = + \> (n, w; a_{0}, a_{1}, a_{0}) - (n, w; a_{1}, a_{0}, a_{1}) + (n + 1, w^{a_{0}}; a_{1}, a_{0}, a_{1}) \\
 & \qquad + \> (n + 1, w^{a_{0}}; a_{0}, a_{4}, a_{1}) - (n + 1, w^{a_{1}}; a_{2}, a_{1}, a_{2}) - (n + 1, w^{a_{1}}; a_{2}, a_{0}, a_{1})
\end{align*}
up to sign, reverse and reflection, with some $w \in R_{7}$ and distinct $a_{0}, a_{1} \in R_{7}$.
Since $\gamma = \partial (n, w; a_{0}, a_{1}, a_{0}, a_{1})$, we have the claim.
\end{proof}

\begin{proposition}[cf.\ Proposition 7.2 of \cite{Sat2016}]
\label{prop:R7_VI}
There are no 3-cycles $\gamma$ of $(R_{7}, S)$ satisfying $l(\gamma) = 7$, $|T_{n}(\gamma)| = 2$ and $|T_{n + 1}(\gamma)| = 5$.
\end{proposition}

\begin{proof}
Replacing $a_{4}$ with $a_{6}$, the argument in the proof of Proposition 7.2 of \cite{Sat2016} works well even if we select $R_{7}$ instead of $R_{5}$, because we have the above lemmas and that $w^{a} = w^{b}$ if and only if $a = b$ ($w, a, b \in R_{7}$).
We thus have the claim.
\end{proof}

\begin{proposition}[cf.\ Proposition 8.1 of \cite{Sat2016}]
\label{prop:R7_VII}
If $\gamma$ is a 3-cycle of $(R_{7}, S)$ satisfying $l(\gamma) = 6$, $|T_{n}(\gamma)| = 3$ and $|T_{n + 1}(\gamma)| = 3$, then $\gamma$ is a 3-boundary.
\end{proposition}

\begin{proof}
Replacing $a_{4}$ with $a_{6}$, the argument in the proof of Proposition 8.1 of \cite{Sat2016} works well even if we select $R_{7}$ instead of $R_{5}$, because we have the above lemmas, $w^{a_{6} a_{1} a_{0} a_{1}} = w + 6 s$, and that $w^{a} = w^{b}$ if and only if $a = b$ ($w, a, b \in R_{7}$).
Therefore, if $\gamma$ is a 3-cycle of $(R_{7}, S)$ satisfying $l(\gamma) = 6$, $|T_{n}(\gamma)| = 3$ and $|T_{n + 1}(\gamma)| = 3$, we have
\begin{align*}
 \gamma
 & = + \> (n, w; a_{0}, a_{1}, a_{0}) - (n, w; a_{6}, a_{0}, a_{1}) - (n, w; a_{6}, a_{1}, a_{0}) \\
 & \qquad + \> (n + 1, w^{a_{0}}; a_{1}, a_{0}, a_{6}) - (n + 1, w^{a_{1}}; a_{3}, a_{2}, a_{0}) - (n + 1, w^{a_{6}}; a_{0}, a_{1}, a_{0})
\end{align*}
up to sign, reverse and reflection, with some $w \in R_{7}$ and distinct $a_{0}, a_{1} \in R_{7}$.
Since $\gamma = - \partial (n, w; a_{6}, a_{0}, a_{1}, a_{0})$, we have the claim.
\end{proof}

\begin{proposition}[cf.\ Proposition 8.2 of \cite{Sat2016}]
\label{prop:R7_VIII}
If $\gamma$ is a 3-cycle of $(R_{7}, S)$ satisfying $l(\gamma) = 7$, $|T_{n}(\gamma)| = 3$ and $|T_{n + 1}(\gamma)| = 4$, then $\gamma$ is a 3-boundary.
\end{proposition}

\begin{proof}
Replace $a_{3}$ and $a_{4}$ with $a_{5}$ and $a_{6}$ respectively, the argument in Case 1 of the proof of Proposition 8.2 of \cite{Sat2016} works well even if we select $R_{7}$ instead of $R_{5}$, because we have the above lemmas,
\begin{align*}
 w^{a_{1} a_{0} a_{2} a_{0}} & = w + s, &
 w^{a_{0} a_{2} a_{0} a_{2}} & = w + 6 s, &
 w^{a_{2} a_{0} a_{1} a_{0}} & = w + s, &
 w^{a_{2} a_{1} a_{0} a_{1}} & = w, \\
 w^{a_{0} a_{5} a_{1} a_{0}} & = w + s, &
 w^{a_{0} a_{5} a_{0} a_{6}} & = w + s, &
 w^{a_{0} a_{1} a_{0} a_{6}} & = w,
\end{align*}
and that $w^{a} = w^{b}$ if and only if $a = b$ ($w, a, b \in R_{7}$).
Therefore, in this case, we have
\begin{align*}
 \gamma
 & = + \> (n, w; a_{0}, a_{1}, a_{0}) - (n, w; a_{2}, a_{0}, a_{1}) - (n, w; a_{2}, a_{1}, a_{0}) \\
 & \phantom{=} \ + (n + 1, w^{a_{0}}; a_{5}, a_{1}, a_{0}) + (n + 1, w^{a_{0}}; a_{5}, a_{0}, a_{6}) \\
 & \phantom{=} \ - (n + 1, w^{a_{1}}; a_{0}, a_{2}, a_{0}) - (n + 1, w^{a_{2}}; a_{0}, a_{1}, a_{0}) \\
 & = - \> \partial (n, w; a_{2}, a_{0}, a_{1}, a_{0})
\end{align*}
up to sign, reverse and reflection, with some $w \in R_{7}$ and distinct $a_{0}, a_{1} \in R_{7}$.

Replacing $a_{4}$ with $a_{6}$, the argument in Case 2 of the proof of Proposition 8.2 of \cite{Sat2016} also works well even if we select $R_{7}$ instead of $R_{5}$.
Therefore, in this case, we have no candidates of $\gamma$.

To consider the remaining cases, we first review the following things in the style of \cite{Sat2016}.
Let $\gamma = \sum_{i = 1}^{7} \gamma_{i}$ be a 3-cycle of $(R_{7}, S)$ satisfying $l(\gamma) = 7$, $|T_{n}(\gamma)| = 3$ and $|T_{n + 1}(\gamma)| = 4$.
Then, in light of Lemmas \ref{lem:f-connected1} and \ref{lem:f-connected3}, we may assume that
\[
 \gamma_{1} + \gamma_{2} + \gamma_{3}
 = + \> (n, w; a_{0}, a_{1}, a_{0}) - (n, w; a_{p}, a_{0}, a_{1}) - (n, w; a_{p}, a_{1}, a_{0})
\]
up to sign and reflection, with some $w \in R_{7}$, distinct $a_{0}, a_{1} \in R_{7}$, and $p \neq 0, 1$.
Immediately, we have
\begin{align*}
 & g(\gamma_{1} + \gamma_{2} + \gamma_{3}) = \\
 & + (n + 1, w^{a_{0}}; a_{1}, a_{0}) - (n + 1, w^{a_{1}}; a_{2}, a_{0}) + (n + 1, w^{a_{0}}; a_{0}, a_{6}) \\
 & - (n + 1, w^{a_{p}}; a_{0}, a_{1}) + (n + 1, w^{a_{0}}; a_{- p}, a_{1}) - (n + 1, w^{a_{1}}; a_{2 - p}, a_{2}) \\
 & - (n + 1, w^{a_{p}}; a_{1}, a_{0}) + (n + 1, w^{a_{1}}; a_{2 - p}, a_{0}) - (n + 1, w^{a_{0}}; a_{- p}, a_{6}).
\end{align*}
We note that $w^{a_{0}}$, $w^{a_{1}}$ and $w^{a_{p}}$ are mutually different.
We let $s = a_{1} - a_{0}$ for the subsequent arguments.

Since we already considered the case $p = 2$ or $6$ as Case 1 or 2, respectively, we assume that $p = 3$, $4$ or $5$.
Since $g(\gamma_{1} + \gamma_{2} + \gamma_{3}) + f(\gamma_{4} + \gamma_{5} + \gamma_{6} + \gamma_{7}) = 0$ by Lemma \ref{lem:3-cycle_condition} (i) and reduced $g(\gamma_{1} + \gamma_{2} + \gamma_{3})$ has four 2-terms of index $w^{a_{0}}$, at least two of $\gamma_{4}, \gamma_{5}, \gamma_{6}$ and $\gamma_{7}$ have index $w^{a_{0}}$.
Therefore, we may assume that
\begin{align*}
 \gamma_{6} & = - \> (n + 1, w^{a_{1}}; a_{2 - p}, a_{2}, a_{0}), \\
 \gamma_{7} & = - \> (n + 1, w^{a_{p}}; a_{0}, a_{1}, a_{0}) \ \text{or} \, - (n + 1, w^{a_{p}}; a_{1}, a_{0}, a_{1}).
\end{align*}
We note that the index of $\overline{\gamma_{6}}$ is $w - 2 (1 + p) s$ and the index of $\overline{\gamma_{7}}$ is respectively $w - 2 (1 + p) s$ or $w + 2 (2 - p) s$.
Furthermore, it is routine to check that we may assume that
\[
 \gamma_{4} + \gamma_{5} =
 \begin{cases}
  + \> (n + 1, w^{a_{0}}; a_{-p}, a_{1}, a_{0}) + (n + 1, w^{a_{0}}; a_{-p}, a_{0}, a_{6}) \ \text{or} \\
  + \> (n + 1, w^{a_{0}}; a_{1}, a_{0}, a_{6}) + (n + 1, w^{a_{0}}; a_{-p}, a_{1}, a_{6}).
 \end{cases} 
\]
We note that the indices of $\overline{\gamma_{4}}$ and $\overline{\gamma_{5}}$ are both $w - 2 (1 + p) s$ in the former case, and $w$ and $w + 2 (5 - p) s$ respectively in the latter case.
Since $\overline{\gamma_{7}}$ is not of type (0 or) 1, in light of Lemmas \ref{lem:3-cycle_condition} (iii), \ref{lem:f-/g-connectedness} (i), \ref{lem:f-connected1} and \ref{lem:f-connected2}, $\gamma_{4}, \gamma_{5}, \gamma_{6}$ and $\gamma_{7}$ are $g$-connected.
Therefore, in light of Lemma \ref{lem:f-/g-connectedness} (iii), we have
\begin{align*}
 \gamma
 & = + \> (n, w; a_{0}, a_{1}, a_{0}) - (n, w; a_{p}, a_{0}, a_{1}) - (n, w; a_{p}, a_{1}, a_{0}) \\
 & \phantom{=} \ + (n + 1, w^{a_{0}}; a_{-p}, a_{1}, a_{0}) + (n + 1, w^{a_{0}}; a_{-p}, a_{0}, a_{6}) \\
 & \phantom{=} \ - (n + 1, w^{a_{1}}; a_{2 - p}, a_{2}, a_{0}) - (n + 1, w^{a_{p}}; a_{0}, a_{1}, a_{0}).
\end{align*}
Since $\gamma = - \partial (n, w; a_{p}, a_{0}, a_{1}, a_{0})$, we have the claim.
\end{proof}

We are now ready to prove Theorem \ref{thm:lower_bound_R7}.

\begin{proof}[Proof of Theorem \ref{thm:lower_bound_R7}]
It is routine to see that a 3-chain
\begin{align*}
 \gamma
 & = + \> (0, 0; 0, 6, 1) - (1, 5; 5, 1, 4) - (0, 0; 1, 6, 0) + (1, 2; 6, 3, 0) \\
 & \phantom{=} \ + (0, 0; 0, 1, 6) - (1, 2; 2, 6, 3) - (0, 0; 6, 1, 0) + (1, 5; 1, 4, 0)
\end{align*}
of $(R_{7}, \mathbb{Z} \times R_{7})$ is a 3-cycle satisfying $\zeta(\pi(\gamma)) = 6$.
We thus have $l(\zeta, \mathbb{Z} \times R_{7}) \neq 0$.
Then, in light of Propositions \ref{prop:R7_list}, \ref{prop:R7_I}--\ref{prop:R7_VIII} and Lemma \ref{lem:reverse} (iv), we immediately have the claim.
\end{proof}

\section{A lower bound of the length of the 3-cocycle $\eta$ of $O_{6}$}
\label{sec:lower_bound_of_eta}

The aim of this section is to show the following key theorem.

\begin{theorem}
\label{thm:lower_bound_O6}
For the 3-cocycle $\eta$ of the octahedral quandle $O_{6}$, defined in Section \ref{sec:preliminaries}, we have $l(\eta, \mathbb{Z} \times O_{6}) \geq 8$.
\end{theorem}

We start with discussing the following things.
Let $\omega_{0}$ and $\omega_{1}$ be distinct vertices of the regular octahedron depicted in the left-hand side of Figure \ref{fig:octahedral_quandle} which stand as the end points of an edge.
Furthermore, let $\omega_{2}$, $\omega_{3}$, $\omega_{4}$ and $\omega_{5}$ denote the other vertices of the octahedron in the manner as depicted in the right-hand side of Figure \ref{fig:octahedral_quandle}.
Then, we have $\omega_{a}^{\omega_{b}} = \omega_{a^{b}}$ by the definition of the binary operation of $O_{6}$.

For each 3-chain $\gamma = \sum_{i = 1}^{k} \varepsilon_{i} (a_{i}, b_{i}, c_{i})$ of $O_{6}$ and $g \in G(O_{6})$, define a 3-chain $\gamma^{g}$ of $O_{6}$ by $\gamma^{g} = \sum_{i = 1}^{k} \varepsilon_{i} (a_{i}^{g}, b_{i}^{g}, c_{i}^{g})$.
Assume that $\gamma$ is a 3-cycle.
Then, it is routine to see that $\gamma^{g}$ is a 3-cycle homologous to $\gamma$ (see Lemma 5 of \cite{IK2014} for example).
Therefore, we have $\eta(\gamma) = \eta(\gamma^{g})$.
We note that there is an element $g$ of $G(X)$ satisfying $\omega_{a}^{g} = a$ for any $a \in O_{6}$, because any rotational transformation of the octahedron is realized as compositions of some $\pi / 2$-rotations about lines each of which passes through a vertex and the center of the octahedron.
Therefore, we immediately have the following lemma.

\begin{lemma}
\label{lem:O6_weight}
If $\gamma = \sum_{i = 1}^{k} \varepsilon_{i} (\omega_{a_{i}}, \omega_{b_{i}}, \omega_{c_{i}})$ is a 3-cycle of $O_{6}$, then we have $\eta(\gamma) = \sum_{i = 1}^{k} \varepsilon_{i} \eta (a_{i}, b_{i}, c_{i})$.
\end{lemma}

Let $S$ denote our $O_{6}$-set $\mathbb{Z} \times O_{6}$.
To enumerate candidates of 3-cycles $\gamma$ of $(O_{6}, S)$ satisfying $1 \leq l(\gamma) \leq 7$ and $\eta(\pi(\gamma)) \neq 0$ in a similar way to Proposition \ref{prop:R7_list}, we show the following lemmas.

\begin{lemma}
\label{lem:O6_f2_usual}
Let $\gamma$ be the 3-chain in Lemma \ref{lem:f-connected2} with $X = O_{6}$.
Assume that $b \neq [a + 3]$.
Then, we have $g(\gamma) \neq 0$.
Furthermore, let $\delta_{3}, \delta_{4}, \dots, \delta_{k + 2}$ be 3-terms of $(O_{6}, S)$ whose degrees are $n + 1$, and $\delta = \sum_{i = 3}^{k + 2} \delta_{i}$.
If $g(\gamma) + f(\delta) = 0$, then we have $k \geq 4$.
\end{lemma}

\begin{proof}
Since $b \neq [a + 3]$, we may assume that $a = \omega_{0}$ and $b = \omega_{1}$.
Then, we have
\begin{align}
 g(\gamma)
 & = + \> (n + 1, u^{\omega_{0}}; \omega_{1}, \omega_{0}) - (n + 1, u^{\omega_{1}}; \omega_{5}, \omega_{0}) + (n + 1, u^{\omega_{0}}; \omega_{0}, \omega_{2}) \notag \\
 & \phantom{=} \ - (n + 1, u^{\omega_{1}}; \omega_{0}, \omega_{1}) + (n + 1, u^{\omega_{0}}; \omega_{2}, \omega_{1}) - (n + 1, u^{\omega_{1}}; \omega_{1}, \omega_{5}), \label{eq:O6_f2_usual}
\end{align}
and thus $g(\gamma) \neq 0$.

It is routine to check that we have $u^{\omega_{0}} \neq u^{\omega_{1}}$.
Assume that $g(\gamma) + f(\delta) = 0$.
Then, since reduced $g(\gamma)$ has three positive (or negative) 2-terms of index $u^{\omega_{0}}$ (or $u^{\omega_{1}}$), at least two of $\delta_{3}, \delta_{4}, \dots, \delta_{k + 2}$ have index $u^{\omega_{0}}$ (or $u^{\omega_{1}}$).
We thus have $k \geq 4$.
\end{proof}

\begin{lemma}
\label{lem:O6_f2_special}
Let $\gamma$ be the 3-chain in Lemma \ref{lem:f-connected2} with $X = O_{6}$.
Assume that $b = [a + 3]$.
If $u = a$ or $b$, then $\gamma$ is a 3-cycle satisfying $\eta(\pi(\gamma)) = 0$.
Otherwise, we have $g(\gamma) \neq 0$.
\end{lemma}

\begin{proof}
We may assume that $a = \omega_{0}$.
Then, since $b = [a + 3]$, we have $b = \omega_{3}$.
If $u = \omega_{0}$ or $\omega_{3}$, we immediately have $g(\gamma) = 0$.
Therefore, $\gamma$ is a 3-cycle.
In light of Lemma \ref{lem:O6_weight}, we have $\eta(\pi(\gamma)) = 0$.

If $u \neq \omega_{0}, \omega_{3}$, we may assume that $u = \omega_{1}$.
Obviously, we have
\begin{align}
 g(\gamma)
 & = + \> (n + 1, \omega_{2}; \omega_{3}, \omega_{0}) + (n + 1, \omega_{2}; \omega_{0}, \omega_{3}) \notag \\
 & \phantom{=} \ - (n + 1, \omega_{5}; \omega_{0}, \omega_{3}) - (n + 1, \omega_{5}; \omega_{3}, \omega_{0}), \label{eq:O6_f2_special}
\end{align}
and thus $g(\gamma) \neq 0$.
\end{proof}

\begin{lemma}
\label{lem:O6_f2_special_2-5}
Let $\gamma$ be the 3-chain in Lemma \ref{lem:f-connected2} with $X = O_{6}$.
Assume that $b = [a + 3]$ and $u \neq a, b$.
Furthermore, for $1 \leq k \leq 5$, let $\gamma_{3}, \gamma_{4}, \dots, \gamma_{k + 2}$ be 3-terms of $(O_{6}, S)$ whose degrees are at least $n + 1$.
If $\gamma + \sum_{i = 3}^{k + 2} \gamma_{i}$ is a 3-cycle of length $k + 2$, then we have $\eta \left( \pi \left( \gamma + \sum_{i = 3}^{k + 2} \gamma_{i} \right) \right) = 0$.
\end{lemma}

Although the proof of Lemma \ref{lem:O6_f2_special_2-5} is straightforward, its argument is rather complicated.
We thus defer the proof of this lemma to Section \ref{sec:proof_of_lemma_O6_f2_special_2-5}.

\begin{lemma}
\label{lem:O6_f3}
Let $\gamma$ be one of the 3-chains in Lemma \ref{lem:f-connected3} with $X = O_{6}$.
Then, we have $g(\gamma) \neq 0$.
\end{lemma}

\begin{proof}
We first consider the case that
\[
 \gamma = + \> (n, u; a, b, a) - (n, u; c, a, b) - (n, u; c, b, a)
\]
up to sign, with some mutually distinct $a, b, c \in O_{6}$.
If $b = [a + 3]$, we may assume that $a = \omega_{0}$ and $c = \omega_{1}$, and have $b = \omega_{3}$.
Then, we have
\begin{align}
 g(\gamma)
 & = + \> (n + 1, u^{\omega_{0}}; \omega_{3}, \omega_{0}) + (n + 1, u^{\omega_{0}}; \omega_{0}, \omega_{3}) \notag \\
 & \phantom{=} \ - (n + 1, u^{\omega_{1}}; \omega_{0}, \omega_{3}) - (n + 1, u^{\omega_{1}}; \omega_{3}, \omega_{0}). \label{eq:O6_f3_special_i}
\end{align}
It is routine to see that we have $u^{\omega_{0}} \neq u^{\omega_{1}}$.
We thus have $g(\gamma) \neq 0$.
If $b \neq [a + 3]$, we may assume that $a = \omega_{0}$ and $b = \omega_{1}$.
Then, we have
\begin{align}
 g(\gamma)
 & = + \> (u^{\omega_{0}}; \omega_{1}, \omega_{0}) - (u^{\omega_{1}}; \omega_{5}, \omega_{0}) + (u^{\omega_{0}}; \omega_{0}, \omega_{2}) \notag \\
 & \phantom{=} \ - (u^{c}; \omega_{0}, \omega_{1}) \> \uline{+ \> (u^{\omega_{0}}; c^{\omega_{0}}, \omega_{1})} - (u^{\omega_{1}}; c^{\omega_{1}}, \omega_{5}) \notag \\
 & \phantom{=} \ - (u^{c}; \omega_{1}, \omega_{0}) \> \uuline{+ \> (u^{\omega_{1}}; c^{\omega_{1}}, \omega_{0})} - (u^{\omega_{0}}; c^{\omega_{0}}, \omega_{2}). \label{eq:O6_f3_usual_i}
\end{align}
We note that the single- or double-underlined 2-term vanishes if and only if $c = \omega_{5}$ or $\omega_{2}$ respectively, and the other 2-terms never vanish.
Since reduced $g(\gamma)$ has exactly one 2-term of color $(c^{\omega_{1}}, \omega_{5})$, we have $g(\gamma) \neq 0$.

We next consider the case that
\[
 \gamma = + \> (n, u; a, b, a) - (n, u; a, b, c) - (n, u; b, a, c)
\]
up to sign, with some mutually distinct $a, b, c \in O_{6}$.
If $b = [a + 3]$, we may assume that $a = \omega_{0}$ and $c = \omega_{1}$, and have $b = \omega_{3}$.
Then, we have
\begin{align}
 g(\gamma)
 & = + \> (n + 1, u^{\omega_{0}}; \omega_{3}, \omega_{0}) + (n + 1, u^{\omega_{0}}; \omega_{0}, \omega_{3}) \notag \\
 & \phantom{=} \ - (n + 1, u^{\omega_{1}}; \omega_{5}, \omega_{2}) - (n + 1, u^{\omega_{1}}; \omega_{2}, \omega_{5}). \label{eq:O6_f3_special_ii}
\end{align}
Since $u^{\omega_{0}} \neq u^{\omega_{1}}$, we have $g(\gamma) \neq 0$.
If $b \neq [a + 3]$, we may assume that $a = \omega_{0}$ and $b = \omega_{1}$.
Then, we have
\begin{align}
 g(\gamma)
 & = + \> (u^{\omega_{0}}; \omega_{1}, \omega_{0}) - (u^{\omega_{1}}; \omega_{5}, \omega_{0}) + (u^{\omega_{0}}; \omega_{0}, \omega_{2}) \notag \\
 & \phantom{=} \ - (u^{\omega_{0}}; \omega_{1}, c) \> \uline{+ \> (u^{\omega_{1}}; \omega_{5}, c)} - (u^{c}; \omega_{0}^{c}, \omega_{1}^{c}) \notag \\
 & \phantom{=} \ - (u^{\omega_{1}}; \omega_{0}, c) \> \uuline{+ \> (u^{\omega_{0}}; \omega_{2}, c)} - (u^{c}; \omega_{1}^{c}, \omega_{0}^{c}). \label{eq:O6_f3_usual_ii}
\end{align}
We note that the single- or double-underlined 2-term vanishes if and only if $c = \omega_{5}$ or $\omega_{2}$ respectively, and the other 2-terms never vanish.
Since reduced $g(\gamma)$ has exactly one 2-term of color $(\omega_{1}, \omega_{0})$, we have $g(\gamma) \neq 0$.
\end{proof}

\begin{lemma}
\label{lem:O6_f3_usual}
Let $\gamma$ be one of the 3-chains in Lemma \ref{lem:f-connected3} with $X = O_{6}$.
Assume that $b \neq [a + 3]$.
Furthermore, let $\delta_{4}, \delta_{5}, \dots, \delta_{k + 3}$ be 3-terms of $(O_{6}, S)$ whose degrees are $n + 1$, and $\delta = \sum_{i = 4}^{k + 3} \delta_{i}$.
If $g(\gamma) + f(\delta) = 0$, then we have $k \geq 3$.
\end{lemma}

\begin{proof}
We first consider the case that
\[
 \gamma = + \> (n, u; a, b, a) - (n, u; c, a, b) - (n, u; c, b, a)
\]
up to sign, with some mutually distinct $a, b, c \in O_{6}$.
Since $b \neq [a + 3]$, we may assume that $a = \omega_{0}$ and $b = \omega_{1}$.
Then, we have the formula (\ref{eq:O6_f3_usual_i}).
We note that we have $u^{\omega_{0}} \neq u^{\omega_{1}}$.
Furthermore, it is routine to check that $u^{\omega_{0}}$, $u^{\omega_{1}}$ and $u^{c}$ are mutually different except the following cases:
\begin{itemize}
\item[(a)]
$u = \omega_{0}$ or $\omega_{3}$ and $c = \omega_{3}$, and
\item[(b)]
$u = \omega_{1}$ or $\omega_{4}$ and $c = \omega_{4}$.
\end{itemize}
Assume that $g(\gamma) + f(\delta) = 0$.

In the case other than (a) or (b), since reduced $g(\gamma)$ has at least one 2-term of index $u^{\omega_{1}}$ (or $u^{c}$), at least one of $\delta_{4}, \delta_{5}, \dots, \delta_{k + 3}$ has index $u^{\omega_{1}}$ (or $u^{c}$).
If $c = \omega_{5}$, since reduced $g(\gamma)$ has at least one 2-terms of index $u^{\omega_{0}}$, at least one of $\delta_{4}, \delta_{5}, \dots, \delta_{k + 3}$ has index $u^{\omega_{0}}$.
We thus have $k \geq 3$.
If $c \neq \omega_{5}$, since reduced $g(\gamma)$ has four 2-terms of index $u^{\omega_{0}}$, at least two of $\delta_{4}, \delta_{5}, \dots, \delta_{k + 3}$ have index $u^{\omega_{0}}$.
We thus have $k \geq 4$.

In the case (a), we have $u^{\omega_{0}} = u^{c} = u$ and
\begin{align}
 g(\gamma)
 & = + \> (u; \omega_{0}, \omega_{2}) + (u; \omega_{3}, \omega_{1}) - (u; \omega_{0}, \omega_{1}) - (u; \omega_{3}, \omega_{2}) \notag \\
 & \phantom{=} \ + (u^{\omega_{1}}; \omega_{2}, \omega_{0}) - (u^{\omega_{1}}; \omega_{5}, \omega_{0}) - (u^{\omega_{1}}; \omega_{2}, \omega_{5}). \label{eq:O6_f3_usual_i_a}
\end{align}
Since reduced $g(\gamma)$ has four 2-terms of index $u$, at least two of $\delta_{4}, \delta_{5}, \dots, \delta_{k + 3}$ have index $u$.
Furthermore, since reduced $g(\gamma)$ has three 2-terms of index $u^{\omega_{1}}$, at least one of $\delta_{4}, \delta_{5}, \dots, \delta_{k + 3}$ has index $u^{\omega_{1}}$.
We thus have $k \geq 3$.

In the case (b), we have $u^{\omega_{1}} = u^{c} = u$ and
\begin{align*}
 g(\gamma)
 & = + \> (u^{\omega_{0}}; \omega_{1}, \omega_{0}) + (u^{\omega_{0}}; \omega_{0}, \omega_{2}) + (u^{\omega_{0}}; \omega_{5}, \omega_{1}) - (u^{\omega_{0}}; \omega_{5}, \omega_{2}) \\
 & \phantom{=} \ + (u; \omega_{4}, \omega_{0}) - (u; \omega_{5}, \omega_{0}) - (u; \omega_{0}, \omega_{1}) - (u; \omega_{4}, \omega_{5}) - (u; \omega_{1}, \omega_{0}).
\end{align*}
Since reduced $g(\gamma)$ has four 2-terms of index $u^{\omega_{0}}$, at least two of $\delta_{4}, \delta_{5}, \dots, \delta_{k + 3}$ have index $u^{\omega_{0}}$.
Furthermore, since reduced $g(\gamma)$ has five 2-terms of index $u$, at least two of $\delta_{4}, \delta_{5}, \dots, \delta_{k + 3}$ have index $u$.
We thus have $k \geq 4$.

We next consider the case that
\[
 \gamma = + \> (n, u; a, b, a) - (n, u; a, b, c) - (n, u; b, a, c)
\]
up to sign, with some mutually distinct $a, b, c \in O_{6}$.
We may assume that $a = \omega_{0}$ and $b = \omega_{1}$.
Then, we have the formula (\ref{eq:O6_f3_usual_i}).
Assume that $g(\gamma) + f(\delta) = 0$.
We note that $u^{\omega_{0}} \neq u^{\omega_{1}}$, and $u^{\omega_{0}}$, $u^{\omega_{1}}$ and $u^{c}$ are mutually different except the case (a) or (b).

In the case other than (a) or (b), since reduced $g(\gamma)$ has at least one 2-term of index $u^{\omega_{1}}$ (or $u^{c}$), at least one of $\delta_{4}, \delta_{5}, \dots, \delta_{k + 3}$ has index $u^{\omega_{1}}$ (or $u^{c}$).
If $c = \omega_{2}$, since reduced $g(\gamma)$ has at least one 2-term of index $u^{\omega_{0}}$, at least one of $\delta_{4}, \delta_{5}, \dots, \delta_{k + 3}$ has index $u^{\omega_{0}}$.
We thus have $k \geq 3$.
If $c \neq \omega_{2}$, since reduced $g(\gamma)$ has four 2-terms of index $u^{\omega_{0}}$, at least two of $\delta_{4}, \delta_{5}, \dots, \delta_{k + 3}$ have index $u^{\omega_{0}}$.
We thus have $k \geq 4$.

In the case (a), we have $u^{\omega_{0}} = u^{c} = u$ and
\begin{align}
 g(\gamma)
 & = + \> (u; \omega_{1}, \omega_{0}) + (u; \omega_{0}, \omega_{2}) + (u; \omega_{2}, \omega_{3}) \notag \\
 & \phantom{=} \ - (u; \omega_{1}, \omega_{3}) - (u; \omega_{0}, \omega_{5}) - (u; \omega_{5}, \omega_{0}) \notag \\
 & \phantom{=} \ + (u^{\omega_{1}}; \omega_{5}, \omega_{3}) - (u^{\omega_{1}}; \omega_{5}, \omega_{0}) - (u^{\omega_{1}}; \omega_{0}, \omega_{3}) \label{eq:O6_f3_usual_ii_a}.
\end{align}
Since reduced $g(\gamma)$ has six 2-terms of index $u$, at least two of $\delta_{4}, \delta_{5}, \dots, \delta_{k + 3}$ have index $u$.
Furthermore, since reduced $g(\gamma)$ has three 2-terms of index $u^{\omega_{1}}$, at least one of $\delta_{4}, \delta_{5}, \dots, \delta_{k + 3}$ has index $u^{\omega_{1}}$.
We thus have $k \geq 3$.

In the case (b), we have $u^{\omega_{1}} = u^{c} = u$ and
\begin{align*}
 g(\gamma)
 & = + \> (u^{\omega_{0}}; \omega_{1}, \omega_{0}) + (u^{\omega_{0}}; \omega_{0}, \omega_{2}) + (u^{\omega_{0}}; \omega_{2}, \omega_{4}) - (u^{\omega_{0}}; \omega_{1}, \omega_{4}) \\
 & \phantom{=} \ + (u; \omega_{5}, \omega_{4}) - (u; \omega_{5}, \omega_{0}) - (u; \omega_{2}, \omega_{1}) - (u; \omega_{0}, \omega_{4}) - (u; \omega_{1}, \omega_{2}).
\end{align*}
Since reduced $g(\gamma)$ has four 2-terms of index $u^{\omega_{0}}$, at least two of $\delta_{4}, \delta_{5}, \dots, \delta_{k + 3}$ have index $u^{\omega_{0}}$.
Furthermore, since reduced $g(\gamma)$ has five 2-terms of index $u$, at least two of $\delta_{4}, \delta_{5}, \dots, \delta_{k + 3}$ have index $u$.
We thus have $k \geq 4$.
\end{proof}

\begin{lemma}
\label{lem:O6_f3_special}
Let $\gamma$ be one of the 3-chains in Lemma \ref{lem:f-connected3} with $X = O_{6}$.
Assume that $b = [a + 3]$.
Furthermore, for $1 \leq k \leq 4$, let $\gamma_{4}, \gamma_{5}, \dots, \gamma_{k + 3}$ be 3-terms of $(O_{6}, S)$ whose degrees are at least $n + 1$.
If $\gamma + \sum_{i = 4}^{k + 3} \gamma_{i}$ is a 3-cycle of length $k + 3$, then we have $\eta \left( \pi \left( \gamma + \sum_{i = 4}^{k + 3} \gamma_{i} \right) \right) = 0$.
\end{lemma}

Although the proof of Lemma \ref{lem:O6_f3_special} is straightforward, its argument is somewhat complicated.
We thus defer the proof of this lemma to Section \ref{sec:proof_of_lemma_O6_f3_special}.
In light of Lemmas \ref{lem:3-cycle_condition}--\ref{lem:f-/g-connectedness}, \ref{lem:f-connected1}, \ref{lem:O6_f2_usual}--\ref{lem:O6_f3_special} and Remark \ref{rem:reverse}, we immediately have the following proposition in a similar way to Proposition \ref{prop:R7_list}.

\begin{proposition}
\label{prop:O6_list}
Let $\gamma$ be a 3-cycle of $(O_{6}, S)$ whose length is $l$ {\upshape (}$1 \leq l \leq 7${\upshape )}, and $n$ the minimal degree of $\gamma$.
If $\eta(\pi(\gamma)) \neq 0$, then we have the following cases up to reverse.
\begin{itemize}
\item[(A)]
$4 \leq l \leq 7$ and $|T_{n}(\gamma)| = l$.
\item[(B)]
$6 \leq l \leq 7$, $|T_{n}(\gamma)| = 2$ and $|T_{n + 1}(\gamma)| = l - 2$.
\item[(C)]
$6 \leq l \leq 7$, $|T_{n}(\gamma)| = 3$ and $|T_{n + 1}(\gamma)| = l - 3$.
\end{itemize}
\end{proposition}

Moreover, to prove Theorem \ref{thm:lower_bound_O6}, we require the following claims.

\begin{lemma}
\label{lem:O6_f_and_g}
Let $\gamma_{1}, \gamma_{2}, \dots, \gamma_{l}$ be $f$- and $g$-connected 3-terms of $(O_{6}, S)$ having the same degree $n$, and $\gamma = \sum_{i = 1}^{l} \gamma_{i}$.
Then, we have $\eta(\pi(\gamma)) = 0$.
\end{lemma}

\begin{proof}
For the subsequent arguments, we first list the values of $\omega_{0}^{\omega_{a} \omega_{b} \omega_{c}}$ for each $\omega_{a}, \omega_{b}, \omega_{c} \in O_{6}$ ($\omega_{b} \neq \omega_{a}, \omega_{c}$) in Table \ref{tab:O6_f_and_g}.\footnote{The author listed the values of $\omega_{0}^{\omega_{a} \omega_{b} \omega_{c}}$ with the aid of a computer. The C source code for the task is available at \url{https://github.com/ayminoue/LQC/blob/main/LI.c}.}
In the table, $p$ takes value in $\{ 1, 2, 4, 5 \}$ and the values $p_{i}$ are defined as follows.
\begin{align*}
 p_{1} & = p, &
 p_{2} & = \begin{cases}
	    2 & \text{if $p = 1$}, \\
            4 & \text{if $p = 2$}, \\
            5 & \text{if $p = 4$}, \\
            1 & \text{if $p = 5$},
	   \end{cases} &
 p_{4} & = \begin{cases}
	    4 & \text{if $p = 1$}, \\
            5 & \text{if $p = 2$}, \\
            1 & \text{if $p = 4$}, \\
            2 & \text{if $p = 5$},
	   \end{cases} &
 p_{5} & = \begin{cases}
	    5 & \text{if $p = 1$}, \\
            1 & \text{if $p = 2$}, \\
            2 & \text{if $p = 4$}, \\
            4 & \text{if $p = 5$}.
	   \end{cases}
\end{align*}
\begin{table}[htbp]
 \centering
 \caption{The values of $\omega_{0}^{\omega_{a} \omega_{b} \omega_{c}}$ ($\omega_{b} \neq \omega_{a}, \omega_{c}$).}
 \label{tab:O6_f_and_g}
 \begin{tabular}{|c|l|}
  \hline
  $\omega_{0}^{\omega_{a} \omega_{b} \omega_{c}}$ & \multicolumn{1}{c|}{$(a, b, c)$} \\ \hline \hline
  $\omega_{0}$ & 
  \hspace{-0.8em} \begin{tabular}{l}
  $(0, 1, 4)$, $(0, 2, 5)$, $(0, 3, 0)$, $(0, 4, 1)$, $(0, 5, 2)$, $(1, 0, 5)$, $(1, 2, 4)$, $(1, 3, 2)$, \\
  $(1, 4, 0)$, $(1, 4, 3)$, $(1, 5, 4)$, $(2, 0, 1)$, $(2, 1, 5)$, $(2, 3, 4)$, $(2, 4, 5)$, $(2, 5, 0)$, \\
  $(2, 5, 3)$, $(3, 0, 3)$, $(3, 1, 4)$, $(3, 2, 5)$, $(3, 4, 1)$, $(3, 5, 2)$, $(4, 0, 2)$, $(4, 1, 0)$, \\
  $(4, 1, 3)$, $(4, 2, 1)$, $(4, 3, 5)$, $(4, 5, 1)$, $(5, 0, 4)$, $(5, 1, 2)$, $(5, 2, 0)$, $(5, 2, 3)$, \\
  $(5, 3, 1)$, $(5, 4, 2)$
  \end{tabular} \hspace{-0.8em} \\ \hline
  $\omega_{3}$ &
  \hspace{-0.8em} \begin{tabular}{l}
  $(1, 0, 2)$, $(1, 2, 1)$, $(1, 3, 5)$, $(1, 5, 1)$, $(2, 0, 4)$, $(2, 1, 2)$, $(2, 3, 1)$, $(2, 4, 2)$, \\
  $(4, 0, 5)$, $(4, 2, 4)$, $(4, 3, 2)$, $(4, 5, 4)$, $(5, 0, 1)$, $(5, 1, 5)$, $(5, 3, 4)$, $(5, 4, 5)$ \\
  \end{tabular} \hspace{-0.8em} \\ \hline
  $\omega_{p}$ &
  \hspace{-0.8em} \begin{tabular}{l}
  $(0, p_{1}, 0)$, $(0, p_{2}, p_{1})$, $(0, p_{2}, p_{4})$, $(0, 3, p_{2})$, $(0, p_{4}, 3)$, $(p_{1}, 0, p_{1})$, \\
  $(p_{1}, 0, p_{4})$, $(p_{1}, p_{2}, 0)$, $(p_{1}, p_{4}, p_{2})$, $(p_{1}, p_{5}, 0)$, $(p_{2}, 0, 3)$, $(p_{2}, p_{1}, p_{4})$, \\
  $(p_{2}, 3, 0)$, $(p_{2}, p_{4}, p_{1})$, $(p_{2}, p_{5}, p_{2})$, $(3, 0, p_{2})$, $(3, p_{1}, 0)$, $(3, p_{2}, p_{1})$, \\
  $(3, p_{2}, p_{4})$, $(3, p_{4}, 3)$, $(p_{4}, p_{1}, p_{2})$, $(p_{4}, p_{2}, 3)$, $(p_{4}, 3, p_{1})$, $(p_{4}, 3, p_{4})$, \\
  $(p_{4}, p_{5}, 3)$
  \end{tabular} \hspace{-0.8em} \\ \hline
 \end{tabular}
\end{table}

In light of Lemma \ref{lem:f-/g-connectedness} (ii), we may assume that $\gamma_{1}, \gamma_{2}, \dots, \gamma_{l}$ have the same index $\omega_{0}$.
Furthermore, in light of Lemma \ref{lem:f-/g-connectedness} (iii), we are allowed to rewrite $\gamma$ as $\sum \alpha_{a b c} (n, \omega_{0}; \omega_{a}, \omega_{b}, \omega_{c})$ ($\alpha_{a b c} \in \mathbb{Z}$), where the sum runs over all triples $(a, b, c)$ appearing in a cell in the second column of Table \ref{tab:O6_f_and_g}.
The assumption $f(\gamma) = g(\gamma) = 0$ yields a system of linear equations in variables $\alpha_{a b c}$.
Solving it in a similar way to the proof of Lemma 6.1 of \cite{Sat2016}, we have
\[
 \gamma =
 \begin{cases}
  \alpha_{01} \gamma_{01} + \alpha_{02} \gamma_{02} + \alpha_{03} \gamma_{03} + \alpha_{04} \gamma_{04} + \alpha_{05} \gamma_{05} + \alpha_{06} \gamma_{06} & \text{if $\omega_{0}^{\omega_{a} \omega_{b} \omega_{c}} = \omega_{0}$}, \\
  0 & \text{if $\omega_{0}^{\omega_{a} \omega_{b} \omega_{c}} = \omega_{3}$}, \\
  \alpha_{p1} \gamma_{p1} + \alpha_{p2} \gamma_{p2} + \alpha_{p3} \gamma_{p3} & \text{if $\omega_{0}^{\omega_{a} \omega_{b} \omega_{c}} = \omega_{p}$},
 \end{cases}
\]
where $p \in \{ 1, 2, 4, 5 \}$, $\alpha_{ij} \in \mathbb{Z}$, and
\begin{align*}
 \gamma_{01}
 & = + \> (\omega_{0}, \omega_{3}, \omega_{0}) - (\omega_{3}, \omega_{0}, \omega_{3}), \\
 \gamma_{02}
 & = + \> (\omega_{0}, \omega_{1}, \omega_{4}) + (\omega_{0}, \omega_{4}, \omega_{1}) - (\omega_{3}, \omega_{1}, \omega_{4}) - (\omega_{3}, \omega_{4}, \omega_{1}), \\
 \gamma_{03}
 & = + \> (\omega_{0}, \omega_{2}, \omega_{5}) + (\omega_{0}, \omega_{5}, \omega_{2}) - (\omega_{3}, \omega_{2}, \omega_{5}) - (\omega_{3}, \omega_{5}, \omega_{2}), \\
 \gamma_{04}
 & = + \> (\omega_{1}, \omega_{4}, \omega_{0}) + (\omega_{4}, \omega_{1}, \omega_{0}) - (\omega_{1}, \omega_{4}, \omega_{3}) - (\omega_{4}, \omega_{1}, \omega_{3}), \\
 \gamma_{05}
 & = + \> (\omega_{1}, \omega_{4}, \omega_{0}) + (\omega_{4}, \omega_{1}, \omega_{0}) + (\omega_{2}, \omega_{5}, \omega_{3}) + (\omega_{5}, \omega_{2}, \omega_{3}) \\
 & \phantom{=} \ - (\omega_{3}, \omega_{1}, \omega_{4}) - (\omega_{3}, \omega_{4}, \omega_{1}) - (\omega_{3}, \omega_{2}, \omega_{5}) - (\omega_{3}, \omega_{5}, \omega_{2}), \\
 \gamma_{06}
 & = + \> (\omega_{1}, \omega_{4}, \omega_{0}) + (\omega_{4}, \omega_{1}, \omega_{0}) + (\omega_{2}, \omega_{5}, \omega_{0}) + (\omega_{5}, \omega_{2}, \omega_{0}) \\
 & \phantom{=} \ - (\omega_{3}, \omega_{1}, \omega_{4}) - (\omega_{3}, \omega_{4}, \omega_{1}) - (\omega_{3}, \omega_{2}, \omega_{5}) - (\omega_{3}, \omega_{5}, \omega_{2}), \\
 \gamma_{p1}
 & = + \> (\omega_{0}, \omega_{3}, \omega_{p_{2}}) + (\omega_{0}, \omega_{p_{2}}, \omega_{p_{4}}) + (\omega_{0}, \omega_{p_{4}}, \omega_{3}) \\
 & \phantom{=} \ - (\omega_{3}, \omega_{p_{2}}, \omega_{p_{4}}) - (\omega_{3}, \omega_{p_{4}}, \omega_{3}), \\
 \gamma_{p2}
 & = + \> (\omega_{3}, \omega_{0}, \omega_{p_{2}}) + (\omega_{3}, \omega_{p_{1}}, \omega_{0}) + (\omega_{3}, \omega_{p_{2}}, \omega_{p_{1}}) \\
 & \phantom{=} \ - (\omega_{0}, \omega_{p_{1}}, \omega_{0}) - (\omega_{0}, \omega_{p_{2}}, \omega_{p_{1}}), \\
 \gamma_{p3}
 & = + \> (\omega_{p_{2}}, \omega_{0}, \omega_{3}) + (\omega_{p_{2}}, \omega_{3}, \omega_{0}) + (\omega_{p_{2}}, \omega_{p_{1}}, \omega_{p_{4}}) + (\omega_{p_{2}}, \omega_{p_{4}}, \omega_{1}) \\
 & \phantom{=} \ + (\omega_{0}, \omega_{p_{2}}, \omega_{p_{4}}) + (\omega_{0}, \omega_{p_{4}}, \omega_{3}) + (\omega_{3}, \omega_{p_{1}}, \omega_{0}) + (\omega_{3}, \omega_{p_{2}}, \omega_{p_{1}}) \\
 & \phantom{=} \ - (\omega_{0}, \omega_{p_{1}}, \omega_{0}) - (\omega_{0}, \omega_{p_{2}}, \omega_{p_{1}}) - (\omega_{p_{1}}, \omega_{p_{4}}, \omega_{p_{2}}) - (\omega_{p_{4}}, \omega_{p_{1}}, \omega_{p_{2}}) \\
 & \phantom{=} \ - (\omega_{3}, \omega_{p_{2}}, \omega_{p_{4}}) - (\omega_{3}, \omega_{p_{4}}, \omega_{3}).
\end{align*}
In light of Lemma \ref{lem:O6_weight}, we have $\eta(\pi(\gamma_{ij})) = 0$ for each $\gamma_{ij}$.
Therefore, we have $\eta(\pi(\gamma)) = 0$.
\end{proof}

\begin{proposition}
\label{prop:O6_I}
There are no 3-cycles $\gamma$ of $(O_{6}, S)$ satisfying $l(\gamma) = |T_{n}(\gamma)| = 4$ and $\eta(\pi(\gamma)) \neq 0$.
\end{proposition}

\begin{proof}
Let $\gamma_{1}, \gamma_{2}, \gamma_{3}$ and $\gamma_{4}$ be 3-terms of $(O_{6}, S)$ and $\gamma = \gamma_{1} + \gamma_{2} + \gamma_{3} + \gamma_{4}$.
Assume that $\gamma$ is a 3-cycle satisfying $l(\gamma) = |T_{n}(\gamma)| = 4$ and $\eta(\pi(\gamma)) \neq 0$.
Then, in light of Lemma \ref{lem:O6_f_and_g}, $\gamma_{1}, \gamma_{2}, \gamma_{3}$ and $\gamma_{4}$ are not $f$- and $g$-connected.
Therefore, in light of Lemma \ref{lem:f-/g-connectedness} (i), we may assume that $\gamma_{1}, \gamma_{2}, \gamma_{3}$ and $\gamma_{4}$ are not $f$-connected taking reverse if necessary.
Furthermore, in light of Lemmas \ref{lem:3-cycle_condition} (ii), \ref{lem:f-connected1} and \ref{lem:f-connected2}, we may assume that
\begin{align*}
 \gamma_{1} + \gamma_{2} & = + \> (u; a, b, a) - (u; b, a, b), \\
 \gamma_{3} + \gamma_{4} & = + \> (v; c, d, c) - (v; d, c, d)
\end{align*}
with some $u, v \in O_{6}$, distinct $a, b \in O_{6}$, and distinct $c, d \in O_{6}$.

If $b = [a + 3]$ and $u = a$ or $b$, in light of Lemma \ref{lem:O6_f2_special}, $\gamma_{1} + \gamma_{2}$ is a 3-cycle satisfying $\eta(\pi(\gamma_{1} + \gamma_{2})) = 0$.
Therefore, $\gamma_{3} + \gamma_{4}$ is a 3-cycle of length two satisfying $\eta(\pi(\gamma_{3} + \gamma_{4})) \neq 0$.
It contradicts to Proposition \ref{prop:O6_list}.
Obviously, we are faced with the same situation if $d = [c + 3]$ and $v = c$ or $d$
We thus have $b \neq [a + 3]$ or $u \neq a, b$, and $d \neq [c + 3]$ or $v \neq c, d$.

Assume that $b = [a + 3]$ and $u \neq a, b$.
Then, since reduced $g(\gamma_{1} + \gamma_{2})$ has four 2-terms, in light of Lemma \ref{lem:3-cycle_condition} (iii), reduced $g(\gamma_{3} + \gamma_{4})$ has four 2-terms.
It yields $d = [c + 3]$ (and $v \neq c, d$).
Therefore, we have $\eta(\pi(\gamma)) = 0$ contradicting to our assumption.
We are obviously faced with the same situation if $d = [c + 3]$ and $v \neq c, d$.

Assume that $b \neq [a + 3]$ and $d \neq [c + 3]$.
Then, we may assume that $a = \omega_{0}$ and $b = \omega_{1}$.
Immediately, we have
\begin{align*}
 g(\gamma_{1} + \gamma_{2})
 & = + \> (u^{\omega_{0}}; \omega_{1}, \omega_{0}) - (u^{\omega_{1}}; \omega_{5}, \omega_{0}) + (u^{\omega_{0}}; \omega_{0}, \omega_{2}) \\
 & \phantom{=} \ - (u^{\omega_{1}}; \omega_{0}, \omega_{1}) + (u^{\omega_{0}}; \omega_{2}, \omega_{1}) - (u^{\omega_{1}}; \omega_{1}, \omega_{5}), \\
 g(\gamma_{3} + \gamma_{4})
 & = + \> (v^{c}; d, c) - (v^{d}; c^{d}, c) + (v^{c}; c, d^{c}) \\
 & \phantom{=} \ - (v^{d}; c, d) + (v^{c}; d^{c}, d) - (v^{d}; d, c^{d}).
\end{align*}
Since the 2-term $+ (v^{c}; d, c)$ survives in reduced $g(\gamma_{3} + \gamma_{4})$, in light of Lemma \ref{lem:3-cycle_condition} (iii), we have
\[
 (v^{c}; d, c) = (u^{\omega_{1}}; \omega_{5}, \omega_{0}), \, (u^{\omega_{1}}; \omega_{0}, \omega_{1}) \ \text{or} \ (u^{\omega_{1}}; \omega_{1}, \omega_{5}).
\]
If $(v^{c}; d, c) = (u^{\omega_{1}}; \omega_{0}, \omega_{1})$, we have $\gamma = 0$ contradicting to our assumption.
If $(v^{c}; d, c) = (u^{\omega_{1}}; \omega_{5}, \omega_{0})$ or $(u^{\omega_{1}}; \omega_{1}, \omega_{5})$, since $c^{d} = \omega_{4}$ or $\omega_{3}$ respectively, the 2-term $- (v^{d}; c^{d}, c)$ survives in reduced $g(\gamma)$.
We thus have $g(\gamma) \neq 0$ contradicting to Lemma \ref{lem:3-cycle_condition} (iii).
\end{proof}

\begin{proposition}
\label{prop:O6_II}
There are no 3-cycles $\gamma$ of $(O_{6}, S)$ satisfying $l(\gamma) = |T_{n}(\gamma)| = 5$ and $\eta(\pi(\gamma)) \neq 0$.
\end{proposition}

\begin{proof}
Let $\gamma_{1}, \gamma_{2}, \dots, \gamma_{5}$ be 3-terms of $(O_{6}, S)$ and $\gamma = \gamma_{1} + \gamma_{2} + \dots + \gamma_{5}$.
Assume that $\gamma$ is a 3-cycle satisfying $l(\gamma) = |T_{n}(\gamma)| = 5$ and $\eta(\pi(\gamma)) \neq 0$.
Then, for the same reason as in the proof of Proposition \ref{prop:O6_I}, we may assume that $\gamma_{1}, \gamma_{2}, \dots, \gamma_{5}$ are not $f$-connected taking reverse if necessary.
Furthermore, in light of Lemmas \ref{lem:3-cycle_condition} (ii) and \ref{lem:f-connected1}--\ref{lem:f-connected3}, we may assume that
\begin{align*}
 \gamma_{1} + \gamma_{2} & = + \> (u; a, b, a) - (u; b, a, b), \\
 \gamma_{3} + \gamma_{4} + \gamma_{5} & =
 \begin{cases}
  \pm \> (v; c, d, c) \mp (v; e, c, d) \mp (v; e, d, c) \ \text{or} \\
  \pm \> (v; c, d, c) \mp (v; c, d, e) \mp (v; d, c, e),
 \end{cases} 
\end{align*}
with some $u, v \in O_{6}$, distinct $a, b \in O_{6}$, and mutually distinct $c, d, e \in O_{6}$.
Then, we respectively have
\begin{align*}
 g(\gamma_{3} + \gamma_{4} + \gamma_{5})
 & =
 \begin{cases}
  \pm \> (v^{c}; d, c) \> \uline{\mp \> (v^{d}; c^{d}, c)} \pm (v^{c}; c, d^{c}) \\
  \mp \> (v^{e}; c, d) \> \uuline{\pm \> (v^{c}; e^{c}, d)} \mp (v^{d}; e^{d}, c^{d}) \\[-1ex]
  \mp \> (v^{e}; d, c) \> \uuline{\pm \> (v^{d}; e^{d}, c)} \mp (v^{c}; e^{c}, d^{c}) \ \text{or} \\[-0.5ex]
  \pm \> (v^{c}; d, c) \> \uline{\mp \> (v^{d}; c^{d}, c)} \pm (v^{c}; c, d^{c}) \\
  \mp \> (v^{c}; d, e) \> \uuline{\pm \> (v^{d}; c^{d}, e)} \mp (v^{e}; c^{e}, d^{e}) \\[-1ex]
  \mp \> (v^{d}; c, e) \> \uuline{\pm \> (v^{c}; d^{c}, e)} \mp (v^{e}; d^{e}, c^{e}). \\
 \end{cases}
\end{align*}
Since $g(\gamma_{1} + \gamma_{2})$ has the same number of positive and negative 2-terms, in light of Lemma \ref{lem:3-cycle_condition} (iii), the single-underlined 2-terms vanish and the double-underlined 2-terms survive in the above formula.
It yields $d = [c + 3]$.
Thus, we respectively have
\begin{align*}
 g(\gamma_{3} + \gamma_{4} + \gamma_{5}) & =
 \begin{cases}
  \pm \> (v^{c}; d, c) \pm (v^{c}; c, d) \mp (v^{e}; c, d) \mp (v^{e}; d, c) \ \text{or} \\
  \pm \> (v^{c}; d, c) \pm (v^{c}; c, d) \mp (v^{e}; c^{e}, d^{e}) \mp (v^{e}; d^{e}, c^{e}).
 \end{cases}
\end{align*}
Since reduced $g(\gamma_{3} + \gamma_{4} + \gamma_{5})$ has four 2-terms, in light of Lemma \ref{lem:3-cycle_condition} (iii), reduced $g(\gamma_{1} + \gamma_{2})$ has four 2-terms.
It yields $b = [a + 3]$ (and $u \neq a, b$).
Then, it is routine to check that we have $\eta(\pi(\gamma)) = 0$ contradicting to our assumption.
\end{proof}

\begin{proposition}
\label{prop:O6_III}
There are no 3-cycles $\gamma$ of $(O_{6}, S)$ satisfying $l(\gamma) = |T_{n}(\gamma)| = 6$ and $\eta(\pi(\gamma)) \neq 0$.
\end{proposition}

\begin{proposition}
\label{prop:O6_IV}
There are no 3-cycles $\gamma$ of $(O_{6}, S)$ satisfying $l(\gamma) = |T_{n}(\gamma)| = 7$ and $\eta(\pi(\gamma)) \neq 0$.
\end{proposition}

Although the proofs of Propositions \ref{prop:O6_III} and \ref{prop:O6_IV} are straightforward, their arguments are complicated.
Therefore, we defer the proofs of these propositions to Sections \ref{sec:proof_of_proposition_O6_III} and \ref{sec:proof_of_proposition_O6_IV}, respectively.

\begin{proposition}
\label{prop:O6_V}
There are no 3-cycles $\gamma$ of $(O_{6}, S)$ satisfying $l(\gamma) = 6$, $|T_{n}(\gamma)| = 2$, $|T_{n + 1}(\gamma)| = 4$ and $\eta(\pi(\gamma)) \neq 0$.
\end{proposition}

\begin{proof}
Let $\gamma_{1}, \gamma_{2}, \dots, \gamma_{6}$ be 3-terms of $(O_{6}, S)$ and $\gamma = \gamma_{1} + \gamma_{2} + \dots + \gamma_{6}$.
Assume that $\gamma$ is a 3-cycle satisfying $l(\gamma) = 6$, $|T_{n}(\gamma)| = 2$, $|T_{n + 1}(\gamma)| = 4$ and $\eta(\pi(\gamma)) \neq 0$.
Then, in light of Lemmas \ref{lem:3-cycle_condition} (ii), \ref{lem:f-connected1}, \ref{lem:f-connected2}, \ref{lem:O6_f2_special} and \ref{lem:O6_f2_special_2-5}, we may assume that
\[
 \gamma_{1} + \gamma_{2} =  + \> (n, u; \omega_{0}, \omega_{1}, \omega_{0}) - (n, u; \omega_{1}, \omega_{0}, \omega_{1})
\]
with some $u \in O_{6}$.
Since $g(\gamma_{1} + \gamma_{2}) + f(\gamma_{3} + \gamma_{4} + \gamma_{5} + \gamma_{6}) = 0$ by Lemma \ref{lem:3-cycle_condition} (i), as we saw in the proof of Lemma \ref{lem:O6_f2_usual}, at least two of $\gamma_{3}, \gamma_{4}, \gamma_{5}$ and $\gamma_{6}$ have index $u^{\omega_{0}}$ (or $u^{\omega_{1}}$).
Therefore, in light of the formula (\ref{eq:O6_f2_usual}), we may assume that
\begin{align*}
 f(\gamma_{3} + \gamma_{4})
 & = - \> (n + 1, u^{\omega_{0}}; \omega_{1}, \omega_{0}) - (n + 1, u^{\omega_{0}}; \omega_{0}, \omega_{2}) - (n + 1, u^{\omega_{0}}; \omega_{2}, \omega_{1}), \\
 f(\gamma_{5} + \gamma_{6})
 & = + \> (n + 1, u^{\omega_{1}}; \omega_{5}, \omega_{0}) + (n + 1, u^{\omega_{1}}; \omega_{0}, \omega_{1}) + (n + 1, u^{\omega_{1}}; \omega_{1}, \omega_{5}).
\end{align*}

Consider the following four 3-terms
\begin{align*}
 \delta_{3} & = + (n + 1, u^{\omega_{0}}; \omega_{0}, \omega_{2}, \omega_{1}), &
 \delta_{4} & = + (n + 1, u^{\omega_{0}}; \omega_{0}, \omega_{1}, \omega_{0}), \\
 \delta_{5} & = - (n + 1, u^{\omega_{1}}; \omega_{0}, \omega_{1}, \omega_{5}), &
 \delta_{6} & = - (n + 1, u^{\omega_{1}}; \omega_{0}, \omega_{5}, \omega_{0}).
\end{align*}
Then, since $f(\delta_{3} + \delta_{4}) = f(\gamma_{3} + \gamma_{4})$ and $f(\delta_{5} + \delta_{6}) = f(\gamma_{5} + \gamma_{6})$, we have the following case:
\begin{itemize}
\item[(a)]
$\gamma_{3} + \gamma_{4} = \delta_{3} + \delta_{4}$,
\item[(b)]
$l(\gamma_{3} + \gamma_{4} - \delta_{3} - \delta_{4}) = l(\gamma_{i} - \delta_{j}) = 2$ and $\{ \gamma_{i}, - \delta_{j} \}$ is $f$-connected for some $3 \leq i, j \leq 4$, or
\item[(c)]
$\{ \gamma_{3}, \gamma_{4}, - \delta_{3}, - \delta_{4} \}$ is $f$-connected,
\end{itemize}
and
\begin{itemize}
\item[(d)]
$\gamma_{5} + \gamma_{6} = \delta_{5} + \delta_{6}$,
\item[(e)]
$l(\gamma_{5} + \gamma_{6} - \delta_{5} - \delta_{6}) = l(\gamma_{k} - \delta_{l}) = 2$ and $\{ \gamma_{k}, - \delta_{l} \}$ is $f$-connected for some $5 \leq k, l \leq 6$, or
\item[(f)]
$\{ \gamma_{5}, \gamma_{6}, - \delta_{5}, - \delta_{6} \}$ is $f$-connected.
\end{itemize}
In the cases (b) and (e), we respectively have
\[
 \gamma_{3} + \gamma_{4} = + \> (n + 1, u^{\omega_{0}}; \omega_{0}, \omega_{2}, \omega_{1}) + (n + 1, u^{\omega_{0}}; \omega_{1}, \omega_{0}, \omega_{1})
\]
and
\[
 \gamma_{5} + \gamma_{6} = - \> (n + 1, u^{\omega_{1}}; \omega_{0}, \omega_{1}, \omega_{5}) - (n + 1, u^{\omega_{1}}; \omega_{5}, \omega_{0}, \omega_{5})
\]
by Lemma \ref{lem:f-connected2}.
Furthermore, in the cases (c) and (f), we respectively have
\[
 \gamma_{3} + \gamma_{4} =
 \begin{cases}
  + \> (n + 1, u^{\omega_{0}}; \omega_{2}, \omega_{1}, \omega_{0}) + \langle n + 1, u^{\omega_{0}}; \omega_{0}, \omega_{2} \rangle \ \text{or} \\
  + \> (n + 1, u^{\omega_{0}}; \omega_{1}, \omega_{0}, \omega_{2}) + \langle n + 1, u^{\omega_{0}}; \omega_{1}, \omega_{2} \rangle
 \end{cases}
\]
and
\[
 \gamma_{5} + \gamma_{6} =
 \begin{cases}
  - \> (n + 1, u^{\omega_{1}}; \omega_{1}, \omega_{5}, \omega_{0}) - \langle n + 1, u^{\omega_{1}}; \omega_{0}, \omega_{1} \rangle \ \text{or} \\
  - \> (n + 1, u^{\omega_{1}}; \omega_{5}, \omega_{0}, \omega_{1}) - \langle n + 1, u^{\omega_{1}}; \omega_{1}, \omega_{5} \rangle
 \end{cases}
\]
by Lemma \ref{lem:f-connected4}.
Thus, in light of Lemma \ref{lem:3-cycle_condition} (iii), we eventually have
\begin{align*}
 & \gamma_{3} + \gamma_{4} + \gamma_{5} + \gamma_{6} = \\
 &
 \begin{cases}
  + \> (u^{\omega_{0}}; \omega_{0}, \omega_{2}, \omega_{1}) + (u^{\omega_{0}}; \omega_{1}, \omega_{0}, \omega_{1}) - (u^{\omega_{1}}; \omega_{5}, \omega_{0}, \omega_{1}) - (u^{\omega_{1}}; \omega_{5}, \omega_{1}, \omega_{5}) \ \text{or} \\
  + \> (u^{\omega_{0}}; \omega_{2}, \omega_{1}, \omega_{0}) + (u^{\omega_{0}}; \omega_{2}, \omega_{0}, \omega_{2}) - (u^{\omega_{1}}; \omega_{1}, \omega_{5}, \omega_{0}) - (u^{\omega_{1}}; \omega_{0}, \omega_{1}, \omega_{0}).
 \end{cases}
\end{align*}
In the former and latter cases, we respectively have $\gamma = \partial (n, u; \omega_{0}, \omega_{1}, \omega_{0}, \omega_{1})$ and $\gamma = - \partial (n, u; \omega_{1}, \omega_{0}, \omega_{1}, \omega_{0})$ contradicting to our assumption.
\end{proof}

\begin{proposition}
\label{prop:O6_VI}
There are no 3-cycles $\gamma$ of $(O_{6}, S)$ satisfying $l(\gamma) = 7$, $|T_{n}(\gamma)| = 2$, $|T_{n + 1}(\gamma)| = 5$ and $\eta(\pi(\gamma)) \neq 0$.
\end{proposition}

\begin{proof}
Let $\gamma_{1}, \gamma_{2}, \dots, \gamma_{7}$ be 3-terms of $(O_{6}, S)$ and $\gamma = \gamma_{1} + \gamma_{2} + \dots + \gamma_{7}$.
Assume that $\gamma$ is a 3-cycle satisfying $l(\gamma) = 7$, $|T_{n}(\gamma)| = 2$, $|T_{n + 1}(\gamma)| = 5$ and $\eta(\pi(\gamma)) \neq 0$.
Then, in a similar way to the proof of Proposition \ref{prop:O6_V}, we may assume that
\begin{align*}
 \gamma_{1} + \gamma_{2}
 & =  + \> (n, u; \omega_{0}, \omega_{1}, \omega_{0}) - (n, u; \omega_{1}, \omega_{0}, \omega_{1}), \\
 f(\gamma_{3} + \gamma_{4})
 & = - \> (n + 1, u^{\omega_{0}}; \omega_{1}, \omega_{0}) - (n + 1, u^{\omega_{0}}; \omega_{0}, \omega_{2}) - (n + 1, u^{\omega_{0}}; \omega_{2}, \omega_{1}), \\
 f(\gamma_{5} + \gamma_{6} + \gamma_{7})
 & = + \> (n + 1, u^{\omega_{1}}; \omega_{5}, \omega_{0}) + (n + 1, u^{\omega_{1}}; \omega_{0}, \omega_{1}) + (n + 1, u^{\omega_{1}}; \omega_{1}, \omega_{5})
\end{align*}
with some $u \in O_{6}$.
Furthermore, in a similar way to the proof of Proposition \ref{prop:O6_V}, we have
\[
 \gamma_{3} + \gamma_{4} =
 \begin{cases}
  + \> (n + 1, u^{\omega_{0}}; \omega_{0}, \omega_{2}, \omega_{1}) + \langle n + 1, u^{\omega_{0}}; \omega_{0}, \omega_{1} \rangle, \\
  + \> (n + 1, u^{\omega_{0}}; \omega_{2}, \omega_{1}, \omega_{0}) + \langle n + 1, u^{\omega_{0}}; \omega_{0}, \omega_{2} \rangle, \, \text{or} \\
  + \> (n + 1, u^{\omega_{0}}; \omega_{1}, \omega_{0}, \omega_{2}) + \langle n + 1, u^{\omega_{0}}; \omega_{1}, \omega_{2} \rangle.
 \end{cases}
\]

Let $N_{i}^{\varepsilon}$ denote the number of 3-terms among $\gamma_{3}, \gamma_{4}, \dots, \gamma_{7}$ whose types are $i$ and signs $\varepsilon$.
Then, since one of $\gamma_{3}$ and $\gamma_{4}$ is of type 2 and the other of type 1, we have $N_{1}^{+} \geq 1$ and $N_{2}^{+} \geq 1$.
Furthermore, since $g(\gamma_{1} + \gamma_{2}) + f(\gamma_{3} + \gamma_{4} + \gamma_{5} + \gamma_{6} + \gamma_{7}) = 0$ and $g(\gamma_{3} + \gamma_{4} + \gamma_{5} + \gamma_{6} + \gamma_{7}) = 0$ by Lemma \ref{lem:3-cycle_condition}, we respectively have
\[
 2 (N_{0}^{+} - N_{0}^{-}) + 2 (N_{1}^{+} - N_{1}^{-}) + (N_{2}^{+} - N_{2}^{-}) + (N_{3}^{+} - N_{3}^{-}) = 0
\]
and
\[
 2 (N_{0}^{+} - N_{0}^{-}) + (N_{1}^{+} - N_{1}^{-}) + 2 (N_{2}^{+} - N_{2}^{-}) + (N_{3}^{+} - N_{3}^{-}) = 0.
\]
Thus, remarking that $N_{0}^{+}, N_{0}^{-}, \dots, N_{3}^{-} \geq 0$ and $N_{0}^{+} + N_{0}^{-} + \dots + N_{3}^{-} = 5$, we have
\[
 (N_{0}^{+}, N_{0}^{-}, N_{1}^{+}, N_{1}^{-}, N_{2}^{+}, N_{2}^{-}, N_{3}^{+}, N_{3}^{-}) =
 \begin{cases}
  (0, 0, 1, 0, 1, 0, 0, 3) \ \text{or} \\
  (0, 2, 1, 0, 1, 0, 1, 0).
 \end{cases}
\]
In the latter case, since two of $\gamma_{5}, \gamma_{6}$ and $\gamma_{7}$ are negative and of type 0, and the other is not of type 0, reduced $f(\gamma_{5} + \gamma_{6} + \gamma_{7})$ has at least one 2-term of the form $+ (n + 1, u^{\omega_{1}}; a, [a + 3])$ for some $a \in O_{6}$.
Since $\omega_{0} \neq [\omega_{5} + 3]$, $\omega_{1} \neq [\omega_{0} + 3]$ and $\omega_{5} \neq [\omega_{1} + 3]$, it leads to a contradiction.
Therefore, we assume the former case in the remaining.

Let $\delta_{5}$ and $\delta_{6}$ be the 3-terms defined in the proof of Proposition \ref{prop:O6_V}.
Since $f(\delta_{5} + \delta_{6}) = f(\gamma_{5} + \gamma_{6} + \gamma_{7})$, in light of Lemma \ref{lem:f-connected1}, we have the following case:
\begin{itemize}
\item[(a)]
$l(\gamma_{5} + \gamma_{6} + \gamma_{7} - \delta_{5} - \delta_{6}) = l(\gamma_{i} + \gamma_{j} - \delta_{k}) = 3$ and $\{ \gamma_{i}, \gamma_{j}, - \delta_{k} \}$ is $f$-connected for some $5 \leq i < j \leq 7$ and $5 \leq k \leq 6$, or
\item[(b)]
$\{ \gamma_{5}, \gamma_{6}, \gamma_{7}, - \delta_{5}, - \delta_{6} \}$ is $f$-connected.
\end{itemize}

In the case (a), since $\delta_{6}$ and $\delta_{5}$ are respectively of type 1 and 2, at least one of $\gamma_{5}, \gamma_{6}$ and $\gamma_{7}$ is of type 1 or 2.
It contradicts to $(N_{1}^{+}, N_{1}^{-}, N_{2}^{+}, N_{2}^{-}) = (1, 0, 1, 0)$.

In the case (b), in light of Lemma \ref{lem:f-connected5}, we have
\begin{align*}
 & \gamma_{5} + \gamma_{6} + \gamma_{7} = \\
 &
 \begin{cases}
  - \> (n + 1, u^{\omega_{1}}; d, \omega_{0}, \omega_{1}) - (n + 1, u^{\omega_{1}}; d, \omega_{1}, \omega_{5}) - (n + 1, u^{\omega_{1}}; d, \omega_{5}, \omega_{0}) \ \text{or} \\
  - \> (n + 1, u^{\omega_{1}}; \omega_{0}, \omega_{1}, d) - (n + 1, u^{\omega_{1}}; \omega_{1}, \omega_{5}, d) - (n + 1, u^{\omega_{1}}; \omega_{5}, \omega_{0}, d)
 \end{cases}
\end{align*}
with some $d \in O_{6}$ ($d \neq \omega_{0}, \omega_{1}, \omega_{5}$).
In each case, it is routine to check that we have $g(\gamma_{3} + \gamma_{4} + \gamma_{5} + \gamma_{6} + \gamma_{7}) \neq 0$ contradicting to Lemma \ref{lem:3-cycle_condition} (iii).
\end{proof}

\begin{proposition}
\label{prop:O6_VII}
There are no 3-cycles $\gamma$ of $(O_{6}, S)$ satisfying $l(\gamma) = 6$, $|T_{n}(\gamma)| = 3$, $|T_{n + 1}(\gamma)| = 3$ and $\eta(\pi(\gamma)) \neq 0$.
\end{proposition}

\begin{proof}
Let $\gamma_{1}, \gamma_{2}, \dots, \gamma_{6}$ be 3-terms of $(O_{6}, S)$ and $\gamma = \gamma_{1} + \gamma_{2} + \dots + \gamma_{6}$.
Assume that $\gamma$ is a 3-cycle satisfying $l(\gamma) = 6$, $|T_{n}(\gamma)| = 3$, $|T_{n + 1}(\gamma)| = 3$ and $\eta(\pi(\gamma)) \neq 0$.
Then, in light of Lemmas \ref{lem:3-cycle_condition} (ii), \ref{lem:f-connected1}, \ref{lem:f-connected3} and \ref{lem:O6_f3_special}, we may assume that
\begin{itemize}
\item[(i)]
$\gamma_{1} + \gamma_{2} + \gamma_{3} =  + \> (n, u; \omega_{0}, \omega_{1}, \omega_{0}) - (n, u; c, \omega_{0}, \omega_{1}) - (n, u; c, \omega_{1}, \omega_{0})$ or
\item[(ii)]
$\gamma_{1} + \gamma_{2} + \gamma_{3} =  + \> (n, u; \omega_{0}, \omega_{1}, \omega_{0}) - (n, u; \omega_{0}, \omega_{1}, c) - (n, u; \omega_{1}, \omega_{0}, c)$
\end{itemize}
up to sign, with some $u, c \in O_{6}$ ($c \neq \omega_{0}, \omega_{1}$).
In light of Lemmas \ref{lem:3-cycle_condition} (iii), \ref{lem:f-/g-connectedness} (i) and \ref{lem:f-connected1}, $\gamma_{4}, \gamma_{5}$ and $\gamma_{6}$ are $g$-connected.
Therefore, in light of Lemmas \ref{lem:reverse} (ii), \ref{lem:f-/g-connectedness} (i) and \ref{lem:f-connected3}, one of $\gamma_{4}, \gamma_{5}$ and $\gamma_{6}$ is of type 0 or 2.

In the case (i), as we saw in the proof of Lemma \ref{lem:O6_f3_usual}, we have
\begin{itemize}
\item
$c = \omega_{5}$, or
\item
$c = \omega_{3}$ and $u = \omega_{0}$ or $\omega_{3}$.
\end{itemize}
In the former case, we have
\begin{align*}
 g(\gamma_{1} + \gamma_{2} + \gamma_{3})
 & = + \> (u^{\omega_{0}}; \omega_{1}, \omega_{0}) + (u^{\omega_{0}}; \omega_{0}, \omega_{2}) - (u^{\omega_{0}}; \omega_{1}, \omega_{2}) \\
 & \phantom{=} \ + (u^{\omega_{1}}; \omega_{3}, \omega_{0}) - (u^{\omega_{1}}; \omega_{5}, \omega_{0}) - (u^{\omega_{1}}; \omega_{3}, \omega_{5}) \\
 & \phantom{=} \ - (u^{\omega_{5}}; \omega_{0}, \omega_{1}) - (u^{\omega_{5}}; \omega_{1}, \omega_{0}).
\end{align*}
Since $g(\gamma_{1} + \gamma_{2} + \gamma_{3}) + f(\gamma_{4} + \gamma_{5} + \gamma_{6}) = 0$ by Lemma \ref{lem:3-cycle_condition} (i), and $u^{\omega_{0}}, u^{\omega_{1}}$ and $u^{\omega_{5}}$ are mutually different, we immediately have
\begin{align*}
 & \gamma_{4} + \gamma_{5} + \gamma_{6} = \\
 & + (n + 1, u^{\omega_{0}}; \omega_{1}, \omega_{0}, \omega_{2}) - (n + 1, u^{\omega_{1}}; \omega_{3}, \omega_{5}, \omega_{0}) - \langle n + 1, u^{\omega_{5}}; \omega_{0}, \omega_{1} \rangle.
\end{align*}
If $\langle n + 1, u^{\omega_{5}}; \omega_{0}, \omega_{1} \rangle = (n + 1, u^{\omega_{5}}; \omega_{1}, \omega_{0}, \omega_{1})$, we obviously have $g(\gamma_{4} + \gamma_{5} + \gamma_{6}) \neq 0$ contradicting to Lemma \ref{lem:3-cycle_condition} (iii).
If $\langle n + 1, u^{\omega_{5}}; \omega_{0}, \omega_{1} \rangle = (n + 1, u^{\omega_{5}}; \omega_{0}, \omega_{1}, \omega_{0})$, we have $\gamma = - \partial (n, u; \omega_{5}, \omega_{0}, \omega_{1}, \omega_{0})$ contradicting to our assumption.

In the latter case, since $g(\gamma_{1} + \gamma_{2} + \gamma_{3}) + f(\gamma_{4} + \gamma_{5} + \gamma_{6}) = 0$, in light of the formula (\ref{eq:O6_f3_usual_i_a}), we may assume that
\begin{align*}
 f(\gamma_{4} + \gamma_{5}) & = - \> (u; \omega_{0}, \omega_{2}) - (u; \omega_{3}, \omega_{1}) + (u; \omega_{0}, \omega_{1}) + (u; \omega_{3}, \omega_{2}), \\
 \gamma_{6} & = - \> (n + 1; u^{\omega_{1}}; \omega_{2}, \omega_{5}, \omega_{0}).
\end{align*}
We note that $\gamma_{6}$ is of type 3, and thus one of $\gamma_{4}$ and $\gamma_{5}$ of type 0 or 2.
Consider the following two 3-terms
\begin{align*}
 \delta_{4} & = + (n + 1, u; \omega_{0}, \omega_{2}, \omega_{1}), &
 \delta_{5} & = - (n + 1, u; \omega_{3}, \omega_{2}, \omega_{1}).
\end{align*}
Then, since $f(\delta_{4} + \delta_{5}) = f(\gamma_{4} + \gamma_{5})$, we have the following case:
\begin{itemize}
\item[(a)]
$\gamma_{4} + \gamma_{5} = \delta_{4} + \delta_{5}$,
\item[(b)]
$l(\gamma_{4} + \gamma_{5} - \delta_{4} - \delta_{5}) = l(\gamma_{i} - \delta_{j}) = 2$ and $\{ \gamma_{i}, - \delta_{j} \}$ is $f$-connected for some $4 \leq i, j \leq 5$, or
\item[(c)]
$\{ \gamma_{4}, \gamma_{5}, - \delta_{4}, - \delta_{5} \}$ is $f$-connected.
\end{itemize}
In the case (a), we have $g(\gamma_{4} + \gamma_{5} + \gamma_{6}) \neq 0$ contradicting to Lemma \ref{lem:3-cycle_condition} (iii).
In the case (b), since both of $\delta_{4}$ and $\delta_{5}$ are not of type 0 or 1, we have no candidates of $\gamma_{4} + \gamma_{5}$ by Lemma \ref{lem:f-connected2}.
In the case (c), in light of Lemma \ref{lem:f-connected4}, we have
\[
 \gamma_{4} + \gamma_{5} = + \> (n + 1, u; \omega_{3}, \omega_{1}, \omega_{2}) - (n + 1, u; \omega_{0}, \omega_{1}, \omega_{2}).
\]
We thus have $g(\gamma_{4} + \gamma_{5} + \gamma_{6}) \neq 0$ contradicting to Lemma \ref{lem:3-cycle_condition} (iii).

In the case (ii), as we saw in the proof of Lemma \ref{lem:O6_f3_usual}, we have
\begin{itemize}
\item
$c = \omega_{2}$, or
\item
$c = \omega_{3}$ and $u = \omega_{0}$ or $\omega_{3}$.
\end{itemize}
In the former case, we have
\begin{align*}
 g(\gamma_{1} + \gamma_{2} + \gamma_{3})
 & = + \> (u^{\omega_{0}}; \omega_{1}, \omega_{0}) + (u^{\omega_{0}}; \omega_{0}, \omega_{2}) - (u^{\omega_{0}}; \omega_{1}, \omega_{2}) \\
 & \phantom{=} \ + (u^{\omega_{1}}; \omega_{5}, \omega_{2}) - (u^{\omega_{1}}; \omega_{5}, \omega_{0}) - (u^{\omega_{1}}; \omega_{0}, \omega_{2}) \\
 & \phantom{=} \ - (u^{\omega_{2}}; \omega_{1}, \omega_{3}) - (u^{\omega_{2}}; \omega_{3}, \omega_{1}).
\end{align*}
Since $g(\gamma_{1} + \gamma_{2} + \gamma_{3}) + f(\gamma_{4} + \gamma_{5} + \gamma_{6}) = 0$, and $u^{\omega_{0}}, u^{\omega_{1}}$ and $u^{\omega_{2}}$ are mutually different, we immediately have
\begin{align*}
 & \gamma_{4} + \gamma_{5} + \gamma_{6} = \\
 & + (n + 1, u^{\omega_{0}}; \omega_{1}, \omega_{0}, \omega_{2}) - (n + 1, u^{\omega_{1}}; \omega_{5}, \omega_{0}, \omega_{2}) - \langle n + 1, u^{\omega_{2}}; \omega_{1}, \omega_{3} \rangle.
\end{align*}
If $\langle n + 1, u^{\omega_{2}}; \omega_{1}, \omega_{3} \rangle = (n + 1, u^{\omega_{2}}; \omega_{3}, \omega_{1}, \omega_{3})$, we obviously have $g(\gamma_{4} + \gamma_{5} + \gamma_{6}) \neq 0$ contradicting to Lemma \ref{lem:3-cycle_condition} (iii).
If $\langle n + 1, u^{\omega_{2}}; \omega_{1}, \omega_{3} \rangle = (n + 1, u^{\omega_{2}}; \omega_{1}, \omega_{3}, \omega_{1})$, we have $\gamma = \partial (n, u; \omega_{0}, \omega_{1}, \omega_{0}, \omega_{2})$ contradicting to our assumption.

In the latter case, since $g(\gamma_{1} + \gamma_{2} + \gamma_{3}) + f(\gamma_{4} + \gamma_{5} + \gamma_{6}) = 0$, in light of the formula (\ref{eq:O6_f3_usual_ii_a}), we may assume that
\begin{align*}
 f(\gamma_{4} + \gamma_{5})
 & = - \> (u; \omega_{1}, \omega_{0}) - (u; \omega_{0}, \omega_{2}) - (u; \omega_{2}, \omega_{3}) \\
 & \phantom{=} \ + (u; \omega_{1}, \omega_{3}) + (u; \omega_{0}, \omega_{5}) + (u; \omega_{5}, \omega_{0}).
\end{align*}
It is easy to see that we have no candidates of $\gamma_{4} + \gamma_{5}$ satisfying this formula.
\end{proof}

\begin{proposition}
\label{prop:O6_VIII}
There are no 3-cycles $\gamma$ of $(O_{6}, S)$ satisfying $l(\gamma) = 7$, $|T_{n}(\gamma)| = 3$, $|T_{n + 1}(\gamma)| = 4$ and $\eta(\pi(\gamma)) \neq 0$.
\end{proposition}

Although the proof of Proposition \ref{prop:O6_VIII} is straightforward, its argument is slightly complicated.
We thus defer the proof of this proposition to Section \ref{sec:proof_of_proposition_O6_VIII}.
We are now ready to prove Theorem \ref{thm:lower_bound_O6}.

\begin{proof}[Proof of Theorem \ref{thm:lower_bound_O6}]
It is routine to see that a 3-chain
\begin{align*}
 \gamma
 & = - \> (0, 0; 0, 2, 1) + (0, 0; 1, 5, 0) - (0, 0; 0, 1, 5) + (0, 0; 5, 4, 0) \\
 & \phantom{=} \ - (0, 0; 0, 5, 4) + (0, 0; 4, 2, 0) - (0, 0; 0, 4, 2) + (0, 0; 2, 1, 0)
\end{align*}
of $(O_{6}, \mathbb{Z} \times O_{6})$ is a 3-cycle satisfying $\eta(\pi(\gamma)) = 1$.
We thus have $l(\eta, \mathbb{Z} \times O_{6}) \neq 0$.
Then, in light of Propositions \ref{prop:O6_list}, \ref{prop:O6_I}--\ref{prop:O6_VIII} and Lemma \ref{lem:reverse} (iv), we immediately have the claim.
\end{proof}

\section{Proof of Lemma \ref{lem:O6_f2_special_2-5}}
\label{sec:proof_of_lemma_O6_f2_special_2-5}

The aim of this section is to prove Lemma \ref{lem:O6_f2_special_2-5}.
We start with preparing the following lemmas.

\begin{lemma}
\label{lem:O6_ab_ba_1}
Let $\gamma$ be a 3-term of $(O_{6}, S)$ having degree $n$ and index $u$.
Assume that $f(\gamma) = + \> (a, b) + (b, a)$ for some distinct $a, b \in O_{6}$.
Then, we have $\gamma = - \> \langle a, b \rangle$.
\end{lemma}

\begin{proof}
We obviously have the claim.
\end{proof}

\begin{lemma}
\label{lem:O6_ab_ba_2}
Let $\gamma_{1}$ and $\gamma_{2}$ be 3-terms of $(O_{6}, S)$ having the same degree $n$ and the same index $u$, and $\gamma = \gamma_{1} + \gamma_{2}$.
Assume that $f(\gamma) = + \> (a, b) + (b, a)$ for some distinct $a, b \in O_{6}$.
Then, we have the following two cases with some $c \in O_{6}$ {\upshape (}$c \neq a, b${\upshape )}.
\begin{itemize}
\item[(i)]
$\gamma = - \> (c, a, b) - (c, b, a)$.
\item[(ii)]
$\gamma = - \> (a, b, c) - (b, a, c)$.
\end{itemize}
\end{lemma}

\begin{proof}
If $\gamma_{2} = - \> \gamma_{1}$ or $\gamma_{i} = - \> \langle n, u; a, b \rangle$ for some $i$, then we have $f(\gamma_{1} + \gamma_{2}) \neq  + \> (a, b) + (b, a)$ contradicting to our assumption.
Therefore, $\gamma_{1}, \gamma_{2}$ and $+ \langle n, u; a, b \rangle$ are efficient.
Since they are also $f$-connected by Lemma \ref{lem:f-connected1}, in light of Lemma \ref{lem:f-connected3}, we have the claim.
\end{proof}

\begin{lemma}
\label{lem:O6_ab_ba_3}
Let $\gamma_{1}, \gamma_{2}$ and $\gamma_{3}$ be 3-terms of $(O_{6}, S)$ having the same degree $n$ and the same index $u$, and $\gamma = \gamma_{1} + \gamma_{2} + \gamma_{3}$.
Assume that $f(\gamma) = + \> (a, b) + (b, a)$ for some distinct $a, b \in O_{6}$, and $f \left( \sum_{i \in I} \gamma_{i} \right) \neq  + \> (a, b) + (b, a)$ for any non-empty proper subset $I$ of $\{ 1, 2, 3 \}$.
Then, we have the following four cases with some $c \in O_{6}$ {\upshape (}$c \neq a, b${\upshape )}.
\begin{itemize}
\item[(i)]
$\gamma = + \> (a, c, b) - (b, a, c) - \langle c, b \rangle$.
\item[(ii)]
$\gamma = + \> (b, c, a) - (a, b, c) - \langle c, a \rangle$.
\item[(iii)]
$\gamma = - \> (c, a, b) - \langle c, b \rangle + (b, c, a)$
\item[(iv)]
$\gamma = - \> (c, b, a) - \langle c, a \rangle + (a, c, b)$
\end{itemize}
\end{lemma}

\begin{proof}
Since $\gamma_{1}, \gamma_{2}, \gamma_{3}$ and $+ \langle n, u; a, b \rangle$ are $f$-connected by the assumption, in light of Lemma \ref{lem:f-connected4}, we have the claim.
\end{proof}

\begin{lemma}
\label{lem:O6_ab_ba_4}
Let $\gamma_{1}, \gamma_{2}, \gamma_{3}$ and $\gamma_{4}$ be 3-terms of $(O_{6}, S)$ having the same degree $n$ and the same index $u$, and $\gamma = \gamma_{1} + \gamma_{2} + \gamma_{3} + \gamma_{4}$.
Assume that $f(\gamma) = + \> (a, b) + (b, a)$ for some distinct $a, b \in O_{6}$, and $f \left( \sum_{i \in I} \gamma_{i} \right) \neq  + \> (a, b) + (b, a)$ for any non-empty proper subset $I$ of $\{ 1, 2, 3, 4 \}$.
Then, we have the following twenty-one cases with some $c \in O_{6}$ {\upshape (}$c \neq a, b${\upshape )} and $d \in O_{6}$.
\begin{itemize}
\item[(i)]
$\gamma = - \> \langle a, c \rangle - \langle b, c \rangle + (a, c, b) + (b, c, a)$
\item[(ii)]
$\gamma = + \> \langle b, c \rangle - (a, b, c) - (a, c, a) - (c, b, a)$
\item[(iii)]
$\gamma = + \> \langle a, c \rangle - (b, a, c) - (b, c, b) - (c, a, b)$
\item[(iv)]
$\gamma = + \> (d, a, c) - (b, a, c) - (d, a, b) - (d, b, c)$ \enskip {\upshape (}$d \neq a, b, c${\upshape )}.
\item[(v)]
$\gamma = + \> (d, b, c) - (a, b, c) - (d, a, c) - (d, b, a)$ \enskip {\upshape (}$d \neq a, b, c${\upshape )}.
\item[(vi)]
$\gamma = + \> (a, d, c) - (a, b, c) - (b, a, d) - (b, d, c)$ \enskip {\upshape (}$d \neq a, b, c${\upshape )}.
\item[(vii)]
$\gamma = + \> (b, d, c) - (a, b, d) - (a, d, c) - (b, a, c)$ \enskip {\upshape (}$d \neq a, b, c${\upshape )}.
\item[(viii)]
$\gamma = + \> (c, d, a) - (c, b, a) - (c, d, b) - (d, a, b)$ \enskip {\upshape (}$d \neq a, b, c${\upshape )}.
\item[(ix)]
$\gamma = + \> (d, c, b) - (c, b, a) - (d, a, b) - (d, c, a)$ \enskip {\upshape (}$d \neq a, b, c${\upshape )}.
\item[(x)]
$\gamma = + \> (a, c, b) - (d, a, c) - (d, b, a) - (d, c, b)$ \enskip {\upshape (}$d \neq a, b, c${\upshape )}.
\item[(xi)]
$\gamma = + \> (b, c, a) - (d, a, b) - (d, b, c) - (d, c, a)$ \enskip {\upshape (}$d \neq a, b, c${\upshape )}.
\item[(xii)]
$\gamma = + \> (a, c, b) - (a, c, d) - (b, a, d) - (c, b, d)$ \enskip {\upshape (}$d \neq a, b, c${\upshape )}.
\item[(xiii)]
$\gamma = + \> (b, c, a) - (a, b, d) - (b, c, d) - (c, a, d)$ \enskip {\upshape (}$d \neq a, b, c${\upshape )}.
\item[(xiv)]
$\gamma = + \> (a, c, b) - (b, a, c) - (d, b, c) - (d, c, b)$ \enskip {\upshape (}$d \neq a, b, c${\upshape )}.
\item[(xv)]
$\gamma = + \> (b, c, a) - (a, b, c) - (d, a, c) - (d, c, a)$ \enskip {\upshape (}$d \neq a, b, c${\upshape )}.
\item[(xvi)]
$\gamma = + \> (a, c, b) - (a, c, d) - (c, a, d) - (c, b, a)$ \enskip {\upshape (}$d \neq a, b, c${\upshape )}.
\item[(xvii)]
$\gamma = + \> (b, c, a) - (b, c, d) - (c, a, b) - (c, b, d)$ \enskip {\upshape (}$d \neq a, b, c${\upshape )}.
\item[(xviii)]
$\gamma = + \> (a, c, b) - (c, b, a) - (d, a, c) - (d, c, a)$ \enskip {\upshape (}$d \neq a, c${\upshape )}.
\item[(xix)]
$\gamma = + \> (b, c, a) - (c, a, b) - (d, b, c) - (d, c, b)$ \enskip {\upshape (}$d \neq b, c${\upshape )}.
\item[(xx)]
$\gamma = + \> (b, c, a) - (a, b, c) - (a, c, d) - (c, a, d)$ \enskip {\upshape (}$d \neq a, c${\upshape )}.
\item[(xxi)]
$\gamma = + \> (a, c, b) - (b, a, c) - (b, c, d) - (c, b, d)$ \enskip {\upshape (}$d \neq b, c${\upshape )}.
\end{itemize}
\end{lemma}

\begin{proof}
Since $\gamma_{1}, \gamma_{2}, \gamma_{3}, \gamma_{4}$ and $+ \langle n, u; a, b \rangle$ are $f$-connected by the assumption, in light of Lemma \ref{lem:f-connected5}, we have the claim.
\end{proof}

Let $\gamma$ be the 3-chain in Lemma \ref{lem:f-connected2} with $X = O_{6}$.
Assume that $b = [a + 3]$ and $u \neq a, b$.
Then, as mentioned in the proof of Lemma \ref{lem:O6_f2_special}, we may assume that $a = \omega_{0}$ and $u = \omega_{1}$, and have $b = \omega_{3}$.
To achieve our goal, we show the following lemmas.

\begin{lemma}
\label{lem:O6_f2_special_geq2}
Let $\delta_{3}, \delta_{4}, \dots, \delta_{k + 2}$ be 3-terms of $(O_{6}, S)$ whose degrees are $n + 1$, and $\delta = \sum_{i = 3}^{k + 2} \delta_{i}$.
Assume that $g(\gamma) + f(\delta) = 0$.
Then, we have $k \geq 2$.
\end{lemma}

\begin{proof}
Since $g(\gamma) + f(\delta) = 0$ and reduced $g(\gamma)$ has two 2-terms of index $\omega_{2}$ (or $\omega_{5}$) (see the formula (\ref{eq:O6_f2_special})), at least one of $\delta_{3}, \delta_{4}, \dots, \delta_{k + 2}$ has index $\omega_{2}$ (or $\omega_{5}$).
We thus have $k \geq 2$.
\end{proof}

\begin{lemma}
\label{lem:O6_f2_special_2}
Let $\delta_{3}$ and $\delta_{4}$ be 3-terms of $(O_{6}, S)$ whose degrees are $n + 1$, and $\delta = \delta_{3} + \delta_{4}$.
Assume that $g(\gamma) + f(\delta) = 0$.
Then, we have the following cases.
\begin{itemize}
\item[(i)]
$\delta = + \> (n + 1, \omega_{2}; \omega_{0}, \omega_{3}, \omega_{0}) - (n + 1, \omega_{5}; \omega_{0}, \omega_{3}, \omega_{0})$.
\item[(ii)]
$\delta = + \> (n + 1, \omega_{2}; \omega_{0}, \omega_{3}, \omega_{0}) - (n + 1, \omega_{5}; \omega_{3}, \omega_{0}, \omega_{3})$.
\item[(iii)]
$\delta = + \> (n + 1, \omega_{2}; \omega_{3}, \omega_{0}, \omega_{3}) - (n + 1, \omega_{5}; \omega_{0}, \omega_{3}, \omega_{0})$.
\item[(iv)]
$\delta = + \> (n + 1, \omega_{2}; \omega_{3}, \omega_{0}, \omega_{3}) - (n + 1, \omega_{5}; \omega_{3}, \omega_{0}, \omega_{3})$.
\end{itemize}
In the case {\upshape (i)} or {\upshape (iv)}, we have $g(\delta) \neq 0$.
In the case {\upshape (ii)}, $\gamma + \delta$ is a 3-cycle of length four satisfying $\eta(\pi(\gamma + \delta)) = 0$.
In the case {\upshape (iii)}, $\gamma + \delta$ is a 3-boundary of length four.
\end{lemma}

\begin{proof}
Since $g(\gamma) + f(\delta) = 0$, in light of the formula (\ref{eq:O6_f2_special}), we may assume that
\begin{align*}
 f(\delta_{3}) & = - \> (n + 1, \omega_{2}; \omega_{0}, \omega_{3}) - (n + 1, \omega_{2}; \omega_{3}, \omega_{0}), \\
 f(\delta_{4}) & = + \> (n + 1, \omega_{5}; \omega_{0}, \omega_{3}) + (n + 1, \omega_{5}; \omega_{3}, \omega_{0}).
\end{align*}
Therefore, in light of Lemma \ref{lem:O6_ab_ba_1}, we have the cases (i)--(iv).
Obviously, we have $l(\gamma + \delta) = 4$.

In the case (i) or (iv), we have
\begin{align}
 g(\delta)
 & = \mp \> (n + 2, \omega_{1}; \omega_{0}, \omega_{3}) \mp (n + 2, \omega_{1}; \omega_{3}, \omega_{0}) \notag \\
 & \phantom{=} \ \pm (n + 2, \omega_{4}; \omega_{0}, \omega_{3}) \pm (n + 2, \omega_{4}; \omega_{3}, \omega_{0}), \label{eq:O6_f2_special_2}
\end{align}
where the upper one occurs in the case (i) and the lower one in the case (iv).
Obviously, we have $g(\delta) \neq 0$.

In the case (ii), since $g(\delta) = 0$, $\gamma + \delta$ is a 3-cycle.
Then, in light of Lemma \ref{lem:O6_weight}, we have $\eta(\pi(\gamma + \delta)) = 0$.

In the case (iii), we have $\gamma + \delta = \partial (n, \omega_{1}; \omega_{0}, \omega_{3}, \omega_{0}, \omega_{3})$.
\end{proof}

\begin{lemma}
\label{lem:O6_f2_special_2_geq2}
Let $\delta$ be the 3-chain {\upshape (i)} or {\upshape (iv)} in Lemma \ref{lem:O6_f2_special_2}.
Furthermore, let $\lambda_{5}, \lambda_{6}, \dots, \lambda_{k + 4}$ be 3-terms of $(O_{6}, S)$ whose degrees are $n + 2$, and $\lambda = \sum_{i = 5}^{k + 4} \lambda_{i}$.
Assume that $g(\delta) + f(\lambda) = 0$.
Then, we have $k \geq 2$.
\end{lemma}

\begin{proof}
In a similar way to the proof of Lemma \ref{lem:O6_f2_special_geq2}, in light of the formula (\ref{eq:O6_f2_special_2}), we have the claim.
\end{proof}

\begin{lemma}
\label{lem:O6_f2_special_2_2}
Let $\delta$ be the 3-chain {\upshape (i)} or {\upshape (iv)} in Lemma \ref{lem:O6_f2_special_2}.
Furthermore, let $\lambda_{5}$ and $\lambda_{6}$ be 3-terms of $(O_{6}, S)$ whose degrees are $n + 2$, and $\lambda = \lambda_{5} + \lambda_{6}$.
Assume that $g(\delta) + f(\lambda) = 0$.
Then, we have the following cases.
\begin{itemize}
\item[(I)]
$\lambda = \mp \> (n + 2, \omega_{1}; \omega_{0}, \omega_{3}, \omega_{0}) \pm (n + 2, \omega_{4}; \omega_{0}, \omega_{3}, \omega_{0})$.
\item[(II)]
$\lambda = \mp \> (n + 2, \omega_{1}; \omega_{0}, \omega_{3}, \omega_{0}) \pm (n + 2, \omega_{4}; \omega_{3}, \omega_{0}, \omega_{3})$.
\item[(III)]
$\lambda = \mp \> (n + 2, \omega_{1}; \omega_{3}, \omega_{0}, \omega_{3}) \pm (n + 2, \omega_{4}; \omega_{0}, \omega_{3}, \omega_{0})$.
\item[(IV)]
$\lambda = \mp \> (n + 2, \omega_{1}; \omega_{3}, \omega_{0}, \omega_{3}) \pm (n + 2, \omega_{4}; \omega_{3}, \omega_{0}, \omega_{3})$.
\end{itemize}
Here, the upper ones are for the 3-chain {\upshape (i)} and the lower ones for the 3-chain {\upshape (iv)}.
In the case {\upshape (I)} or {\upshape (IV)}, we have $g(\lambda) \neq 0$.
In the case {\upshape (II)} or {\upshape (III)}, $\gamma + \delta + \lambda$ is a 3-cycle of length six satisfying $\eta(\pi(\gamma + \delta + \lambda)) = 0$.
\end{lemma}

\begin{proof}
Since $g(\delta) + f(\lambda) = 0$, in light of the formula (\ref{eq:O6_f2_special_2}), we may assume that
\begin{align*}
 f(\lambda_{5}) & = \pm \> (n + 2, \omega_{1}; \omega_{0}, \omega_{3}) \pm (n + 2, \omega_{1}; \omega_{3}, \omega_{0}), \\
 f(\lambda_{6}) & = \mp (n + 2, \omega_{4}; \omega_{0}, \omega_{3}) \mp (n + 2, \omega_{4}; \omega_{3}, \omega_{0}).
\end{align*}
Therefore, in light of Lemma \ref{lem:O6_ab_ba_1}, we have the cases (I)--(IV).
Obviously, we have $l(\gamma + \delta + \lambda) = 6$.

In the case (I) or (IV), we have
\begin{align}
 g(\lambda)
 & = \mp \> (n + 3, \omega_{2}; \omega_{0}, \omega_{3}) \mp (n + 3, \omega_{2}; \omega_{3}, \omega_{0}) \notag \\
 & \phantom{=} \ \pm (n + 3, \omega_{5}; \omega_{0}, \omega_{3}) \pm (n + 3, \omega_{5}; \omega_{3}, \omega_{0}), \label{eq:O6_f2_special_2_2}
\end{align}
where the upper one occurs in the case (i)-(I) or (iv)-(IV) and the lower one in the case (i)-(IV) or (iv)-(I).
Obviously, we have $g(\lambda) \neq 0$.

In the case (II) or (III), since $g(\lambda) = 0$, $\gamma + \delta + \lambda$ is a 3-cycle.
Then, in light of Lemma \ref{lem:O6_weight}, we have $\eta(\pi(\gamma + \delta + \lambda)) = 0$.
\end{proof}

\begin{lemma}
\label{lem:O6_f2_special_2_2_geq2}
Let $\lambda$ be the 3-chain {\upshape (I)} or {\upshape (IV)} in Lemma \ref{lem:O6_f2_special_2_2}.
Furthermore, let $\mu_{7}, \mu_{8}, \dots, \mu_{k + 6}$ be 3-terms of $(O_{6}, S)$ whose degrees are $n + 3$, and $\mu = \sum_{i = 7}^{k + 6} \mu_{i}$.
Assume that $g(\lambda) + f(\mu) = 0$.
Then, we have $k \geq 2$.
\end{lemma}

\begin{proof}
In a similar way to the proof of Lemma \ref{lem:O6_f2_special_geq2}, in light of the formula (\ref{eq:O6_f2_special_2_2}), we have the claim.
\end{proof}

\begin{lemma}
\label{lem:O6_f2_special_2_3}
Let $\delta$ be the 3-chain {\upshape (i)} or {\upshape (iv)} in Lemma \ref{lem:O6_f2_special_2}.
Furthermore, let $\lambda_{5}, \lambda_{6}, \lambda_{7}$ be 3-terms of $(O_{6}, S)$ whose degrees are $n + 2$, and $\lambda = \lambda_{5} + \lambda_{6} + \lambda_{7}$.
Assume that $g(\delta) + f(\lambda) = 0$.
Then, we have $g(\lambda) \neq 0$.
\end{lemma}

\begin{proof}
Let $\varepsilon = \pm 1$ (or $\mp 1$), $\omega_{p} = \omega_{1}$ (or $\omega_{4}$), and $\omega_{q} = \omega_{4}$ (or $\omega_{1}$).
Then, since $g(\delta) + f(\lambda) = 0$, in light of the formula (\ref{eq:O6_f2_special_2}), we may assume that
\begin{align*}
 f(\lambda_{5}) & = + \> \varepsilon (n + 2, \omega_{p}; \omega_{0}, \omega_{3}) + \varepsilon (n + 2, \omega_{p}; \omega_{3}, \omega_{0}), \\
 f(\lambda_{6} + \lambda_{7}) & = - \> \varepsilon (n + 2, \omega_{q}; \omega_{0}, \omega_{3}) - \varepsilon (n + 2, \omega_{q}; \omega_{3}, \omega_{0}).
\end{align*}
Therefore, in light of Lemmas \ref{lem:O6_ab_ba_1} and \ref{lem:O6_ab_ba_2}, we have the following cases with some $a \in O_{6}$ ($a \neq \omega_{0}, \omega_{3}$).
\begin{itemize}
\item
$\lambda = \mp \> \langle n + 2, \omega_{1}; \omega_{0}, \omega_{3} \rangle \pm (n + 2, \omega_{4}; a, \omega_{0}, \omega_{3}) \pm (n + 2, \omega_{4}; a, \omega_{3}, \omega_{0})$.
\item
$\lambda = \mp \> \langle n + 2, \omega_{1}; \omega_{0}, \omega_{3} \rangle \pm (n + 2, \omega_{4}; \omega_{0}, \omega_{3}, a) \pm (n + 2, \omega_{4}; \omega_{3}, \omega_{0}, a)$.
\item
$\lambda = \pm \> \langle n + 2, \omega_{4}; \omega_{0}, \omega_{3} \rangle \mp (n + 2, \omega_{1}; a, \omega_{0}, \omega_{3}) \mp (n + 2, \omega_{1}; a, \omega_{3}, \omega_{0})$.
\item
$\lambda = \pm \> \langle n + 2, \omega_{4}; \omega_{0}, \omega_{3} \rangle \mp (n + 2, \omega_{1}; \omega_{0}, \omega_{3}, a) \mp (n + 2, \omega_{1}; \omega_{3}, \omega_{0}, a)$.
\end{itemize}
It is routine to see that we have $g(\lambda) \neq 0$ in each case.
\end{proof}

\begin{lemma}
\label{lem:O6_f2_special_2+}
Let $\delta$ be the 3-chain {\upshape (i)} or {\upshape (iv)} in Lemma \ref{lem:O6_f2_special_2}.
Furthermore, let $\lambda_{5}, \lambda_{6}, \dots, \lambda_{k + 4}$ be 3-terms of $(O_{6}, S)$ whose degrees are $n + 1$, and $\lambda = \sum_{i = 5}^{k + 4} \lambda_{i}$.
Assume that $\gamma + \delta + \lambda$ is a 3-cycle of length $k + 4$.
Then, we have $k \geq 4$.
\end{lemma}

\begin{proof}
Since $f(\lambda) = 0$ by Lemma \ref{lem:3-cycle_condition} (i), we have $k \geq 2$ by Lemma \ref{lem:f-connected1}.

Assume that $k = 2$.
Then, in light of Lemma \ref{lem:f-connected2}, we have
\[
 \lambda = \pm \> (v; a, b, a) \mp (v; b, a, b)
\]
with some $v \in O_{6}$ and distinct $a, b \in O_{6}$.
If $b = [a + 3]$ and $v = a$ or $b$, in light of Lemma \ref{lem:O6_f2_special}, we have $g(\lambda) = 0$.
Therefore, we have $g(\delta + \lambda) \neq 0$ contradicting to Lemma \ref{lem:3-cycle_condition} (iii).
Thus, we assume that $b \neq [a + 3]$ or $v \neq a, b$.
Immediately, we have
\begin{align*}
 g(\lambda)
 & = \pm \> (n + 2, v^{a}; b, a) \> \uline{\mp \> (n + 2, v^{b}; a^{b}, a)} \pm (n + 2, v^{a}; a, b^{a}) \\
 & \phantom{=} \ \mp (n + 2, v^{b}; a, b) \> \uline{\pm \> (n + 2, v^{a}; b^{a}, b)} \mp (n + 2, v^{b}; b, a^{b}).
\end{align*}
In the formula, the underlined 2-terms vanish if and only if $b = [a + 3]$, and the other 2-terms never vanish.
Therefore, since $g(\delta + \lambda) = 0$ by Lemma \ref{lem:3-cycle_condition} (iii), in light of the formula (\ref{eq:O6_f2_special_2}), we have the following cases.
\begin{itemize}
\item
$v = \omega_{2}$, $a = \omega_{3}$ and $b = \omega_{0}$.
\item
$v = \omega_{5}$, $a = \omega_{0}$ and $b = \omega_{3}$.
\end{itemize}
In each case, we have
\[
 \delta + \lambda =
 \begin{cases}
  + \> (\omega_{2}; \omega_{0}, \omega_{3}, \omega_{0}) - (\omega_{5}; \omega_{3}, \omega_{0}, \omega_{3}) \ \text{or} \\
  + \> (\omega_{2}; \omega_{3}, \omega_{0}, \omega_{3}) - (\omega_{5}; \omega_{0}, \omega_{3}, \omega_{0}).
 \end{cases}
\]
It contradicts to our assumption $l(\gamma + \delta + \lambda) = 6$.

Assume that $k = 3$.
Then, in light of Lemma \ref{lem:f-connected3}, we have the following cases with some $\varepsilon \in \{ + 1, - 1 \}$, $v \in O_{6}$, and mutually distinct $a, b, c \in O_{6}$:
\begin{itemize}
\item[(i)]
$\lambda = + \> \varepsilon (n + 1, v; a, b, a) - \varepsilon (n + 1, v; c, a, b) - \varepsilon (n + 1, v; c, b, a)$, and
\item[(ii)]
$\lambda = + \> \varepsilon (n + 1, v; a, b, a) - \varepsilon (n + 1, v; a, b, c) - \varepsilon (n + 1, v; b, a, c)$.
\end{itemize}
Since $g(\delta + \lambda) = 0$, in light of the formula (\ref{eq:O6_f2_special_2}), reduced $g(\lambda)$ has four 2-terms.
Therefore, in light of the argument in the proof of Lemma \ref{lem:O6_f3_usual}, we have $b = [a + 3]$.\footnote{We note that Lemma \ref{lem:O6_f3_usual} was proved independent from the claim of Lemma \ref{lem:O6_f2_special_2-5}.}
Obviously, we respectively have the following formulae:
\begin{itemize}
\item[(i')]
$g(\lambda) = + \> \varepsilon (n + 2, v^{a}; b, a) + \> \varepsilon (n + 2, v^{a}; a, b)$ \\[0.5ex]
$\phantom{g(\lambda) =} - \varepsilon (n + 2, v^{c}; a, b) - \varepsilon (n + 2, v^{c}; b, a)$, and \\[-2ex]
\item[(ii')]
$g(\lambda) = + \> \varepsilon (n + 2, v^{a}; b, a) + \> \varepsilon (n + 2, v^{a}; a, b)$ \\[0.5ex]
$\phantom{g(\lambda) =} - \varepsilon (n + 2, v^{c}; a^{c}, b^{c}) - \varepsilon (n + 2, v^{c}; b^{c}, a^{c})$.
\end{itemize}
Since $g(\delta + \lambda) = 0$, we have $v^{a} = \omega_{1}$ (or $\omega_{4}$) by the formula (\ref{eq:O6_f2_special_2}).
On the other hand, since $c \neq [a + 3]$, we have $v^{c} \neq \omega_{4}$ (or $\omega_{1}$).
Therefore, we have $g(\delta + \lambda) \neq 0$ contradicting to Lemma \ref{lem:3-cycle_condition} (iii).
\end{proof}

\begin{lemma}
\label{lem:O6_f2_special_2++}
Let $\delta$ be the 3-chain {\upshape (i)} or {\upshape (iv)} in Lemma \ref{lem:O6_f2_special_2}.
Furthermore, let $\lambda_{5}, \lambda_{6}, \dots, \lambda_{j + 4}$ be 3-terms of $(O_{6}, S)$ whose degrees are $n + 1$, $\lambda = \sum_{i = 5}^{j + 4} \lambda_{i}$, $\mu_{j + 5}, \mu_{j + 6}, \dots, \mu_{j + k + 4}$ 3-terms of $(O_{6}, S)$ whose degrees are at least $n + 2$, and $\mu = \sum_{i = j + 5}^{j + k + 4} \mu_{i}$.
Assume that $\gamma + \delta + \lambda + \mu$ is a 3-cycle of length $j + k + 4$.
Then, we have $j \geq 2$ and $k \geq 2$.
\end{lemma}

\begin{proof}
As mentioned in the proof of Lemma \ref{lem:O6_f2_special_2+}, we have $j \geq 2$.
Furthermore, in light of Lemmas \ref{lem:3-cycle_condition} (iii), \ref{lem:f-/g-connectedness} (i) and \ref{lem:f-connected1}, at least two of $\mu_{j + 5}, \mu_{j + 6}, \dots, \mu_{j + k + 4}$ have the maximal degree of $\gamma + \delta + \lambda + \mu$.
We thus have $k \geq 2$.
\end{proof}

\begin{lemma}
\label{lem:O6_f2_special_2_more}
Let $\delta$ be the 3-chain {\upshape (ii)} or {\upshape (iii)} in Lemma \ref{lem:O6_f2_special_2}.
Furthermore, let $\lambda_{5}, \lambda_{6}, \dots, \lambda_{k + 4}$ be 3-terms of $(O_{6}, S)$ whose degrees are at least $n + 1$, and $\lambda = \sum_{i = 5}^{k + 4} \lambda_{i}$.
Assume that $\gamma + \delta + \lambda$ is a 3-cycle of length $k + 4$ satisfying $\eta(\pi(\gamma + \delta + \lambda)) \neq 0$.
Then, we have $k \geq 4$.
\end{lemma}

\begin{proof}
Since $\gamma + \delta$ is a 3-cycle satisfying $\eta(\pi(\gamma + \delta)) = 0$ by Lemma \ref{lem:O6_f2_special_2}, $\lambda$ is a 3-cycle of length $k$ satisfying $\eta(\pi(\lambda)) \neq 0$.
In light of Lemmas \ref{lem:3-cycle_condition} (ii) and \ref{lem:f-connected1}, at least two of $\lambda_{5}, \lambda_{6}, \dots, \lambda_{k + 4}$ have the minimal degree of $\lambda$.
Therefore, we have $k \geq 2$.

If $k = 2$, in light of Lemmas \ref{lem:3-cycle_condition} (ii) and \ref{lem:f-connected1}, $\lambda_{5}$ and $\lambda_{6}$ have the same degree.
Therefore, in light of Lemmas \ref{lem:O6_f2_usual} and \ref{lem:O6_f2_special}, we have $g(\lambda) \neq 0$ contradicting to Lemma \ref{lem:3-cycle_condition} (iii), or $\eta(\pi(\lambda)) = 0$ contradicting to our assumption.

If $k = 3$, in light of Lemmas \ref{lem:3-cycle_condition}, \ref{lem:f-/g-connectedness} and \ref{lem:f-connected1}, $\lambda_{5}$, $\lambda_{6}$ and $\lambda_{7}$ have the same degree.
Therefore, in light of Lemma \ref{lem:O6_f3}, we have $g(\lambda) \neq 0$ contradicting to Lemma \ref{lem:3-cycle_condition} (iii).\footnote{We note that Lemma \ref{lem:O6_f3} was also proved independent from the claim of Lemma \ref{lem:O6_f2_special_2-5}.}
\end{proof}

\begin{lemma}
\label{lem:O6_f2_special_3}
Let $\delta_{3}, \delta_{4}$ and $\delta_{5}$ be 3-terms of $(O_{6}, S)$ whose degrees are $n + 1$, and $\delta = \delta_{3} + \delta_{4} + \delta_{5}$.
Assume that $g(\gamma) + f(\delta) = 0$.
Then, we have the following cases with some $a \in O_{6}$ {\upshape (}$a \neq \omega_{0}, \omega_{3}${\upshape )}.
\begin{itemize}
\item[(1)]
$\delta = + \> \langle n + 1, \omega_{2}; \omega_{0}, \omega_{3} \rangle - (n + 1, \omega_{5}; a, \omega_{0}, \omega_{3}) - (n + 1, \omega_{5}; a, \omega_{3}, \omega_{0})$.
\item[(2)]
$\delta = + \> \langle n + 1, \omega_{2}; \omega_{0}, \omega_{3} \rangle - (n + 1, \omega_{5}; \omega_{0}, \omega_{3}, a) - (n + 1, \omega_{5}; \omega_{3}, \omega_{0}, a)$.
\item[(3)]
$\delta = - \> \langle n + 1, \omega_{5}; \omega_{0}, \omega_{3} \rangle + (n + 1, \omega_{2}; a, \omega_{0}, \omega_{3}) + (n + 1, \omega_{2}; a, \omega_{3}, \omega_{0})$.
\item[(4)]
$\delta = - \> \langle n + 1, \omega_{5}; \omega_{0}, \omega_{3} \rangle + (n + 1, \omega_{2}; \omega_{0}, \omega_{3}, a) + (n + 1, \omega_{2}; \omega_{3}, \omega_{0}, a)$.
\end{itemize}
In each case, we have $g(\delta) \neq 0$.
\end{lemma}

\begin{proof}
In a similar way to the proof of Lemma \ref{lem:O6_f2_special_2_3}, in light of the formula (\ref{eq:O6_f2_special}), we have the cases (1)--(4).
It is routine to see that we have $g(\delta) \neq 0$ in each case.
\end{proof}

\begin{lemma}
\label{lem:O6_f2_special_3+}
Let $\delta$ be one of the 3-chains in Lemma \ref{lem:O6_f2_special_3}.
Furthermore, let $\lambda_{6}, \lambda_{7}, \dots, \lambda_{k + 5}$ be 3-terms of $(O_{6}, S)$ whose degrees are at least $n + 1$, and $\lambda = \sum_{i = 6}^{k + 5} \lambda_{i}$.
Assume that $\gamma + \delta + \lambda$ is a 3-cycle of length $k + 5$.
Then, we have $k \geq 3$.
\end{lemma}

\begin{proof}
We only prove the cases that $\delta$ is the 3-chain (1) or (2) in Lemma \ref{lem:O6_f2_special_3}, because the other cases are proved in similar ways.
Let $\delta$ be the 3-chain (1) or (2) in Lemma \ref{lem:O6_f2_special_3}.
Then, we respectively have the following formula:
\begin{itemize}
\item[(1')]
$g(\delta) = + \> (n + 2, \omega_{p}; \omega_{0}, \omega_{3}) + (n + 2, \omega_{p}; \omega_{3}, \omega_{0})$ \\[0.5ex]
$\phantom{g(\delta) =} - (n + 2, \omega_{5}^{a}; \omega_{0}, \omega_{3}) - (n + 2, \omega_{5}^{a}; \omega_{3}, \omega_{0})$, or \\[-2ex]
\item[(2')]
$g(\delta) = + \> (n + 2, \omega_{p}; \omega_{0}, \omega_{3}) + (n + 2, \omega_{p}; \omega_{3}, \omega_{0})$ \\[0.5ex]
$\phantom{g(\delta) =} - (n + 2, \omega_{5}^{a}; \omega_{0}^{a}, \omega_{3}^{a}) - (n + 2, \omega_{5}^{a}; \omega_{3}^{a}, \omega_{0}^{a})$,
\end{itemize}
where $p = 1$ or $4$.

Assume that at least one of $\lambda_{6}, \lambda_{7}, \dots, \lambda_{k + 5}$ has degree $n + 1$.
Then, since $g(\gamma) + f(\delta) = 0$, in light of Lemmas \ref{lem:3-cycle_condition} (i) and \ref{lem:f-connected1}, at least two of $\lambda_{6}, \lambda_{7}, \dots, \lambda_{k + 5}$ have degree $n + 1$.
We thus have $k \geq 2$.
Assume that $k = 2$.
Then, in light of Lemma \ref{lem:f-connected2}, we have
\[
 \lambda = + \> (n + 1, u; b, c, b) - (n + 1, u; c, b, c)
\]
with some $u \in O_{6}$ and distinct $b, c \in O_{6}$.
Since $g(\delta + \lambda) = 0$ by Lemma \ref{lem:3-cycle_condition} (iii), we have $\{ b, c \} = \{ \omega_{0}, \omega_{3} \}$.
On the other hand, since $\omega_{5}^{a} \neq \omega_{1}, \omega_{4}$, we have $\{ u^{\omega_{0}}, u^{\omega_{3}} \} \neq \{ \omega_{p}, \omega_{5}^{a} \}$.
It leads to a contradiction.

Assume that none of $\lambda_{6}, \lambda_{7}, \dots, \lambda_{k + 5}$ have degree $n + 1$.
Then, since $\omega_{p} \neq \omega_{5}^{a}$, in light of Lemma \ref{lem:3-cycle_condition} (i) and the formula (1') or (2'), at least two of $\lambda_{6}, \lambda_{7}, \dots, \lambda_{k + 5}$ have degree $n + 2$.
We thus have $k \geq 2$.
Assume that $k = 2$.
Then, in light of Lemma \ref{lem:O6_ab_ba_1}, we respectively have
\begin{align*}
 \lambda =
 \begin{cases}
  + \> \langle n + 2, \omega_{p}; \omega_{0}, \omega_{3} \rangle - \langle n + 2, \omega_{5}^{a}; \omega_{0}, \omega_{3} \rangle \ \text{or} \\
  + \> \langle n + 2, \omega_{p}; \omega_{0}, \omega_{3} \rangle - \langle n + 2, \omega_{5}^{a}; \omega_{0}^{a}, \omega_{3}^{a} \rangle.
 \end{cases}
\end{align*}
In each case, it is easy to see that we have $g(\lambda) \neq 0$ contradicting to Lemma \ref{lem:3-cycle_condition} (iii).
\end{proof}

\begin{lemma}
\label{lem:O6_f2_special_4}
Let $\delta_{3}, \delta_{4}, \delta_{5}$ and $\delta_{6}$ be 3-terms of $(O_{6}, S)$ whose degrees are $n + 1$, and $\delta = \delta_{3} + \delta_{4} + \delta_{5} + \delta_{6}$.
Assume that $g(\gamma) + f(\delta) = 0$ and $g(\gamma) + f \left( \sum_{i \in I} \delta_{i} \right) \neq 0$ for any non-empty proper subset $I$ of $\{ 3, 4, 5, 6 \}$.
Then, $\gamma + \delta$ is a 3-cycle satisfying $\eta(\pi(\gamma + \delta)) = 0$, or we have $g(\delta) \neq 0$.
\end{lemma}

\begin{proof}
Let $\varepsilon = + 1$ (or $- 1$), $\omega_{p} = \omega_{2}$ (or $\omega_{5}$), and $\omega_{q} = \omega_{5}$ (or $\omega_{2}$).
Then, since $g(\gamma) + f(\delta) = 0$, in light of the formula (\ref{eq:O6_f2_special}), we may assume that
\begin{itemize}
\item[(a)]
$f(\delta_{3} + \delta_{4}) = - \> (n + 1, \omega_{2}; \omega_{0}, \omega_{3}) - (n + 1, \omega_{2}; \omega_{3}, \omega_{0})$ and \\
$f(\delta_{5} + \delta_{6}) = + \> (n + 1, \omega_{5}; \omega_{0}, \omega_{3}) + (n + 1, \omega_{5}; \omega_{3}, \omega_{0})$, or
\item[(b)]
$f(\delta_{3}) = - \> \varepsilon (n + 1, \omega_{p}; \omega_{0}, \omega_{3}) - \varepsilon (n + 1, \omega_{p}; \omega_{3}, \omega_{0})$ and \\
$f(\delta_{4} + \delta_{5} + \delta_{6}) = + \> \varepsilon (n + 1, \omega_{q}; \omega_{0}, \omega_{3}) + \varepsilon (n + 1, \omega_{q}; \omega_{3}, \omega_{0})$.
\end{itemize}

In the case (a), we have the following case with some $a, b \in O_{8}$ ($a \neq \omega_{0}, \omega_{3}$, $b \neq \omega_{0}, \omega_{3}$) by Lemma \ref{lem:O6_ab_ba_2}:
\begin{itemize}
\item[(i)]
$\delta_{3} + \delta_{4} = + \> (\omega_{2}; a, \omega_{0}, \omega_{3}) + (\omega_{2}; a, \omega_{3}, \omega_{0})$ or
\item[(ii)]
$\delta_{3} + \delta_{4} = + \> (\omega_{2}; \omega_{0}, \omega_{3}, a) + (\omega_{2}; \omega_{3}, \omega_{0}, a)$,
\end{itemize}
and
\begin{itemize}
\item[(iii)]
$\delta_{5} + \delta_{6} = - \> (\omega_{5}; b, \omega_{0}, \omega_{3}) - (\omega_{5}; b, \omega_{3}, \omega_{0})$ or
\item[(iv)]
$\delta_{5} + \delta_{6} = - \> (\omega_{5}; \omega_{0}, \omega_{3}, b) - (\omega_{5}; \omega_{3}, \omega_{0}, b)$.
\end{itemize}
It is routine to check that we have $g(\delta) \neq 0$ except the cases (i)-(iii) and (ii)-(iv) with $\{ a, b \} = \{ \omega_{1}, \omega_{4} \}$.
In the cases (i)-(iii) and (ii)-(iv) with $\{ a, b \} = \{ \omega_{1}, \omega_{4} \}$, $\gamma + \delta$ is a 3-cycle.
Then, in light of Lemma \ref{lem:O6_weight}, we have $\eta(\pi(\gamma + \delta)) = 0$.

In the case (b), we have
\[
 \delta_{3} = + \> \varepsilon \langle \omega_{p}; \omega_{0}, \omega_{3} \rangle
\]
by Lemma \ref{lem:O6_ab_ba_1}.
Furthermore, by the assumption that $g(\gamma) + f \left( \sum_{i \in I} \delta_{i} \right) \neq 0$ for any non-empty proper subset $I$ of $\{ 3, 4, 5, 6 \}$, we have
\[
 f \Big( \sum_{j \in J} \delta_{j} \Big) \neq + \> \varepsilon (n + 1, \omega_{q}; \omega_{0}, \omega_{3}) + \varepsilon (n + 1, \omega_{q}; \omega_{3}, \omega_{0})
\]
for any non-empty proper subset $J$ of $\{ 4, 5, 6 \}$.
Therefore, in light of Lemma \ref{lem:O6_ab_ba_3}, we have the following cases with some $a \in O_{6}$ ($a \neq \omega_{0}, \omega_{3}$).
\begin{itemize}
\item
$\delta_{4} + \delta_{5} + \delta_{6} = - \> \varepsilon (\omega_{q}; a, \omega_{0}, \omega_{3}) + \varepsilon (\omega_{q}; \omega_{3}, a, \omega_{0}) - \varepsilon \langle \omega_{q}; \omega_{3}, a \rangle$.
\item
$\delta_{4} + \delta_{5} + \delta_{6} = + \> \varepsilon (\omega_{q}; \omega_{0}, a, \omega_{3}) - \varepsilon (\omega_{q}; a, \omega_{3}, \omega_{0}) - \varepsilon \langle \omega_{q}; \omega_{0}, a \rangle$.
\item
$\delta_{4} + \delta_{5} + \delta_{6} = + \> \varepsilon (\omega_{q}; \omega_{0}, a, \omega_{3}) - \varepsilon (\omega_{q}; \omega_{3}, \omega_{0}, a) - \varepsilon \langle \omega_{q}; \omega_{3}, a \rangle$.
\item
$\delta_{4} + \delta_{5} + \delta_{6} = - \> \varepsilon (\omega_{q}; \omega_{0}, \omega_{3}, a) + \varepsilon (\omega_{q}; \omega_{3}, a, \omega_{0}) - \varepsilon \langle \omega_{q}; \omega_{0}, a \rangle$.
\end{itemize}
It is routine to see that we have $g(\delta) \neq 0$ in each case.
\end{proof}

\begin{lemma}
\label{lem:O6_f2_special_4+}
Let $\delta_{3}, \delta_{4}, \delta_{5}$ and $\delta_{6}$ be 3-terms of $(O_{6}, S)$ whose degrees are $n + 1$, and $\delta = \delta_{3} + \delta_{4} + \delta_{5} + \delta_{6}$.
Furthermore, let $\lambda_{7}, \lambda_{8}, \dots, \lambda_{k + 6}$ be 3-terms of $(O_{6}, S)$ whose degrees are at least $n + 1$, and $\lambda = \sum_{i = 7}^{k + 6} \lambda_{i}$.
Assume that $g(\gamma) + f(\delta) = 0$ and $\gamma + \delta + \lambda$ is a 3-cycle.
Then, we have $k \geq 2$.
\end{lemma}

\begin{proof}
Assume that at least one of $\lambda_{7}, \lambda_{8}, \dots, \lambda_{k + 6}$ has degree $n + 1$.
Then, since $g(\gamma) + f(\delta) = 0$, in light of Lemmas \ref{lem:3-cycle_condition} (i) and \ref{lem:f-connected1}, at least two of $\lambda_{7}, \lambda_{8}, \dots, \lambda_{k + 6}$ have degree $n + 1$.
We thus have $k \geq 2$.

Assume that none of $\lambda_{7}, \lambda_{8}, \dots, \lambda_{k + 6}$ have degree $n + 1$.
Then, in light of Lemmas \ref{lem:3-cycle_condition} (iii), \ref{lem:f-/g-connectedness} (i) and \ref{lem:f-connected1}, at least two of $\lambda_{7}, \lambda_{8}, \dots, \lambda_{k + 6}$ have the maximal degree of $\gamma + \delta + \lambda$.
We thus have $k \geq 2$.
\end{proof}

\begin{lemma}
\label{lem:O6_f2_special_5}
Let $\delta_{3}, \delta_{4}, \delta_{5}, \delta_{6}$ and $\delta_{7}$ be 3-terms of $(O_{6}, S)$ whose degrees are $n + 1$, and $\delta = \delta_{3} + \delta_{4} + \delta_{5} + \delta_{6} + \delta_{7}$.
Assume that $g(\gamma) + f(\delta) = 0$ and $g(\gamma) + f \left( \sum_{i \in I} \delta_{i} \right) \neq 0$ for any non-empty proper subset $I$ of $\{ 3, 4, 5, 6, 7 \}$.
Then, we have $g(\delta) \neq 0$.
\end{lemma}

\begin{proof}
Let $\varepsilon = + 1$ (or $- 1$), $\omega_{p} = \omega_{2}$ (or $\omega_{5}$), and $\omega_{q} = \omega_{5}$ (or $\omega_{2}$).
Then, since $g(\gamma) + f(\delta) = 0$, in light of the formula (\ref{eq:O6_f2_special}), we may assume that
\begin{itemize}
\item[(a)]
$f(\delta_{3} + \delta_{4}) = - \> \varepsilon (\omega_{p}; \omega_{0}, \omega_{3}) - \varepsilon (\omega_{p}; \omega_{3}, \omega_{0})$ and \\
$f(\delta_{5} + \delta_{6} + \delta_{7}) = + \> \varepsilon (\omega_{q}; \omega_{0}, \omega_{3}) + \varepsilon (\omega_{q}; \omega_{3}, \omega_{0})$, or
\item[(b)]
$f(\delta_{3}) = - \> \varepsilon (\omega_{p}; \omega_{0}, \omega_{3}) - \varepsilon (\omega_{p}; \omega_{3}, \omega_{0})$ and \\
$f(\delta_{4} + \delta_{5} + \delta_{6} + \delta_{7}) = + \> \varepsilon (\omega_{q}; \omega_{0}, \omega_{3}) + \varepsilon (\omega_{q}; \omega_{3}, \omega_{0})$.
\end{itemize}

In the case (a), we have the following case with some $a, b \in O_{8}$ ($a \neq \omega_{0}, \omega_{3}$, $b \neq \omega_{0}, \omega_{3}$) in a similar way to the proof of Lemma \ref{lem:O6_f2_special_4}:
\begin{itemize}
\item
$\delta_{3} + \delta_{4} = + \> \varepsilon (\omega_{p}; a, \omega_{0}, \omega_{3}) + \varepsilon (\omega_{p}; a, \omega_{3}, \omega_{0})$ or
\item
$\delta_{3} + \delta_{4} = + \> \varepsilon (\omega_{p}; \omega_{0}, \omega_{3}, a) + \varepsilon (\omega_{p}; \omega_{3}, \omega_{0}, a)$,
\end{itemize}
and
\begin{itemize}
\item
$\delta_{5} + \delta_{6} + \delta_{7} = - \> \varepsilon (\omega_{q}; b, \omega_{0}, \omega_{3}) + \varepsilon (\omega_{q}; \omega_{3}, b, \omega_{0}) - \varepsilon \langle \omega_{q}; \omega_{3}, b \rangle$,
\item
$\delta_{5} + \delta_{6} + \delta_{7} = + \> \varepsilon (\omega_{q}; \omega_{0}, b, \omega_{3}) - \varepsilon (\omega_{q}; b, \omega_{3}, \omega_{0}) - \varepsilon \langle \omega_{q}; \omega_{0}, b \rangle$,
\item
$\delta_{5} + \delta_{6} + \delta_{7} = + \> \varepsilon (\omega_{q}; \omega_{0}, b, \omega_{3}) - \varepsilon (\omega_{q}; \omega_{3}, \omega_{0}, b) - \varepsilon \langle \omega_{q}; \omega_{3}, b \rangle$, or
\item
$\delta_{5} + \delta_{6} + \delta_{7} = - \> \varepsilon (\omega_{q}; \omega_{0}, \omega_{3}, b) + \varepsilon (\omega_{q}; \omega_{3}, b, \omega_{0}) - \varepsilon \langle \omega_{q}; \omega_{0}, b \rangle$.
\end{itemize}
It is routine to see that we have $g(\delta) \neq 0$ in each case.

In the case (b), we have
\[
 \delta_{3} = + \> \varepsilon \langle \omega_{p}; \omega_{0}, \omega_{3} \rangle
\]
by Lemma \ref{lem:O6_ab_ba_1}.
Furthermore, by the assumption that $g(\gamma) + f \left( \sum_{i \in I} \delta_{i} \right) \neq 0$ for any non-empty proper subset $I$ of $\{ 3, 4, 5, 6, 7 \}$, we have
\[
 f \Big( \sum_{j \in J} \delta_{j} \Big) \neq + \> \varepsilon (n + 1, \omega_{q}; \omega_{0}, \omega_{3}) + \varepsilon (n + 1, \omega_{q}; \omega_{3}, \omega_{0})
\]
for any non-empty proper subset $J$ of $\{ 4, 5, 6, 7 \}$.
Therefore, in light of Lemma \ref{lem:O6_ab_ba_4}, we have the following twenty-one cases.
\begin{itemize}
\item[(i)]
$\delta_{4} + \delta_{5} + \delta_{6} + \delta_{7} = - \> \varepsilon \langle \omega_{q}; \omega_{0}, a \rangle - \varepsilon \langle \omega_{q}; \omega_{3}, a \rangle + \varepsilon (\omega_{q}; \omega_{0}, a, \omega_{3}) + \varepsilon (\omega_{q}; \omega_{3}, a, \omega_{0})$ with some $a \in O_{6}$ ($a \neq \omega_{0}, \omega_{3}$).
\item[(ii)]
$\delta_{4} + \delta_{5} + \delta_{6} + \delta_{7} = - \> \varepsilon \langle \omega_{q}; \omega_{3}, a \rangle - \varepsilon (\omega_{q}; \omega_{0}, \omega_{3}, a) - \varepsilon (\omega_{q}; \omega_{0}, a, \omega_{0}) - \varepsilon (\omega_{q}; a, \omega_{3}, \omega_{0})$ with some $a \in O_{6}$ ($a \neq \omega_{0}, \omega_{3}$).
\item[]
\enskip $\vdots$
\item[(xxi)]
$\delta_{4} + \delta_{5} + \delta_{6} + \delta_{7} = + \> \varepsilon (\omega_{q}; \omega_{0}, a, \omega_{3}) - \varepsilon (\omega_{q}; \omega_{3}, \omega_{0}, a) - \varepsilon (\omega_{q}; \omega_{3}, a, b) - \varepsilon (\omega_{q}; a, \omega_{3}, b)$ with some $a, b \in O_{6}$ ($a \neq \omega_{0}, \omega_{3}$, $b \neq \omega_{3}, a$).
\end{itemize}
It is routine to see that we have $g(\delta) \neq 0$ in each case.
\end{proof}

We are now ready to prove Lemma \ref{lem:O6_f2_special_2-5}.

\begin{proof}[Proof of Lemma \ref{lem:O6_f2_special_2-5}]
By Lemmas \ref{lem:3-cycle_condition}, \ref{lem:O6_f2_special_geq2}--\ref{lem:O6_f2_special_5}, we immediately have the claim.
\end{proof}

\section{Proof of Lemma \ref{lem:O6_f3_special}}
\label{sec:proof_of_lemma_O6_f3_special}

The aim of this section is to prove Lemma \ref{lem:O6_f3_special}.
Let $\gamma$ be the 3-chain in Lemma \ref{lem:f-connected3} (i) (or (ii)) with $X = O_{6}$.
Assume that $b = [a + 3]$.
Then, we may assume that $a = \omega_{0}$ and $c = \omega_{1}$, and have $b = \omega_{3}$.
Let $\omega_{p} = \omega_{0}$ (or $\omega_{5}$) and $\omega_{q} = \omega_{3}$ (or $\omega_{2}$).
Then, in light of the formula (\ref{eq:O6_f3_special_i}) (or (\ref{eq:O6_f3_special_ii})), we have
\begin{align}
 g(\gamma)
 & = + \> (n + 1, u^{\omega_{0}}; \omega_{3}, \omega_{0}) + (n + 1, u^{\omega_{0}}; \omega_{0}, \omega_{3}) \notag \\
 & \phantom{=} \ - (n + 1, u^{\omega_{1}}; \omega_{p}, \omega_{q}) - (n + 1, u^{\omega_{1}}; \omega_{q}, \omega_{p}) \label{eq:O6_f3_special}
\end{align}
up to sign.
We note that we have $u^{\omega_{0}} \neq u^{\omega_{1}}$.
To achieve our goal, we show the following lemmas.

\begin{lemma}
\label{lem:O6_f3_special_geq2}
Let $\delta_{4}, \delta_{5}, \dots, \delta_{k + 3}$ be 3-terms of $(O_{6}, S)$ whose degrees are $n + 1$, and $\delta = \sum_{i = 4}^{k + 3} \delta_{i}$.
Assume that $g(\gamma) + f(\delta) = 0$.
Then, we have $k \geq 2$.
\end{lemma}

\begin{proof}
In a similar way to the proof of Lemma \ref{lem:O6_f2_special_geq2}, in light of the formula (\ref{eq:O6_f3_special}), we have the claim.
\end{proof}

\begin{lemma}
\label{lem:O6_f3_special_2}
Let $\delta_{4}$ and $\delta_{5}$ be 3-terms of $(O_{6}, S)$ whose degrees are $n + 1$, and $\delta = \delta_{4} + \delta_{5}$.
Assume that $g(\gamma) + f(\delta) = 0$.
Then, we have
\[
 \delta = + \> \langle n + 1, u^{\omega_{0}}; \omega_{0}, \omega_{3} \rangle - \langle n + 1, u^{\omega_{1}}; \omega_{p}, \omega_{q} \rangle
\]
and $g(\delta) \neq 0$.
\end{lemma}

\begin{proof}
Since $g(\gamma) + f(\delta) = 0$, in light of the formula (\ref{eq:O6_f3_special}), we may assume that
\begin{align*}
 f(\delta_{4}) & = - \> (n + 1, u^{\omega_{0}}; \omega_{0}, \omega_{3}) - (n + 1, u^{\omega_{0}}; \omega_{3}, \omega_{0}), \\
 f(\delta_{5}) & = + \> (n + 1, u^{\omega_{1}}; \omega_{p}, \omega_{q}) + (n + 1, u^{\omega_{1}}; \omega_{q}, \omega_{p}).
\end{align*}
Therefore, in light of Lemma \ref{lem:O6_ab_ba_1}, we have the former claim.
Since $\omega_{q} = [\omega_{p} + 3]$, we have
\begin{align}
 g(\delta)
 & = + \> (n + 2, v; \omega_{0}, \omega_{3}) + (n + 2, v; \omega_{3}, \omega_{0}) \notag \\
 & \phantom{=} \ - (n + 2, w; \omega_{p}, \omega_{q}) - (n + 2, w; \omega_{q}, \omega_{p}), \label{eq:O6_f3_special_2}
\end{align}
where $v = u^{\omega_{0} \omega_{0}}$ or $u$, and $w = u^{\omega_{1} \omega_{p}}$ or $u^{\omega_{1} \omega_{q}}$.
It is routine to check that we have $v \neq w$.
We thus have $g(\delta) \neq 0$.
\end{proof}

\begin{lemma}
\label{lem:O6_f3_special_2_geq2}
Let $\delta$ be the 3-chain in Lemma \ref{lem:O6_f3_special_2}.
Furthermore, let $\lambda_{6}, \lambda_{7}, \dots, \lambda_{k + 5}$ be 3-terms of $(O_{6}, S)$ whose degrees are $n + 2$, and $\lambda = \sum_{i = 6}^{k + 5} \lambda_{i}$.
Assume that $g(\delta) + f(\lambda) = 0$.
Then, we have $k \geq 2$.
\end{lemma}

\begin{proof}
In a similar way to the proof of Lemma \ref{lem:O6_f2_special_geq2}, in light of the formula (\ref{eq:O6_f3_special_2}), we have the claim.
\end{proof}

\begin{lemma}
\label{lem:O6_f3_special_2_2}
Let $\delta$ be the 3-chain in Lemma \ref{lem:O6_f3_special_2}.
Furthermore, let $\lambda_{6}$ and $\lambda_{7}$ be 3-terms of $(O_{6}, S)$ whose degrees are $n + 2$, and $\lambda = \lambda_{6} + \lambda_{7}$.
Assume that $g(\delta) + f(\lambda) = 0$.
Then, we have $g(\lambda) \neq 0$.
\end{lemma}

\begin{proof}
Since $g(\delta) + f(\lambda) = 0$, in light of the formula (\ref{eq:O6_f3_special_2}), we may assume that
\begin{align*}
 f(\lambda_{6}) & = - \> (n + 2, v; \omega_{0}, \omega_{3}) - (n + 2, v; \omega_{3}, \omega_{0}), \\
 f(\lambda_{7}) & = + \> (n + 2, w; \omega_{p}, \omega_{q}) + (n + 2, w; \omega_{q}, \omega_{p}).
\end{align*}
Therefore, in light of Lemma \ref{lem:O6_ab_ba_1}, we have
\[
 \lambda = + \> \langle n + 2, v; \omega_{0}, \omega_{3} \rangle - \langle n + 2, w; \omega_{p}, \omega_{q} \rangle.
\]
It is routine to check that we have $g(\lambda) \neq 0$.
\end{proof}

\begin{lemma}
\label{lem:O6_f3_special_2+}
Let $\delta$ be the 3-chain in Lemma \ref{lem:O6_f3_special_2}.
Furthermore, let $\lambda_{6}, \lambda_{7}, \dots, \lambda_{k + 5}$ be 3-terms of $(O_{6}, S)$ whose degrees are $n + 1$, and $\lambda = \sum_{i = 6}^{k + 5} \lambda_{i}$.
Assume that $\gamma + \delta + \lambda$ is a 3-cycle of length $k + 5$.
Then, we have $k \geq 3$.
\end{lemma}

\begin{proof}
Since $f(\lambda) = 0$ by Lemma \ref{lem:3-cycle_condition} (i), we have $k \geq 2$ by Lemma \ref{lem:f-connected1}.
Assume that $k = 2$.
Then, in light of Lemma \ref{lem:f-connected2}, we have
\[
 \lambda = + \> (n + 1, x; a, b, a) - (n + 1, x; b, a, b)
\]
with some $x \in O_{6}$ and distinct $a, b \in O_{6}$.
Since reduced $g(\delta)$ has four 2-terms (see the formula (\ref{eq:O6_f3_special_2})), in light of Lemma \ref{lem:3-cycle_condition} (iii), reduced $g(\lambda)$ has four 2-terms.
It yields $b = [a + 3]$ and $x \neq a, b$.
We thus have $x^{b} = [x^{a} + 3]$.
On the other hand, we have $w \neq [v + 3]$.
Therefore, we have $g(\delta + \lambda) \neq 0$ contradicting to Lemma \ref{lem:3-cycle_condition} (iii).
\end{proof}

\begin{lemma}
\label{lem:O6_f3_special_2++}
Let $\delta$ be the 3-chain in Lemma \ref{lem:O6_f3_special_2}.
Furthermore, let $\lambda_{6}, \lambda_{7}, \dots, \lambda_{j + 5}$ be 3-terms of $(O_{6}, S)$ whose degrees are $n + 1$, $\lambda = \sum_{i = 6}^{j + 5} \lambda_{i}$, $\mu_{j + 6}, \mu_{j + 7}, \dots, \mu_{j + k + 5}$ 3-terms of $(O_{6}, S)$ whose degrees are at least $n + 2$, and $\mu = \sum_{i = j + 6}^{j + k + 5} \mu_{i}$.
Assume that $\gamma + \delta + \lambda + \mu$ is a 3-cycle of length $j + k + 5$.
Then, we have $j \geq 2$ and $k \geq 2$.
\end{lemma}

\begin{proof}
As mentioned in the proof of Lemma \ref{lem:O6_f3_special_2+}, we have $j \geq 2$.
Furthermore, in light of Lemmas \ref{lem:3-cycle_condition} (iii), \ref{lem:f-/g-connectedness} (i) and \ref{lem:f-connected1}, at least two of $\mu_{j + 6}, \mu_{j + 7}, \dots, \mu_{j + k + 5}$ have the maximal degree of $\gamma + \delta + \lambda + \mu$.
We thus have $k \geq 2$.
\end{proof}

\begin{lemma}
\label{lem:O6_f3_special_3}
Let $\delta_{4}, \delta_{5}$ and $\delta_{6}$ be 3-terms of $(O_{6}, S)$ whose degrees are $n + 1$, and $\delta = \delta_{4} + \delta_{5} + \delta_{6}$.
Assume that $g(\gamma) + f(\delta) = 0$.
Then, $\gamma + \delta$ is a 3-cycle satisfying $\eta(\pi(\gamma + \delta)) = 0$, or we have $g(\delta) \neq 0$.
\end{lemma}

\begin{proof}
Since $g(\gamma) + f(\delta) = 0$, in light of the formula (\ref{eq:O6_f3_special}), we may assume that
\begin{itemize}
\item
$f(\delta_{4}) = - \> (n + 1, u^{\omega_{0}}; \omega_{0}, \omega_{3}) - (n + 1, u^{\omega_{0}}; \omega_{3}, \omega_{0})$ and \\
$f(\delta_{5} + \delta_{6}) = + \> (n + 1, u^{\omega_{1}}; \omega_{p}, \omega_{q}) + (n + 1, u^{\omega_{1}}; \omega_{q}, \omega_{p})$, or
\item
$f(\delta_{4}) = + \> (n + 1, u^{\omega_{1}}; \omega_{p}, \omega_{q}) + (n + 1, u^{\omega_{1}}; \omega_{q}, \omega_{p})$ and \\
$f(\delta_{5} + \delta_{6}) = - \> (n + 1, u^{\omega_{0}}; \omega_{0}, \omega_{3}) - (n + 1, u^{\omega_{0}}; \omega_{3}, \omega_{0})$.
\end{itemize}
Therefore, in light of Lemmas \ref{lem:O6_ab_ba_1} and \ref{lem:O6_ab_ba_2}, we have the following cases with some $a, b \in O_{6}$ ($a \neq \omega_{p}, \omega_{q}$, $b \neq \omega_{0}, \omega_{3}$).
\begin{itemize}
\item[(i)]
$\delta = + \> (u^{\omega_{0}}; \omega_{0}, \omega_{3}, \omega_{0}) - (u^{\omega_{1}}; a, \omega_{p}, \omega_{q}) - (u^{\omega_{1}}; a, \omega_{q}, \omega_{p})$.
\item[(ii)]
$\delta = + \> (u^{\omega_{0}}; \omega_{3}, \omega_{0}, \omega_{3}) - (u^{\omega_{1}}; a, \omega_{p}, \omega_{q}) - (u^{\omega_{1}}; a, \omega_{q}, \omega_{p})$.
\item[(iii)]
$\delta = + \> (u^{\omega_{0}}; \omega_{0}, \omega_{3}, \omega_{0}) - (u^{\omega_{1}}; \omega_{p}, \omega_{q}, a) - (u^{\omega_{1}}; \omega_{q}, \omega_{p}, a)$.
\item[(iv)]
$\delta = + \> (u^{\omega_{0}}; \omega_{3}, \omega_{0}, \omega_{3}) - (u^{\omega_{1}}; \omega_{p}, \omega_{q}, a) - (u^{\omega_{1}}; \omega_{q}, \omega_{p}, a)$.
\item[(v)]
$\delta = - \> (u^{\omega_{1}}; \omega_{p}, \omega_{q}, \omega_{p}) + (u^{\omega_{0}}; b, \omega_{0}, \omega_{3}) - (u^{\omega_{0}}; b, \omega_{3}, \omega_{0})$.
\item[(vi)]
$\delta = - \> (u^{\omega_{1}}; \omega_{q}, \omega_{p}, \omega_{q}) + (u^{\omega_{0}}; b, \omega_{0}, \omega_{3}) - (u^{\omega_{0}}; b, \omega_{3}, \omega_{0})$.
\item[(vii)]
$\delta = - \> (u^{\omega_{1}}; \omega_{p}, \omega_{q}, \omega_{p}) + (u^{\omega_{0}}; \omega_{0}, \omega_{3}, b) - (u^{\omega_{0}}; \omega_{3}, \omega_{0}, b)$.
\item[(viii)]
$\delta = - \> (u^{\omega_{1}}; \omega_{q}, \omega_{p}, \omega_{q}) + (u^{\omega_{0}}; \omega_{0}, \omega_{3}, b) - (u^{\omega_{0}}; \omega_{3}, \omega_{0}, b)$.
\end{itemize}
It is routine to check that we have $g(\delta) \neq 0$ except the following cases.
\begin{itemize}
\item
(i) or (ii), $u = \omega_{0}$ or $\omega_{3}$, $a = \omega_{4}$, $\omega_{p} = \omega_{0}$, and thus $\omega_{q} = \omega_{3}$.
\item
(ii), $u = \omega_{1}$ or $\omega_{4}$, $a = \omega_{1}$ or $\omega_{4}$, $\omega_{p} = \omega_{0}$, and thus $\omega_{q} = \omega_{3}$.
\item
(i), $u = \omega_{2}$ or $\omega_{5}$, $a = \omega_{1}$, $\omega_{p} = \omega_{0}$, and thus $\omega_{q} = \omega_{3}$.
\item
(ii), $u = \omega_{2}$ or $\omega_{5}$, $a = \omega_{4}$, $\omega_{p} = \omega_{0}$, and thus $\omega_{q} = \omega_{3}$.
\item
(iii) or (iv), $u = \omega_{0}$ or $\omega_{3}$, $a = \omega_{4}$, $\omega_{p} = \omega_{5}$, and thus $\omega_{q} = \omega_{2}$.
\item
(iv), $u = \omega_{1}$ or $\omega_{4}$, $a = \omega_{1}$ or $\omega_{4}$, $\omega_{p} = \omega_{5}$, and thus $\omega_{q} = \omega_{2}$.
\item
(iii), $u = \omega_{2}$ or $\omega_{5}$, $a = \omega_{1}$, $\omega_{p} = \omega_{5}$, and thus $\omega_{q} = \omega_{2}$.
\item
(iv), $u = \omega_{2}$ or $\omega_{5}$, $a = \omega_{4}$, $\omega_{p} = \omega_{5}$, and thus $\omega_{q} = \omega_{2}$.
\item
(v), $u = \omega_{0}$ or $\omega_{3}$, $b = \omega_{2}$, $\omega_{p} = \omega_{0}$, and thus $\omega_{q} = \omega_{3}$.
\item
(vi), $u = \omega_{0}$ or $\omega_{3}$, $b = \omega_{5}$, $\omega_{p} = \omega_{0}$, and thus $\omega_{q} = \omega_{3}$.
\item
(v), $u = \omega_{1}$ or $\omega_{4}$, $b = \omega_{2}, \omega_{5}$, $\omega_{p} = \omega_{0}$, and thus $\omega_{q} = \omega_{3}$.
\item
(v) or (vi), $u = \omega_{2}$ or $\omega_{5}$, $b = \omega_{2}$, $\omega_{p} = \omega_{0}$, and thus $\omega_{q} = \omega_{3}$.
\item
(vii) or (viii), $u = \omega_{0}$ or $\omega_{3}$, $b = \omega_{1}$, $\omega_{p} = \omega_{5}$, and thus $\omega_{q} = \omega_{2}$.
\item
(vii), $u = \omega_{1}$ or $\omega_{4}$, $b = \omega_{1}$, $\omega_{p} = \omega_{5}$, and thus $\omega_{q} = \omega_{2}$.
\item
(viii), $u = \omega_{1}$ or $\omega_{4}$, $b = \omega_{4}$, $\omega_{p} = \omega_{5}$, and thus $\omega_{q} = \omega_{2}$.
\item
(vii), $u = \omega_{2}$ or $\omega_{5}$, $b = \omega_{1}, \omega_{4}$, $\omega_{p} = \omega_{5}$, and thus $\omega_{q} = \omega_{2}$.
\end{itemize}
Obviously, $\gamma + \delta$ are 3-cycles in the above cases.
Then, in light of Lemma \ref{lem:O6_weight}, we have $\eta(\pi(\gamma + \delta)) = 0$.
\end{proof}

\begin{lemma}
\label{lem:O6_f3_special_3+}
Let $\delta_{4}, \delta_{5}$ and $\delta_{6}$ be 3-terms of $(O_{6}, S)$ whose degrees are $n + 1$, and $\delta = \delta_{4} + \delta_{5} + \delta_{6}$.
Furthermore, let $\lambda_{7}, \lambda_{8}, \dots, \lambda_{k + 6}$ be 3-terms of $(O_{6}, S)$ whose degrees are at least $n + 1$, and $\lambda = \sum_{i = 7}^{k + 6} \lambda_{i}$.
Assume that $g(\gamma) + f(\delta) = 0$ and $\gamma + \delta + \lambda$ is a 3-cycle of length $k + 6$.
Then, we have $k \geq 2$.
\end{lemma}

\begin{proof}
In a similar way to the proof of Lemma \ref{lem:O6_f2_special_4+}, we have the claim.
\end{proof}

\begin{lemma}
\label{lem:O6_f3_special_4}
Let $\delta_{4}, \delta_{5}, \delta_{6}$ and $\delta_{7}$ be 3-terms of $(O_{6}, S)$ whose degrees are $n + 1$, and $\delta = \delta_{4} + \delta_{5} + \delta_{6} + \delta_{7}$.
Assume that $g(\gamma) + f(\delta) = 0$ and $g(\gamma) + f \left( \sum_{i \in I} \delta_{i} \right) \neq 0$ for any non-empty proper subset $I$ of $\{ 4, 5, 6, 7 \}$.
Then, $\gamma + \delta$ is a 3-cycle satisfying $\eta(\pi(\gamma + \delta)) = 0$, or we have $g(\delta) \neq 0$.
\end{lemma}

\begin{proof}
Since $g(\gamma) + f(\delta) = 0$, in light of the formula (\ref{eq:O6_f3_special}), we may assume that
\begin{itemize}
\item[(a)]
$f(\delta_{4} + \delta_{5}) = - \> (n + 1, u^{\omega_{0}}; \omega_{0}, \omega_{3}) - (n + 1, u^{\omega_{0}}; \omega_{3}, \omega_{0})$ and \\
$f(\delta_{6} + \delta_{7}) = + \> (n + 1, u^{\omega_{1}}; \omega_{p}, \omega_{q}) + (n + 1, u^{\omega_{1}}; \omega_{q}, \omega_{p})$,
\item[(b)]
$f(\delta_{4}) = - \> (n + 1, u^{\omega_{0}}; \omega_{0}, \omega_{3}) - (n + 1, u^{\omega_{0}}; \omega_{3}, \omega_{0})$ and \\
$f(\delta_{5} + \delta_{6} + \delta_{7}) = + \> (n + 1, u^{\omega_{1}}; \omega_{p}, \omega_{q}) + (n + 1, u^{\omega_{1}}; \omega_{q}, \omega_{p})$, or
\item[(c)]
$f(\delta_{4}) = + \> (n + 1, u^{\omega_{1}}; \omega_{p}, \omega_{q}) + (n + 1, u^{\omega_{1}}; \omega_{q}, \omega_{p})$ and \\
$f(\delta_{5} + \delta_{6} + \delta_{7}) = - \> (n + 1, u^{\omega_{0}}; \omega_{0}, \omega_{3}) - (n + 1, u^{\omega_{0}}; \omega_{3}, \omega_{0})$.
\end{itemize}

In the case (a), we have the following cases with some $a, b \in O_{6}$ ($a \neq \omega_{0}, \omega_{3}$, $b \neq \omega_{p}, \omega_{q}$) by Lemma \ref{lem:O6_ab_ba_2}.
\begin{itemize}
\item[(i)]
$\delta = + \> (u^{\omega_{0}}; a, \omega_{0}, \omega_{3}) + (u^{\omega_{0}}; a, \omega_{3}, \omega_{0}) - (u^{\omega_{1}}; b, \omega_{p}, \omega_{q}) - (u^{\omega_{1}}; b, \omega_{q}, \omega_{p})$.
\item[(ii)]
$\delta = + \> (u^{\omega_{0}}; a, \omega_{0}, \omega_{3}) + (u^{\omega_{0}}; a, \omega_{3}, \omega_{0}) - (u^{\omega_{1}}; \omega_{p}, \omega_{q}, b) - (u^{\omega_{1}}; \omega_{q}, \omega_{p}, b)$.
\item[(iii)]
$\delta = + \> (u^{\omega_{0}}; \omega_{0}, \omega_{3}, a) + (u^{\omega_{0}}; \omega_{3}, \omega_{0}, a) - (u^{\omega_{1}}; b, \omega_{p}, \omega_{q}) - (u^{\omega_{1}}; b, \omega_{q}, \omega_{p})$.
\item[(iv)]
$\delta = + \> (u^{\omega_{0}}; \omega_{0}, \omega_{3}, a) + (u^{\omega_{0}}; \omega_{3}, \omega_{0}, a) - (u^{\omega_{1}}; \omega_{p}, \omega_{q}, b) - (u^{\omega_{1}}; \omega_{q}, \omega_{p}, b)$.
\end{itemize}
It is routine to check that we have $g(\delta) \neq 0$ except the following cases.
\begin{itemize}
\item
(i), $u = \omega_{0}$ or $\omega_{3}$, $(a, b) = (\omega_{1}, \omega_{2})$ or $(\omega_{1}, \omega_{5})$, $\omega_{p} = \omega_{0}$, and thus $\omega_{q} = \omega_{3}$.
\item
(i), $u = \omega_{1}$ or $\omega_{4}$, $(a, b) = (\omega_{1}, \omega_{5})$ or $(\omega_{4}, \omega_{2})$, $\omega_{p} = \omega_{0}$, and thus $\omega_{q} = \omega_{3}$.
\item
(i), $u = \omega_{2}$ or $\omega_{5}$, $(a, b) = (\omega_{1}, \omega_{5})$ or $(\omega_{4}, \omega_{5})$, $\omega_{p} = \omega_{0}$, and thus $\omega_{q} = \omega_{3}$.
\item
(iv), $u = \omega_{0}$ or $\omega_{3}$, $(a, b) = (\omega_{2}, \omega_{0})$ or $(\omega_{5}, \omega_{3})$, $\omega_{p} = \omega_{5}$, and thus $\omega_{q} = \omega_{2}$.
\item
(iv), $u = \omega_{1}$ or $\omega_{4}$, $(a, b) = (\omega_{2}, \omega_{0})$ or $(\omega_{5}, \omega_{0})$, $\omega_{p} = \omega_{5}$, and thus $\omega_{q} = \omega_{2}$.
\item
(iv), $u = \omega_{2}$ or $\omega_{5}$, $(a, b) = (\omega_{2}, \omega_{0})$ or $(\omega_{2}, \omega_{3})$, $\omega_{p} = \omega_{5}$, and thus $\omega_{q} = \omega_{2}$.
\end{itemize}
Obviously, $\gamma + \delta$ are 3-cycles in the above cases.
Then, in light of Lemma \ref{lem:O6_weight}, we have $\eta(\pi(\gamma + \delta)) = 0$.

In the case (b) or (c), we have $f \big( \sum_{j \in J} \delta_{j} \big) \neq f(\delta_{5} + \delta_{6} + \delta_{7})$ for any non-empty proper subset $J$ of $\{ 5, 6, 7 \}$ by the assumption that $g(\gamma) + f \left( \sum_{i \in I} \delta_{i} \right) \neq 0$ for any non-empty proper subset $I$ of $\{ 4, 5, 6, 7 \}$.
Therefore, in light of Lemmas \ref{lem:O6_ab_ba_1} and \ref{lem:O6_ab_ba_3}, we have the following cases with some $a, b \in O_{6}$ ($a \neq \omega_{p}, \omega_{q}$, $b \neq \omega_{0}, \omega_{3}$).
\begin{itemize}
\item
$\delta = + \> \langle u^{\omega_{0}}; \omega_{0}, \omega_{3} \rangle + (u^{\omega_{1}}; \omega_{p}, a, \omega_{q}) - (u^{\omega_{1}}; \omega_{q}, \omega_{p}, a) - \langle u^{\omega_{1}}; a, \omega_{q} \rangle$.
\item
$\delta = + \> \langle u^{\omega_{0}}; \omega_{0}, \omega_{3} \rangle + (u^{\omega_{1}}; \omega_{q}, a, \omega_{p}) - (u^{\omega_{1}}; \omega_{p}, \omega_{q}, a) - \langle u^{\omega_{1}}; a, \omega_{p} \rangle$.
\item
$\delta = + \> \langle u^{\omega_{0}}; \omega_{0}, \omega_{3} \rangle - (u^{\omega_{1}}; a, \omega_{p}, \omega_{q}) - \langle u^{\omega_{1}}; a, \omega_{q} \rangle + (u^{\omega_{1}}; \omega_{q}, a, \omega_{p})$.
\item
$\delta = + \> \langle u^{\omega_{0}}; \omega_{0}, \omega_{3} \rangle - (u^{\omega_{1}}; a, \omega_{q}, \omega_{p}) - \langle u^{\omega_{1}}; a, \omega_{p} \rangle + (u^{\omega_{1}}; \omega_{p}, a, \omega_{q})$.
\item
$\delta = - \> \langle u^{\omega_{1}}; \omega_{p}, \omega_{q} \rangle - (u^{\omega_{0}}; \omega_{0}, b, \omega_{3}) + (u^{\omega_{0}}; \omega_{3}, \omega_{0}, b) + \langle u^{\omega_{0}}; b, \omega_{3} \rangle$.
\item
$\delta = - \> \langle u^{\omega_{1}}; \omega_{p}, \omega_{q} \rangle - (u^{\omega_{0}}; \omega_{3}, b, \omega_{0}) + (u^{\omega_{0}}; \omega_{0}, \omega_{3}, b) + \langle u^{\omega_{0}}; b, \omega_{0} \rangle$.
\item
$\delta = - \> \langle u^{\omega_{1}}; \omega_{p}, \omega_{q} \rangle + (u^{\omega_{0}}; b, \omega_{0}, \omega_{3}) + \langle u^{\omega_{0}}; b, \omega_{3} \rangle - (u^{\omega_{0}}; \omega_{3}, b, \omega_{0})$.
\item
$\delta = - \> \langle u^{\omega_{1}}; \omega_{p}, \omega_{q} \rangle + (u^{\omega_{0}}; b, \omega_{3}, \omega_{0}) + \langle u^{\omega_{0}}; b, \omega_{0} \rangle - (u^{\omega_{0}}; \omega_{0}, b, \omega_{3})$.
\end{itemize}
It is routine to see that we have $g(\delta) \neq 0$ in each case.
\end{proof}

We are now ready to prove Lemma \ref{lem:O6_f3_special}.

\begin{proof}[Proof of Lemma \ref{lem:O6_f3_special}]
By Lemmas \ref{lem:3-cycle_condition}, \ref{lem:O6_f3_special_geq2}--\ref{lem:O6_f3_special_4}, we immediately have the claim.
\end{proof}

\section{Proof of Proposition \ref{prop:O6_III}}
\label{sec:proof_of_proposition_O6_III}

We devote this section to prove Proposition \ref{prop:O6_III}.
Let $\gamma_{1}, \gamma_{2}, \dots, \gamma_{6}$ be 3-terms of $(O_{6}, S)$ and $\gamma = \gamma_{1} + \gamma_{2} + \dots + \gamma_{6}$.
Assume that $\gamma$ is a 3-cycle satisfying $l(\gamma) = |T_{n}(\gamma)| = 6$ and $\eta(\pi(\gamma)) \neq 0$.
Then, for the same reason as in the proof of Proposition \ref{prop:O6_I}, we may assume that $\gamma_{1}, \gamma_{2}, \dots, \gamma_{6}$ are not $f$-connected taking reverse if necessary.
Therefore, in light of Lemmas \ref{lem:3-cycle_condition} (ii) and \ref{lem:f-connected1}, we have the following cases.
\begin{itemize}
\item[(i)]
$\gamma_{1}, \gamma_{2}, \dots, \gamma_{6}$ are divided into three sets, each of which consists of two $f$-connected 3-terms.
\item[(ii)]
$\gamma_{1}, \gamma_{2}, \dots, \gamma_{6}$ are divided into two sets, each of which consists of three $f$-connected 3-terms.
\item[(iii)]
$\gamma_{1}, \gamma_{2}, \dots, \gamma_{6}$ are divided into two sets, which respectively consist of four and two $f$-connected 3-terms.
\end{itemize}

\subsection{Case (i)}
\label{subsec:proof_of_proposition_O6_III_i}

In light of Lemma \ref{lem:f-connected2}, we may assume that
\[
 \gamma_{2 i - 1} + \gamma_{2 i}
 = + \> (u_{i}; a_{i}, b_{i}, a_{i}) - (u_{i}; b_{i}, a_{i}, b_{i})
\]
with some $u_{i} \in O_{6}$ and distinct $a_{i}, b_{i} \in O_{6}$ ($1 \leq i \leq 3$).
Let $\delta_{i}$ denote $\gamma_{2 i - 1} + \gamma_{2 i}$.
If $b_{j} = [a_{j} + 3]$ and $u_{j} = a_{j}$ or $b_{j}$ for some $j$, in light of Lemma \ref{lem:O6_f2_special}, $\delta_{j}$ is a 3-cycle satisfying $\eta(\pi(\delta_{j})) = 0$.
Therefore, $\gamma - \delta_{j}$ is a 3-cycle satisfying $l(\gamma - \delta_{j}) = |T_{n}(\gamma - \delta_{j})| = 4$ and $\eta(\pi(\gamma - \delta_{j})) \neq 0$.
It contradicts to Proposition \ref{prop:O6_I}.
We thus have $b_{i} \neq [a_{i} + 3]$ or $u_{i} \neq a_{i}, b_{i}$ for each $i$.

Assume that $b_{1} = [a_{1} + 3]$ and $u_{1} \neq a_{1}, b_{1}$.
If $b_{2} \neq [a_{2} + 3]$, it is easy to see that we have $l(g(\delta_{1} + \delta_{2})) = 10$.
We thus have $g(\gamma) \neq 0$ contradicting to Lemma \ref{lem:3-cycle_condition} (iii), because $l(g(\delta_{3})) \leq 6$.
Therefore, we have $b_{2} = [a_{2} + 3]$ (and $u_{2} \neq a_{2}, b_{2}$), and $b_{3} = [a_{3} + 3]$ (and $u_{3} \neq a_{3}, b_{3}$) for the same reason.
Then, we have $\eta(\pi(\gamma)) = 0$ contradicting to our assumption.
Obviously, we are faced with the same situation if $b_{2} = [a_{2} + 3]$ and $u_{2} \neq a_{2}, b_{2}$, or $b_{3} = [a_{3} + 3]$ and $u_{3} \neq a_{3}, b_{3}$.

Assume that $b_{i} \neq [a_{i} + 3]$ for each $i$.
Then, we may assume that
\[
 \delta_{1} = + \> (u_{1}; \omega_{0}, \omega_{1}, \omega_{0}) - (u_{1}; \omega_{1}, \omega_{0}, \omega_{1}).
\]
Obviously, we have
\begin{align*}
 g(\delta_{1})
 & = + \> (u_{1}^{\omega_{0}}; \omega_{1}, \omega_{0}) - (u_{1}^{\omega_{1}}; \omega_{5}, \omega_{0}) + (u_{1}^{\omega_{0}}; \omega_{0}, \omega_{2}) \\
 & \phantom{=} \ - (u_{1}^{\omega_{1}}; \omega_{0}, \omega_{1}) + (u_{1}^{\omega_{0}}; \omega_{2}, \omega_{1}) - (u_{1}^{\omega_{1}}; \omega_{1}, \omega_{5}), \\
 g(\delta_{2})
 & = + \> (u_{2}^{a_{2}}; b_{2}, a_{2}) - (u_{2}^{b_{2}}; a_{2}^{b_{2}}, a_{2}) + (u_{2}^{a_{2}}; a_{2}, b_{2}^{a_{2}}) \\
 & \phantom{=} \ - (u_{2}^{b_{2}}; a_{2}, b_{2}) + (u_{2}^{a_{2}}; b_{2}^{a_{2}}, b_{2}) - (u_{2}^{b_{2}}; b_{2}, a_{2}^{b_{2}}).
\end{align*}
Since the 2-term $+ (u_{1}^{\omega_{0}}; \omega_{1}, \omega_{0})$ survives in reduced $g(\delta_{1})$, in light of Lemma \ref{lem:3-cycle_condition} (iii), we have
\[
 (u_{1}^{\omega_{0}}; \omega_{1}, \omega_{0}) = (u_{2}^{b_{2}}; a_{2}^{b_{2}}, a_{2}), \, (u_{2}^{b_{2}}; a_{2}, b_{2}) \ \text{or} \ (u_{2}^{b_{2}}; b_{2}, a_{2}^{b_{2}}),
\]
swapping the roles of $\delta_{2}$ and $\delta_{3}$ if necessary.
We note that $u_{2}^{b_{2}} = u_{1}^{\omega_{0}}$ yields $u_{2} = u_{1}^{\omega_{0} \overline{b_{2}}} = u_{1}^{\omega_{0} [b_{2} + 3]}$.

If $(u_{1}^{\omega_{0}}; \omega_{1}, \omega_{0}) = (u_{2}^{b_{2}}; a_{2}^{b_{2}}, a_{2})$, since $a_{2} = \omega_{0}$ and $a_{2}^{b_{2}} = \omega_{1}$ yield $b_{2} = \omega_{2}$, we have
\begin{align*}
 g(\delta_{1} + \delta_{2})
 & = - \> (u_{1}^{\omega_{1}}; \omega_{5}, \omega_{0}) - (u_{1}^{\omega_{1}}; \omega_{0}, \omega_{1}) - (u_{1}^{\omega_{1}}; \omega_{1}, \omega_{5}) \\
 & \phantom{=} \ + (u_{1}^{\omega_{0} \omega_{5} \omega_{0}}; \omega_{2}, \omega_{0}) + (u_{1}^{\omega_{0} \omega_{5} \omega_{0}}; \omega_{0}, \omega_{4}) + (u_{1}^{\omega_{0} \omega_{5} \omega_{0}}; \omega_{4}, \omega_{2}).
\end{align*}
Since reduced $g(\delta_{3})$ could not have two 2-terms whose colors are $(\omega_{0}, \omega_{1})$ and $(\omega_{0}, \omega_{4})$ at the same moment, we have $g(\gamma) \neq 0$ contradicting to Lemma \ref{lem:3-cycle_condition} (iii).

If $(u_{1}^{\omega_{0}}; \omega_{1}, \omega_{0}) = (u_{2}^{b_{2}}; a_{2}, b_{2})$, since $u_{2} = u_{1}^{\omega_{0} \overline{\omega_{0}}} = u_{1}$, we have $\delta_{2} = - \delta_{1}$.
We thus have $l(\gamma) \leq 2$ contradicting to our assumption.

If $(u_{1}^{\omega_{0}}; \omega_{1}, \omega_{0}) = (u_{2}^{b_{2}}; b_{2}, a_{2}^{b_{2}})$, since $b_{2} = \omega_{1}$ and $a_{2}^{b_{2}} = \omega_{0}$ yield $a_{2} = \omega_{2}$, we have
\begin{align*}
 g(\delta_{1} + \delta_{2})
 & = - \> (u_{1}^{\omega_{1}}; \omega_{5}, \omega_{0}) - (u_{1}^{\omega_{1}}; \omega_{0}, \omega_{1}) - (u_{1}^{\omega_{1}}; \omega_{1}, \omega_{5}) \\
 & \phantom{=} \ + (u_{1}^{\omega_{0} \omega_{4} \omega_{2}}; \omega_{1}, \omega_{2}) + (u_{1}^{\omega_{0} \omega_{4} \omega_{2}}; \omega_{2}, \omega_{3}) + (u_{1}^{\omega_{0} \omega_{4} \omega_{2}}; \omega_{3}, \omega_{1}).
\end{align*}
Since reduced $g(\gamma_{3})$ could not have two 2-terms whose colors are $(\omega_{0}, \omega_{1})$ and $(\omega_{3}, \omega_{1})$ at the same moment, we have $g(\gamma) \neq 0$ contradicting to Lemma \ref{lem:3-cycle_condition} (iii).

\subsection{Case (ii)}
\label{subsec:proof_of_proposition_O6_III_ii}

In light of Lemma \ref{lem:f-connected3}, we may assume that
\begin{align*}
 \gamma_{1} & = \pm (u; a, b, a), \\
 \gamma_{2} & = \mp (u; c, a, b) \ \text{(or $\mp (u; a, b, c)$)}, \\
 \gamma_{3} & = \mp (u; c, b, a) \ \text{(or $\mp (u; b, a, c)$)}
\end{align*}
with some $u \in O_{6}$ and mutually distinct $a, b, c \in O_{6}$.
For the subsequent arguments, we rewrite $\gamma_{i}$ as $\varepsilon_{i} (u_{i}; a_{i}, b_{i}, c_{i})$ for $4 \leq i \leq 6$.

Assume that $b \neq [a + 3]$.
Then, we may assume that $a = \omega_{0}$ and $b = \omega_{1}$.
We note that $\gamma_{1}$ is of type 1.
If $c = \omega_{2}$ or $\omega_{5}$, $\gamma_{2}$ is of type 3 or 2 and $\gamma_{3}$ of type 2 or 3, respectively.
Furthermore, $u^{a b a}$, $u^{c a b}$ (or $u^{a b c}$) and $u^{c b a}$ (or $u^{b a c}$) are mutually different.
Therefore, in light of Lemmas \ref{lem:3-cycle_condition} (iii), \ref{lem:reverse} (ii), \ref{lem:f-/g-connectedness} and \ref{lem:f-connected1}--\ref{lem:f-connected3},
\begin{itemize}
\item
at least two of $\gamma_{4}, \gamma_{5}$ and $\gamma_{6}$ satisfy $u_{i}^{a_{i} b_{i} c_{i}} = u^{a b a}$,
\item
at least two of $\gamma_{4}, \gamma_{5}$ and $\gamma_{6}$ satisfy $u_{i}^{a_{i} b_{i} c_{i}} = u^{c a b}$ (or $u^{a b c}$) if $c = \omega_{2}$, otherwise $u_{i}^{a_{i} b_{i} c_{i}} = u^{c b a}$ (or $u^{b a c}$), and
\item
at least one of $\gamma_{4}, \gamma_{5}$ and $\gamma_{6}$ satisfies $u_{i}^{a_{i} b_{i} c_{i}} = u^{c b a}$ (or $u^{b a c}$) if $c = \omega_{2}$, otherwise $u_{i}^{a_{i} b_{i} c_{i}} = u^{c a b}$ (or $u^{a b c}$).
\end{itemize}
There are no $\gamma_{4}, \gamma_{5}$ and $\gamma_{6}$ satisfying those conditions simultaneously.

If $c = \omega_{3}$ or $\omega_{4}$, both of $\gamma_{2}$ and $\gamma_{3}$ are of type 3.
Furthermore, $u^{c a b}$ (or $u^{a b c}$) and $u^{c b a}$ (or $u^{b a c}$) are different, although we have $u^{a b a} = u^{c a b}$ (or $u^{a b c}$) or $u^{a b a} = u^{c b a}$ (or $u^{b a c}$) in some cases.
If $u^{a b a}$, $u^{c a b}$ (or $u^{a b c}$) and $u^{c b a}$ (or $u^{b a c}$) are mutually different, in light of Lemmas \ref{lem:3-cycle_condition}--\ref{lem:f-/g-connectedness} and \ref{lem:f-connected1}--\ref{lem:f-connected3},
\begin{itemize}
\item
at least two of $\gamma_{4}, \gamma_{5}$ and $\gamma_{6}$ satisfy $u_{i}^{a_{i} b_{i} c_{i}} = u^{a b a}$,
\item
at least two of $\gamma_{4}, \gamma_{5}$ and $\gamma_{6}$ satisfy $u_{i}^{a_{i} b_{i} c_{i}} = u^{c a b}$ (or $u^{a b c}$) and
\item
at least two of $\gamma_{4}, \gamma_{5}$ and $\gamma_{6}$ satisfy $u_{i}^{a_{i} b_{i} c_{i}} = u^{c b a}$ (or $u^{b a c}$).
\end{itemize}
Otherwise, in light of Lemmas \ref{lem:3-cycle_condition}--\ref{lem:f-/g-connectedness} and \ref{lem:f-connected1}--\ref{lem:f-connected4},
\begin{itemize}
\item
at least two of $\gamma_{4}, \gamma_{5}$ and $\gamma_{6}$ satisfy $u_{i}^{a_{i} b_{i} c_{i}} = u^{c a b}$ (or $u^{a b c}$) and
\item
at least two of $\gamma_{4}, \gamma_{5}$ and $\gamma_{6}$ satisfy $u_{i}^{a_{i} b_{i} c_{i}} = u^{c b a}$ (or $u^{b a c}$),
\end{itemize}
because $\gamma_{1}$ has the opposite sign to $\gamma_{2}$ and $\gamma_{3}$.
At any rate, there are no $\gamma_{4}, \gamma_{5}$ and $\gamma_{6}$ satisfying those conditions simultaneously.

Assume that $b = [a + 3]$.
Then, we may assume that $a = \omega_{0}$, and have $b = \omega_{3}$.
In light of Lemma \ref{lem:f-connected3}, we may assume that
\[
 \gamma_{4} + \gamma_{5} + \gamma_{6} =
 \begin{cases}
  + \> \varepsilon (v; p, q, p) - \varepsilon (v; r, p, q) - \varepsilon (v; r, q, p) \ \text{or} \\
  + \> \varepsilon (v; p, q, p) - \varepsilon (v; p, q, r) - \varepsilon (v; q, p, r)
 \end{cases}
\]
with some $\varepsilon \in \{ + 1, - 1 \}$, $v \in O_{6}$, and mutually distinct $p, q, r \in O_{6}$.
If $q \neq [p + 3]$, in a similar manner as above, we have no candidates of $\gamma$.
Therefore, we have $q = [p + 3]$.
Then, it is routine to see that we have $\eta(\pi(\gamma)) = 0$ contradicting to our assumption.

\subsection{Case (iii)}
\label{subsec:proof_of_proposition_O6_III_iii}

We may assume that $\gamma_{1} + \gamma_{2} + \gamma_{3} + \gamma_{4}$ is one of the 3-chains in Lemma \ref{lem:f-connected4} up to sign, with $(a, b) = (\omega_{0}, \omega_{1})$ or $(a, b, c) = (\omega_{0}, \omega_{3}, \omega_{1})$, and
\[
 \gamma_{5} + \gamma_{6} = + \> (v; p, q, p) - (v; q, p, q)
\]
with some $v \in O_{6}$ and distinct $p, q \in O_{6}$.
If $q = [p + 3]$ and $v = p$ or $q$, in light of Lemma \ref{lem:O6_f2_special}, $\gamma_{5} + \gamma_{6}$ is a 3-cycle satisfying $\eta(\pi(\gamma_{5} + \gamma_{6})) = 0$.
Therefore, $\gamma^{\prime} = \gamma_{1} + \gamma_{2} + \gamma_{3} + \gamma_{4}$ is a 3-cycle satisfying $l(\gamma^{\prime}) = |T_{n}(\gamma^{\prime})| = 4$ and $\eta(\pi(\gamma^{\prime})) \neq 0$.
It contradicts to Proposition \ref{prop:O6_I}.
We thus have $q \neq [p + 3]$ or $v \neq p, q$.

Assume that $q \neq [p + 3]$.
Then, we have $v^{p q p} = v^{q p q}$.
Since reduced $g(\gamma_{5} + \gamma_{6})$ has six 2-terms, we have $g(\gamma_{i} + \gamma_{5} + \gamma_{6}) \neq 0$ for any $i$ ($1 \leq i \leq 4$).
Therefore, in light of Lemmas \ref{lem:3-cycle_condition} (iii), \ref{lem:f-/g-connectedness} (i) and \ref{lem:f-connected1}, we have the following cases rearranging the roles of $\gamma_{1}, \gamma_{2}, \gamma_{3}$ and $\gamma_{4}$ if necessary.
\begin{itemize}
\item[(a)]
$\{ \gamma_{1}, \gamma_{2} \}$ and $\{ \gamma_{3}, \gamma_{4}, \gamma_{5}, \gamma_{6} \}$ are respectively $g$-connected.
\item[(b)]
$\{ \gamma_{1}, \gamma_{2}, \dots, \gamma_{6} \}$ is $g$-connected.
\end{itemize}
For the subsequent arguments, we rewrite $\gamma_{i}$ as $\varepsilon_{i} (u; a_{i}, b_{i}, c_{i})$ for $1 \leq i \leq 4$, and let $U$ denote the set $\{ u^{a_{1} b_{1} c_{1}}, u^{a_{2} b_{2} c_{2}}, u^{a_{3} b_{3} c_{3}}, u^{a_{4} b_{4} c_{4}} \}$.

In the case (a), we have $|U| \leq 2$ by Lemma \ref{lem:f-/g-connectedness} (iii).
Then, we have the cases listed in Tables \ref{tab:O6_4_a}--\ref{tab:O6_4_c}.\footnote{The author enumerated the candidates of $\gamma_{1} + \gamma_{2} + \gamma_{3} + \gamma_{4}$ with the aid of a computer. The C source code for the task is available at \url{https://github.com/ayminoue/LQC/blob/main/EC4.c}.}
Since $\gamma_{1}$ and $\gamma_{2}$ are $g$-connected, in light of Lemmas \ref{lem:reverse} (ii), \ref{lem:f-/g-connectedness} and \ref{lem:f-connected2}, $\gamma_{1}$ and $\gamma_{2}$ are of type 0 or 2 and $u^{a_{1} b_{1} c_{1}} = u^{a_{2} b_{2} c_{2}}$.
This condition is satisfied only in the cases 11 and 14.
Furthermore, since reduced $g(\gamma_{5} + \gamma_{6})$ has six 2-terms and $g(\gamma_{3} + \gamma_{4} + \gamma_{5} + \gamma_{6}) = 0$, reduced $g(\gamma_{3} + \gamma_{4})$ has six 2-terms.
This condition is not satisfied in both cases (11 and 14).
Therefore, we have no candidates of $\gamma$.
\begin{table}[htbp]
 \centering
 \caption{The list of $\gamma_{1} + \gamma_{2} + \gamma_{3} + \gamma_{4}$ satisfying $u^{a_{i} b_{i} c_{i}} = u^{a_{j} b_{j} c_{j}} \neq u^{a_{k} b_{k} c_{k}} = u^{a_{l} b_{l} c_{l}}$ ($\{ i, j, k, l \} = \{ 1, 2, 3, 4 \}$). Here, values in the second column come from Lemma \ref{lem:f-connected4}.}
 \label{tab:O6_4_a}
 \begin{tabular}{|c|c|c|c|c|c|l|} \hline
  no.\ & case & $u$ & $b$ & $c$ & $d$ & \multicolumn{1}{c|}{result} \\ \hline \hline
  1 & (i) & any & $\omega_{1}$ & $\omega_{5}$ & -- & \hspace{-0.8em} \begin{tabular}{l} $u^{a b c} = u^{a c a} = u^{c a c}$, \\ $u^{c a b} = u^{b c b} = u^{c b c}$ \end{tabular} \\ \hline
  2 & (i) & $\omega_{2}$ or $\omega_{5}$ & $\omega_{3}$ & $\omega_{1}$ & -- & \hspace{-0.8em} \begin{tabular}{l} $u^{a b c} = u^{c a b}$, \\ $u^{a c a} = u^{c a c} = u^{b c b} = u^{c b c}$ \end{tabular} \hspace{-0.8em} \\ \hline
  3 & (ii) & $\omega_{0}$ or $\omega_{3}$ & $\omega_{1}$ & $\omega_{2}$ & $\omega_{3}$ & $u^{a b c} = u^{a c d}$, $u^{a b d} = u^{b c d}$ \\ \hline
  4 & (ii) & $\omega_{2}$ or $\omega_{5}$ & $\omega_{1}$ & $\omega_{5}$ & $\omega_{3}$ & $u^{a b c} = u^{a c d}$, $u^{a b d} = u^{b c d}$ \\ \hline
  5 & (iii) & any & $\omega_{3}$ & $\omega_{1}$ & $\omega_{2}$ or $\omega_{5}$ & $u^{a b c} = u^{b a c}$, $u^{a b d} = u^{b a d}$ \\ \hline
  6 & (iii) & $\omega_{0}, \omega_{2}, \omega_{3}$ or $\omega_{5}$ & $\omega_{3}$ & $\omega_{1}$ & $\omega_{4}$ & $u^{a b c} = u^{b a c}$, $u^{a b d} = u^{b a d}$ \\ \hline
  7 & (iv) & $\omega_{2}$ or $\omega_{5}$ & $\omega_{1}$ & $\omega_{2}$ & $\omega_{5}$ & $u^{c a b} = u^{d a b}$, $u^{c b a} = u^{d b a}$ \\ \hline
  8 & (iv) & $\omega_{2}$ or $\omega_{5}$ & $\omega_{1}$ & $\omega_{5}$ & $\omega_{2}$ & $u^{c a b} = u^{d a b}$, $u^{c b a} = u^{d b a}$ \\ \hline
  9 & (iv) & any & $\omega_{3}$ & $\omega_{1}$ & $\omega_{2}$ or $\omega_{5}$ & $u^{c a b} = u^{c b a}$, $u^{d a b} = u^{d b a}$ \\ \hline
  10 & (iv) & $\omega_{0}, \omega_{2}, \omega_{3}$ or $\omega_{5}$ & $\omega_{3}$ & $\omega_{1}$ & $\omega_{4}$ & $u^{c a b} = u^{c b a}$, $u^{d a b} = u^{d b a}$ \\ \hline
  11 & (v) & $\omega_{1}$ or $\omega_{4}$ & $\omega_{1}$ & $\omega_{2}$ & $\omega_{5}$ & $u^{c a b} = u^{b a d}$, $u^{c b a} = u^{a b d}$ \\ \hline
  12 & (v) & $\omega_{2}$ or $\omega_{5}$ & $\omega_{1}$ & $\omega_{3}$ & $\omega_{3}$ & $u^{c a b} = u^{b a d}$, $u^{c b a} = u^{a b d}$ \\ \hline
  13 & (v) & $\omega_{2}$ or $\omega_{5}$ & $\omega_{1}$ & $\omega_{4}$ & $\omega_{4}$ & $u^{c a b} = u^{b a d}$, $u^{c b a} = u^{a b d}$ \\ \hline
  14 & (v) & $\omega_{0}$ or $\omega_{3}$ & $\omega_{1}$ & $\omega_{5}$ & $\omega_{2}$ & $u^{c a b} = u^{b a d}$, $u^{c b a} = u^{a b d}$ \\ \hline
  15 & (v) & any & $\omega_{3}$ & $\omega_{1}$ & $\omega_{2}$ or $\omega_{5}$ & $u^{c a b} = u^{c b a}$, $u^{a b d} = u^{b a d}$ \\ \hline
  16 & (v) & $\omega_{0}, \omega_{2}, \omega_{3}$ or $\omega_{5}$ & $\omega_{3}$ & $\omega_{1}$ & $\omega_{4}$ & $u^{c a b} = u^{c b a}$, $u^{a b d} = u^{b a d}$ \\ \hline
 \end{tabular}
\end{table}
\begin{table}[htbp]
 \centering
 \caption{The list of $\gamma_{1} + \gamma_{2} + \gamma_{3} + \gamma_{4}$ satisfying $u^{a_{i} b_{i} c_{i}} = u^{a_{j} b_{j} c_{j}} = u^{a_{k} b_{k} c_{k}} \neq u^{a_{l} b_{l} c_{l}}$ ($\{ i, j, k, l \} = \{ 1, 2, 3, 4 \}$). Here, values in the second column come from Lemma \ref{lem:f-connected4}.}
 \label{tab:O6_4_b}
 \begin{tabular}{|c|c|c|c|c|c|l|} \hline
  no.\ & case & $u$ & $b$ & $c$ & $d$ & \multicolumn{1}{c|}{result} \\ \hline \hline
  17 & (ii) & $\omega_{0}$ or $\omega_{3}$ & $\omega_{1}$ & $\omega_{2}$ & $\omega_{5}$ & $u^{a b c} = u^{a b d} = u^{b c d}$ \\ \hline
  18 & (ii) & $\omega_{0}$ or $\omega_{3}$ & $\omega_{1}$ & $\omega_{5}$ & $\omega_{2}$ & $u^{a b c} = u^{a b d} = u^{b c d}$ \\ \hline
  19 & (ii) & $\omega_{0}$ or $\omega_{3}$ & $\omega_{3}$ & $\omega_{1}$ & $\omega_{2}$ or $\omega_{5}$ & $u^{a b c} = u^{a c d} = u^{b c d}$ \\ \hline
 \end{tabular}
\end{table}
\begin{table}[htbp]
 \centering
 \caption{The list of $\gamma_{1} + \gamma_{2} + \gamma_{3} + \gamma_{4}$ satisfying $u^{a_{1} b_{1} c_{1}} = u^{a_{2} b_{2} c_{2}} = u^{a_{3} b_{3} c_{3}} = u^{a_{4} b_{4} c_{4}}$. Here, values in the second column come from Lemma \ref{lem:f-connected4}.}
 \label{tab:O6_4_c}
 \begin{tabular}{|c|c|c|c|c|c|} \hline
  no.\ & case & $u$ & $b$ & $c$ & $d$ \\ \hline \hline
  20 & (iii) & $\omega_{1}$ or $\omega_{4}$ & $\omega_{3}$ & $\omega_{1}$ & $\omega_{4}$ \\ \hline
  21 & (iv) & $\omega_{1}$ or $\omega_{4}$ & $\omega_{3}$ & $\omega_{1}$ & $\omega_{4}$ \\ \hline
  22 & (v) & any & $\omega_{3}$ & $\omega_{1}$ & $\omega_{1}$ \\ \hline
  23 & (v) & $\omega_{1}$ or $\omega_{4}$ & $\omega_{3}$ & $\omega_{1}$ & $\omega_{4}$ \\ \hline
 \end{tabular}
\end{table}

In the case (b), we have $|U| = 1$ by Lemma \ref{lem:f-/g-connectedness} (iii).
Therefore, we have the cases listed in Table \ref{tab:O6_4_c}.
Since reduced $g(\gamma_{5} + \gamma_{6})$ has six 2-terms and $g(\gamma) = 0$, reduced $g(\gamma_{1} + \gamma_{2} + \gamma_{3} + \gamma_{4})$ has six 2-terms.
This condition is not satisfied in each case.
We thus have no candidates of $\gamma$.

Assume that $q = [p + 3]$ and $v \neq p, q$.
Since $v^{p q p} \neq v^{q p q}$, in light of Lemmas \ref{lem:3-cycle_condition} (iii), \ref{lem:f-/g-connectedness} and \ref{lem:f-connected1}, we have the following cases rearranging the roles of $\gamma_{1}, \gamma_{2}, \gamma_{3}$ and $\gamma_{4}$ if necessary.
\begin{itemize}
\item
$\{ \gamma_{1}, \gamma_{5}\}$, $\{ \gamma_{2}, \gamma_{6} \}$ and $\{ \gamma_{3}, \gamma_{4} \}$ are respectively $g$-connected.
\item
$\{ \gamma_{1}, \gamma_{2}, \gamma_{5} \}$ and $\{ \gamma_{3}, \gamma_{4}, \gamma_{6} \}$ are respectively $g$-connected.
\item
$\{ \gamma_{1}, \gamma_{5} \}$ (or $\{ \gamma_{1}, \gamma_{6} \}$) and $\{ \gamma_{2}, \gamma_{3}, \gamma_{4}, \gamma_{6} \}$ (or $\{ \gamma_{2}, \gamma_{3}, \gamma_{4}, \gamma_{5} \}$) are respectively $g$-connected.
\end{itemize}
The first and second cases are respectively nothing less than the reverse of the cases (i) and (ii).
Therefore, we have no candidates of $\gamma$ in those cases.

In the third case, we have $|U| = 2$ by Lemma \ref{lem:f-/g-connectedness} (iii).
Therefore, we have the cases listed in Tables \ref{tab:O6_4_a} and \ref{tab:O6_4_b}.
Since both of $\gamma_{5}$ and $\gamma_{6}$ are of type 0 and we have $g(\gamma_{1} + \gamma_{5}) = 0$ (or $g(\gamma_{1} + \gamma_{6}) = 0$), $\gamma_{1}$ is of type 0.
This condition is not satisfied in each case.
We thus have no candidates of $\gamma$.

\section{Proof of Proposition \ref{prop:O6_IV}}
\label{sec:proof_of_proposition_O6_IV}

We devote this section to prove Proposition \ref{prop:O6_IV}.
Let $\gamma_{1}, \gamma_{2}, \dots, \gamma_{7}$ be 3-terms of $(O_{6}, S)$ and $\gamma = \gamma_{1} + \gamma_{2} + \dots + \gamma_{7}$.
Assume that $\gamma$ is a 3-cycle satisfying $l(\gamma) = |T_{n}(\gamma)| = 7$ and $\eta(\pi(\gamma)) \neq 0$.
Then, for the same reason as in the proof of Proposition \ref{prop:O6_I}, we may assume that $\gamma_{1}, \gamma_{2}, \dots, \gamma_{7}$ are not $f$-connected taking reverse if necessary.
Therefore, in light of Lemmas \ref{lem:3-cycle_condition} (ii) and \ref{lem:f-connected1}, we have the following cases.
\begin{itemize}
\item[(i)]
$\gamma_{1}, \gamma_{2}, \dots, \gamma_{7}$ are divided into three sets, which respectively consist of two, two and three $f$-connected 3-terms.
\item[(ii)]
$\gamma_{1}, \gamma_{2}, \dots, \gamma_{7}$ are divided into two sets, which respectively consist of four and three $f$-connected 3-terms.
\item[(iii)]
$\gamma_{1}, \gamma_{2}, \dots, \gamma_{7}$ are divided into two sets, which respectively consist of five and two $f$-connected 3-terms.
\end{itemize}

\subsection{Case (i)}
\label{subsec:proof_of_proposition_O6_IV_i}

We may assume that $\{ \gamma_{1}, \gamma_{2} \}$, $\{ \gamma_{3}, \gamma_{4} \}$ and $\{ \gamma_{5}, \gamma_{6}, \gamma_{7} \}$ are respectively $f$-connected.
Moreover, in light of Lemma \ref{lem:f-connected3}, we may assume that
\begin{align*}
 \gamma_{5} & = \pm (w; p, q, p), \\
 \gamma_{6} & = \mp (w; r, p, q) \ \text{(or $\mp (w; p, q, r)$)}, \\
 \gamma_{7} & = \mp (w; r, q, p) \ \text{(or $\mp (w; q, p, r)$)}
\end{align*}
with some $w \in O_{6}$ and mutually distinct $p, q, r \in O_{6}$.
We note that, in light of Lemma \ref{lem:f-connected2}, $\gamma_{1}, \gamma_{2}, \gamma_{3}$ and $\gamma_{4}$ are of type 0 or 1.
Moreover, $\gamma_{1}$ is of type 0 if and only if $\gamma_{2}$ is of type 0.
Obviously, $\gamma_{3}$ and $\gamma_{4}$ are faced with the same situation.

Assume that $q \neq [p + 3]$.
Then, we may assume that $p = \omega_{0}$ and $q = \omega_{1}$.
Since $|T_{n}(\gamma)| = 7$, in a similar way to Subsection \ref{subsec:proof_of_proposition_O6_III_ii}, we have
\begin{itemize}
\item
$r = \omega_{3}$ or $\omega_{4}$,
\item
$w^{p q p} = w^{r p q}$ (or $w^{p q r}$) or $w^{p q p} = w^{r q p}$ (or $w^{q p r}$), and
\item
$\{ \gamma_{i}, \gamma_{j}, \gamma_{7} \}$ and $\{ \gamma_{k}, \gamma_{l}, \gamma_{5}, \gamma_{6} \}$, or $\{ \gamma_{i}, \gamma_{j}, \gamma_{6} \}$ and $\{ \gamma_{k}, \gamma_{l}, \gamma_{5}, \gamma_{7} \}$ are $g$- \linebreak connected for some mutually distinct $i, j, k, l \in \{ 1, 2, 3, 4 \}$, if $w^{p q p} = w^{r p q}$ (or $w^{p q r}$) or $w^{p q p} = w^{r q p}$ (or $w^{q p r}$), respectively.
\end{itemize}
We note that both of $\gamma_{6}$ and $\gamma_{7}$ are of type 3 and $w^{r p q} \neq w^{r q p}$ (or $w^{p q r} \neq w^{q p r}$).
Furthermore, in light of Lemmas \ref{lem:reverse} (ii), \ref{lem:f-/g-connectedness} (i) and \ref{lem:f-connected3}, $\gamma_{i}$ is of sign $\mp$ and of type 1 or 3, and $\gamma_{j}$ of sign $\pm$ and of type 0 or 2.
Therefore, we may assume that $\gamma_{i} = \gamma_{1}$ which is of sign $\mp$ and of type 1, and $\gamma_{j} = \gamma_{3}$ which is of sign $\pm$ and of type 0.
Let us rewrite $\gamma_{1}$ as $\mp (u; a, b, a)$.
Then, we have $\gamma_{2} = \pm (u; b, a, b)$.
Since $\gamma_{1}$ is of type 1, we have $b \neq [a + 3]$ and thus $u^{b a b} = u^{a b a} = w^{p q p}$.
Therefore, both of $\{ \gamma_{2}, \gamma_{4}, \gamma_{5}, \gamma_{6} \}$ and $\{ \gamma_{2}, \gamma_{4}, \gamma_{5}, \gamma_{7} \}$ are not $g$-connected by Lemma \ref{lem:f-/g-connectedness} (iii).
It leads to a contradiction.

Assume that $q = [p + 3]$.
Then, we may assume that $p = \omega_{0}$ and $r = \omega_{1}$, and have $q = \omega_{3}$.
Obviously, $\gamma_{5}$ is of type 0, both of $\gamma_{6}$ and $\gamma_{7}$ of type 3,
\begin{equation}
 g(\gamma_{5}) = \pm \> (w^{\omega_{0}}; \omega_{3}, \omega_{0}) \pm (w^{\omega_{0}}; \omega_{0}, \omega_{3}), \label{eq:O6_IV_i_a}
\end{equation}
and
\begin{align}
 & g(\gamma_{6} + \gamma_{7}) = \notag \\
 & \mp (w^{\omega_{1}}; \omega_{0}, \omega_{3}) \mp (w^{\omega_{1}}; \omega_{3}, \omega_{0}) \enskip (\text{or} \mp (w^{\omega_{1}}; \omega_{5}, \omega_{2}) \mp (w^{\omega_{1}}; \omega_{2}, \omega_{5})). \label{eq:O6_IV_i_b}
\end{align}
Furthermore, we have
\begin{gather*}
 w^{r p q} = w^{\omega_{1} \omega_{0} \omega_{3}} = w^{\omega_{1}} = w^{\omega_{1} \omega_{3} \omega_{0}} = w^{r q p} \\
 (\text{or} \ w^{p q r} = w^{\omega_{0} \omega_{3} \omega_{1}} = w^{\omega_{1}} = w^{\omega_{3} \omega_{0} \omega_{1}} = w^{q p r})
\end{gather*}
and
\[
 w^{p q p} = w^{\omega_{0} \omega_{3} \omega_{0}} = w^{\omega_{0}} \neq w^{\omega_{1}} = w^{r p q} \ (\text{or} \ w^{p q r}).
\]
Therefore, in light of Lemmas \ref{lem:3-cycle_condition} (iii), \ref{lem:f-/g-connectedness} and \ref{lem:f-connected1}, and the formulae (\ref{eq:O6_IV_i_a}) and (\ref{eq:O6_IV_i_b}), we have the following cases rearranging the roles of $\gamma_{1}, \gamma_{2}, \gamma_{3}$ and $\gamma_{4}$ if necessary.
\begin{itemize}
\item[(a)]
$\{ \gamma_{1}, \gamma_{5} \}$ is $g$-connected.
\item[(b)]
$\{ \gamma_{1}, \gamma_{2}, \gamma_{5}\}$ is $g$-connected.
\item[(c)]
$\{ \gamma_{1}, \gamma_{3}, \gamma_{5} \}$ is $g$-connected.
\item[(d)]
$\{ \gamma_{1}, \gamma_{2}, \gamma_{3}, \gamma_{5} \}$ is $g$-connected.
\end{itemize}
On the other hand, since $\gamma_{1}$ has the opposite sign to $\gamma_{2}$, in light of Lemmas \ref{lem:f-/g-connectedness} (i) and \ref{lem:f-connected3}, the case (b) never occurs.

In the case (a), since $g(\gamma_{1} + \gamma_{5}) = 0$, we have
\[
 \gamma_{1} = \mp \> (w; \omega_{0}, \omega_{3}, \omega_{0}) \ \text{or} \, \mp (w^{\omega_{0} \omega_{0}}; \omega_{3}, \omega_{0}, \omega_{3})
\]
by the formula (\ref{eq:O6_IV_i_a}).
Therefore, we respectively have
\[
 \gamma_{2} = \pm \> (w; \omega_{3}, \omega_{0}, \omega_{3}) \ \text{or} \, \pm (w^{\omega_{0} \omega_{0}}; \omega_{0}, \omega_{3}, \omega_{0}).
\]
Moreover, in light of Lemmas \ref{lem:3-cycle_condition} (iii), \ref{lem:f-/g-connectedness} (i) and \ref{lem:f-connected1}, we have the following cases.
\begin{itemize}
\item
$\{ \gamma_{i}, \gamma_{6}, \gamma_{7} \}$ is $g$-connected for some $2 \leq i \leq 4$.
\item
$\{ \gamma_{2}, \gamma_{3}, \gamma_{4}, \gamma_{6}, \gamma_{7} \}$ is $g$-connected.
\end{itemize}

In the former case, if $\{ \gamma_{2}, \gamma_{6}, \gamma_{7} \}$ is $g$-connected, then $\gamma^{\prime} = \gamma_{1} + \gamma_{2} + \gamma_{5} + \gamma_{6} + \gamma_{7}$ is a 3-cycle satisfying $l(\gamma^{\prime}) = |T_{n}(\gamma^{\prime})| = 5$.
Therefore, in light of Proposition \ref{prop:O6_II}, we have $\eta(\pi(\gamma^{\prime})) = 0$.
Furthermore, in light of Lemmas \ref{lem:3-cycle_condition} (iii), \ref{lem:O6_f2_usual} and \ref{lem:O6_f2_special}, $\gamma_{3} + \gamma_{4}$ is a 3-cycle satisfying $\eta(\pi(\gamma_{3} + \gamma_{4})) = 0$.
We thus have $\eta(\pi(\gamma)) = 0$ contradicting to our assumption.
If $\{ \gamma_{3}, \gamma_{6}, \gamma_{7} \}$ (or $\{ \gamma_{4}, \gamma_{6}, \gamma_{7} \}$) is $g$-connected, in light of Lemmas \ref{lem:reverse} (ii), \ref{lem:f-/g-connectedness} (i) and \ref{lem:f-connected3}, $\gamma_{3}$ (or $\gamma_{4}$) is of type 0 or 2.
Therefore, $\gamma_{3}$ (or $\gamma_{4}$) is of type 0, and thus $\gamma_{4}$ (or $\gamma_{3}$) of type 0.
Then, it is routine to check that we have $\eta(\pi(\gamma)) = 0$ contradicting to our assumption.

In the latter case, since $g(\gamma_{2} + \gamma_{3} + \gamma_{4} + \gamma_{6} + \gamma_{7}) = 0$ and reduced $g(\gamma_{2} + \gamma_{6} + \gamma_{7})$ has four 2-terms by $w^{\omega_{0} \omega_{0} \omega_{0}} = w^{\omega_{3}} \neq w^{\omega_{1}}$, reduced $g(\gamma_{3} + \gamma_{4})$ has four 2-terms.
Therefore, both of $\gamma_{3}$ and $\gamma_{4}$ are of type 0.
Then, it is routine to check that we have $\eta(\pi(\gamma)) = 0$ contradicting to our assumption.

In the case (c), $\gamma_{1}$ is of sign $\mp$ and of type 1 or 3 by Lemmas \ref{lem:reverse} (ii), \ref{lem:f-/g-connectedness} (i) and \ref{lem:f-connected3}.
Therefore, $\gamma_{1}$ is of sign $\mp$ and of type 1.
Let us rewrite $\gamma_{1}$ as $\mp (u; a, b, a)$.
Then, we have $\gamma_{2} = \pm (u; b, a, b)$.
Since $\gamma_{1}$ is of type 1, we have $b \neq [a + 3]$.
We thus have $u^{b a b} = u^{a b a} \neq w^{r p q}$ (or $w^{p q r}$).
Therefore, $\{ \gamma_{2}, \gamma_{4}, \gamma_{6}, \gamma_{7} \}$ is not $g$-connected by Lemma \ref{lem:f-/g-connectedness} (iii).
On the other hand, in light of Lemmas \ref{lem:3-cycle_condition} (iii), \ref{lem:f-/g-connectedness} (i) and \ref{lem:f-connected1}, and the formula (\ref{eq:O6_IV_i_b}), $\{ \gamma_{2}, \gamma_{4}, \gamma_{6}, \gamma_{7} \}$ is $g$-connected.
It leads to a contradiction.

In the case (d), $\{ \gamma_{4}, \gamma_{6}, \gamma_{7} \}$ is $g$-connected by Lemmas \ref{lem:3-cycle_condition} (iii), \ref{lem:f-/g-connectedness} (i) and \ref{lem:f-connected1}, and the formula (\ref{eq:O6_IV_i_b}).
Furthermore, in light of Lemmas \ref{lem:reverse} (ii) and \ref{lem:f-connected3}, $\gamma_{4}$ is of sign $\pm$ and of type 0 or 2.
Therefore, $\gamma_{4}$ is of sign $\pm$ and of type 0, and thus $\gamma_{3}$ of sign $\mp$ and of type 0.
Since $g(\gamma_{1} + \gamma_{2} + \gamma_{3} + \gamma_{5}) = 0$ and reduced $g(\gamma_{3} + \gamma_{5})$ has at most four 2-terms, reduced $g(\gamma_{1} + \gamma_{2})$ has at most four 2-terms.
Therefore, both of $\gamma_{1}$ and $\gamma_{2}$ are of type 0.
Then, it is routine to check that we have $\eta(\pi(\gamma)) = 0$ contradicting to our assumption.

\subsection{Case (ii)}
\label{subsec:proof_of_proposition_O6_IV_ii}

We may assume that $\{ \gamma_{1}, \gamma_{2}, \gamma_{3}, \gamma_{4} \}$ and $\{ \gamma_{5}, \gamma_{6}, \gamma_{7} \}$ are respectively $f$-connected.
Moreover, we may assume that $\gamma_{5}, \gamma_{6}$ and $\gamma_{7}$ are given as in the case (i).
For the subsequent arguments, we rewrite $\gamma_{i}$ as $\varepsilon_{i} (u; a_{i}, b_{i}, c_{i})$ ($1 \leq i \leq 4$).

Assume that $q \neq [p + 3]$.
Then, we may assume that $p = \omega_{0}$ and $q = \omega_{1}$.
Since $|T_{n}(\gamma)| = 7$, in a similar way to Subsection \ref{subsec:proof_of_proposition_O6_III_ii}, we have
\begin{itemize}
\item
$r = \omega_{3}$ or $\omega_{4}$,
\item
$w^{p q p} = w^{r p q}$ (or $w^{p q r}$) or $w^{p q p} = w^{r q p}$ (or $w^{q p r}$), and
\item
$\{ \gamma_{1}, \gamma_{2}, \gamma_{7} \}$ and $\{ \gamma_{3}, \gamma_{4}, \gamma_{5}, \gamma_{6} \}$, or $\{ \gamma_{1}, \gamma_{2}, \gamma_{6} \}$ and $\{ \gamma_{3}, \gamma_{4}, \gamma_{5}, \gamma_{7} \}$ are $g$-connected if $w^{p q p} = w^{r p q}$ (or $w^{p q r}$) or $w^{p q p} = w^{r q p}$ (or $w^{q p r}$), respectively.
\end{itemize}
We note that both of $\gamma_{6}$ and $\gamma_{7}$ are of type 3 and $w^{r p q} \neq w^{r q p}$ (or $w^{p q r} \neq w^{q p r}$).
In light of Lemma \ref{lem:f-/g-connectedness} (iii), we have $u^{a_{1} b_{1} c_{1}} = u^{a_{2} b_{2} c_{2}}$ and $u^{a_{3} b_{3} c_{3}} = u^{a_{4} b_{4} c_{4}}$.
Therefore, since $u^{a_{1} b_{1} c_{1}} \neq u^{a_{3} b_{3} c_{3}}$ by $w^{r p q} \neq w^{r q p}$ (or $w^{p q r} \neq w^{q p r}$), we have the cases listed in Table \ref{tab:O6_4_a}.
Since $\{ \gamma_{1}, \gamma_{2}, \gamma_{7} \}$ or $\{ \gamma_{1}, \gamma_{2}, \gamma_{6} \}$ is $g$-connected, and both of $\gamma_{6}$ and $\gamma_{7}$ are of sign $\mp$ and of type 3, in light of Lemmas \ref{lem:reverse} (ii), \ref{lem:f-/g-connectedness} (i) and \ref{lem:f-connected3}, one of $\gamma_{1}$ and $\gamma_{2}$ is of sign $\pm$ and of type 0 or 2, and the other of sign $\mp$ and of type 1 or 3.
This condition is satisfied only in the cases 7 and 8.
Furthermore, since $\{ \gamma_{3}, \gamma_{4}, \gamma_{5}, \gamma_{6} \}$ or $\{ \gamma_{3}, \gamma_{4}, \gamma_{5}, \gamma_{7} \}$ is $g$-connected, $\gamma_{5}$ of type 1, and both of $\gamma_{6}$ and $\gamma_{7}$ of type 3, in light of Lemma \ref{lem:f-connected4}, each of $\gamma_{3}$ and $\gamma_{4}$ is of type 0 or 2, or of type 1 or 3.
This condition is not satisfied in both cases (7 and 8).
We thus have no candidates of $\gamma$.

Assume that $q = [p + 3]$.
Then, we may assume that $p = \omega_{0}$ and $r = \omega_{1}$, and have $q = \omega_{3}$.
Furthermore, as we saw in the case (i), $\gamma_{5}$ is of type 0, both of $\gamma_{6}$ and $\gamma_{7}$ of type 3, $g(\gamma_{6} + \gamma_{7}) \neq 0$, $w^{r p q} = w^{r q p}$ (or $w^{p q r} = w^{q p r}$), and $w^{p q p} \neq w^{r p q}$ (or $w^{p q r}$).
In light of Lemmas \ref{lem:3-cycle_condition} (iii), \ref{lem:f-/g-connectedness} and \ref{lem:f-connected1}, and the formulae (\ref{eq:O6_IV_i_a}) and (\ref{eq:O6_IV_i_b}), we have the following cases rearranging the roles of $\gamma_{1}, \gamma_{2}, \gamma_{3}$ and $\gamma_{4}$ if necessary.
\begin{itemize}
\item[(e)]
$\{ \gamma_{1}, \gamma_{5} \}$ is $g$-connected.
\item[(f)]
$\{ \gamma_{1}, \gamma_{2}, \gamma_{5} \}$ is $g$-connected.
\item[(g)]
$\{ \gamma_{1}, \gamma_{2}, \gamma_{3}, \gamma_{5} \}$ is $g$-connected.
\end{itemize}

In the case (e), since both of $\gamma_{6}$ and $\gamma_{7}$ are of type 3, we have the following cases by Lemmas \ref{lem:3-cycle_condition} (iii), \ref{lem:f-/g-connectedness} (i), \ref{lem:f-connected1} and \ref{lem:f-connected2}.
\begin{itemize}
\item
$\{ \gamma_{2}, \gamma_{6}, \gamma_{7} \}$ and $\{ \gamma_{3}, \gamma_{4} \}$ are respectively $g$-connected.
\item
$\{ \gamma_{2}, \gamma_{3}, \gamma_{4}, \gamma_{6}, \gamma_{7} \}$ is $g$-connected.
\end{itemize}
The former case is nothing less than the reverse of the case (i).
Therefore, we have no candidates of $\gamma$ in this case.

In the latter case, we have $u^{a_{1} b_{1} c_{1}} = w^{p q p}$ and $u^{a_{2} b_{2} c_{2}} = u^{a_{3} b_{3} c_{3}} = u^{a_{4} b_{4} c_{4}} = w^{r p q}$ (or $w^{p q r}$) by Lemma \ref{lem:f-/g-connectedness} (iii).
Therefore, since $u^{a_{1} b_{1} c_{1}} \neq u^{a_{2} b_{2} c_{2}}$ by $w^{p q p} \neq w^{r p q}$ (or $w^{p q r}$), we have the cases listed in Table \ref{tab:O6_4_b}.
Since $g(\gamma_{1} + \gamma_{5}) = 0$, in light of the formula (\ref{eq:O6_IV_i_a}), $\gamma_{1}$ is of type 0.
This condition is not satisfied in each case.
Therefore, we have no candidates of $\gamma$.

In the case (f), since both of $\gamma_{6}$ and $\gamma_{7}$ are of type 3, $\{ \gamma_{3}, \gamma_{4}, \gamma_{6}, \gamma_{7} \}$ is $g$-connected by Lemmas \ref{lem:3-cycle_condition} (iii), \ref{lem:f-/g-connectedness} (i), \ref{lem:f-connected1} and \ref{lem:f-connected2}.
In light of Lemma \ref{lem:f-/g-connectedness} (iii), we have $u^{a_{1} b_{1} c_{1}} = u^{a_{2} b_{2} c_{2}} = w^{p q p}$ and $u^{a_{3} b_{3} c_{3}} = u^{a_{4} b_{4} c_{4}} = w^{r p q}$ (or $w^{p q r}$).
Therefore, since $u^{a_{1} b_{1} c_{1}} \neq u^{a_{3} b_{3} c_{3}}$ by $w^{p q p} \neq w^{r p q}$ (or $w^{p q r}$), we have the cases listed in Table \ref{tab:O6_4_a}.
Since $\gamma_{5}$ is of sign $\pm$ and of type 0, in light of Lemmas \ref{lem:reverse} (ii), \ref{lem:f-/g-connectedness} and \ref{lem:f-connected3}, each of $\gamma_{1}$ and $\gamma_{2}$ is of sign $\mp$ and of type 1 or 3.
Furthermore, since both of $\gamma_{6}$ and $\gamma_{7}$ are of sign $\mp$ and of type 3, in light of Lemma \ref{lem:f-connected4}, each of $\gamma_{3}$ and $\gamma_{4}$ is of sign $\pm$ and of type 1 or 3.
Those conditions are satisfied only in the cases 5, 6, 9, 10, 15 and 16.
In those cases, we have $\eta(\pi(\gamma)) = 0$ by Lemma \ref{lem:O6_weight}, contradicting to our assumption.

In the case (g), $\{ \gamma_{4}, \gamma_{6}, \gamma_{7} \}$ is $g$-connected by Lemmas \ref{lem:3-cycle_condition} (iii), \ref{lem:f-/g-connectedness} (i) and \ref{lem:f-connected1}.
In light of Lemma \ref{lem:f-/g-connectedness} (iii), we have $u^{a_{1} b_{1} c_{1}} = u^{a_{2} b_{2} c_{2}} = u^{a_{3} b_{3} c_{3}} = w^{p q p}$ and $u^{a_{4} b_{4} c_{4}} = w^{r p q}$ (or $w^{p q r}$).
Therefore, since $u^{a_{1} b_{1} c_{1}} \neq u^{a_{4} b_{4} c_{4}}$ by $w^{p q p} \neq w^{r p q}$ (or $w^{p q r}$), we have the cases listed in Table \ref{tab:O6_4_b}.
Since both of $\gamma_{6}$ and $\gamma_{7}$ are of type 3, in light of Lemma \ref{lem:f-connected3}, $\gamma_{4}$ is of type 0 or 2.
This condition is not satisfied in each case.
We thus have no candidates of $\gamma$.

\subsection{Case (iii)}
\label{subsec:proof_of_proposition_O6_IV_iii}

We may assume that $\gamma^{\prime \prime} = \gamma_{1} + \gamma_{2} + \gamma_{3} + \gamma_{4} + \gamma_{5}$ is one of the 3-chains in Lemma \ref{lem:f-connected5} up to sign with $(a, b) = (\omega_{0}, \omega_{1})$ or $(a, b, c) = (\omega_{0}, \omega_{3}, \omega_{1})$, and
\[
 \gamma_{6} + \gamma_{7} = + \> (v; p, q, p) - (v; q, p, q)
\]
with some $v \in O_{6}$ and distinct $p, q \in O_{6}$.
If $q = [p + 3]$ and $v = p$ or $q$, in light of Lemma \ref{lem:O6_f2_special}, $\gamma_{6} + \gamma_{7}$ is a 3-cycle satisfying $\eta(\pi(\gamma_{6} + \gamma_{7})) = 0$.
Therefore, $\gamma^{\prime \prime}$ is a 3-cycle satisfying $l(\gamma^{\prime \prime}) = |T_{n}(\gamma^{\prime \prime})| = 5$ and $\eta(\pi(\gamma^{\prime \prime})) \neq 0$.
It contradicts to Proposition \ref{prop:O6_II}.
We thus have $q \neq [p + 3]$ or $v \neq p, q$.
For the subsequent arguments, we rewrite $\gamma_{i}$ as $\varepsilon_{i} (u; a_{i}, b_{i}, c_{i})$ for $1 \leq i \leq 5$, and let $U$ denote the set $\{ u^{a_{1} b_{1} c_{1}}, u^{a_{2} b_{2} c_{2}}, u^{a_{3} b_{3} c_{3}}, u^{a_{4} b_{4} c_{4}}, u^{a_{5} b_{5} c_{5}} \}$.

Assume that $q \neq [p + 3]$.
Then, we have $v^{p q p} = v^{q p q}$.
Since reduced $g(\gamma_{6} + \gamma_{7})$ has six 2-terms, we have $g(\gamma_{i} + \gamma_{6} + \gamma_{7}) \neq 0$ for any $i$ ($1 \leq i \leq 5$).
Therefore, since both of $\gamma_{6}$ and $\gamma_{7}$ are of type 1, in light of Lemmas \ref{lem:3-cycle_condition} (iii), \ref{lem:reverse} (ii), \ref{lem:f-/g-connectedness} (i), \ref{lem:f-connected1} and \ref{lem:f-connected2}, we have the following cases rearranging the roles of $\gamma_{1}, \gamma_{2}, \gamma_{3}, \gamma_{4}$ and $\gamma_{5}$ if necessary.
\begin{itemize}
\item
$\{ \gamma_{1}, \gamma_{2}, \gamma_{6} \}$ and $\{ \gamma_{3}, \gamma_{4}, \gamma_{5}, \gamma_{7} \}$ are respectively $g$-connected.
\item
$\{ \gamma_{1}, \gamma_{2}, \gamma_{6}, \gamma_{7} \}$ and $\{ \gamma_{3}, \gamma_{4}, \gamma_{5} \}$ are respectively $g$-connected.
\item
$\{ \gamma_{1}, \gamma_{2}, \gamma_{3}, \gamma_{6}, \gamma_{7} \}$ and $\{ \gamma_{4}, \gamma_{5} \}$ are respectively $g$-connected.
\item
$\{ \gamma_{1}, \gamma_{2}, \dots, \gamma_{7} \}$ is $g$-connected.
\end{itemize}
The first and second cases are nothing less than the reverse of the case (ii).
We thus have no candidates of $\gamma$ in those cases.

In the third case, we have $|U| \leq 2$ by Lemma \ref{lem:f-/g-connectedness} (iii).
Therefore, we have the cases listed in Tables \ref{tab:O6_5_a} and \ref{tab:O6_5_b}.\footnote{The author enumerated the candidates of $\gamma_{1} + \gamma_{2} + \gamma_{3} + \gamma_{4} + \gamma_{5}$ with the aid of a computer. The C source code for the task is available at \url{https://github.com/ayminoue/LQC/blob/main/EC5.c}.}
Since $\{ \gamma_{4}, \gamma_{5} \}$ is $g$-connected, in light of Lemmas \ref{lem:f-/g-connectedness} (i) and \ref{lem:f-connected2}, each of $\gamma_{4}$ and $\gamma_{5}$ is of type 0 or 2.
Furthermore, since reduced $g(\gamma_{6} + \gamma_{7})$ has three positive 2-terms and three negative 2-terms, one of $\gamma_{1}, \gamma_{2}, \gamma_{3}$ is positive (or negative) and of type 0 or 2, and each of the others is negative (or positive) and of type 1 or 3 (count the number of positive or negative 2-terms of $g(\gamma_{1} + \gamma_{2} + \gamma_{3})$).
Therefore, three of $\gamma_{1}, \gamma_{2}, \gamma_{3}, \gamma_{4}$ and $\gamma_{5}$ are of type 0 or 2.
This condition is not satisfied for each case.
We thus have no candidates of $\gamma$.
\begin{table}[htbp]
 \centering
 \caption{The list of $\gamma_{1} + \gamma_{2} + \gamma_{3} + \gamma_{4} + \gamma_{5}$ satisfying $u^{a_{i} b_{i} c_{i}} = u^{a_{j} b_{j} c_{j}} = u^{a_{k} b_{k} c_{k}} \neq u^{a_{l} b_{l} c_{l}} = u^{a_{m} b_{m} c_{m}}$ ($\{ i, j, k, l, m \} = \{ 1, 2, 3, 4, 5 \}$). Here, values in the second column come from Lemma \ref{lem:f-connected5}.}
 \label{tab:O6_5_a}
 \begin{tabular}{|c|c|c|c|c|c|l|} \hline
  no.\ & case & $u$ & $b$ & $c$ & $d$ & \multicolumn{1}{c|}{result} \\ \hline \hline
  1 & (ii) & $\omega_{0}$ or $\omega_{3}$ & $\omega_{1}$ & $\omega_{5}$ & $\omega_{3}$ & $u^{a b c} = u^{d b c}$, $u^{b d b} = u^{d b d} = u^{a b d} = u^{a d c}$ \\ \hline
  2 & (iii) & $\omega_{1}$ or $\omega_{4}$ & $\omega_{1}$ & $\omega_{5}$ & $\omega_{2}$ & $u^{a b c} = u^{a d a} = u^{d a d}$, $u^{a d c} = u^{d a b} = u^{d b c}$ \\ \hline
  3 & (iv) & $\omega_{1}$ or $\omega_{4}$ & $\omega_{1}$ & $\omega_{5}$ & $\omega_{3}$ & $u^{a b c} = u^{c d c} = u^{d c d}$, $u^{a d c} = u^{a b d} = u^{b c d}$ \\ \hline
  4 & (v) & $\omega_{0}$ or $\omega_{3}$ & $\omega_{1}$ & $\omega_{2}$ & $\omega_{3}$ & $u^{a b c} = u^{d a b} = u^{d b c}$, $u^{a c a} = u^{c a c} = u^{d c a}$ \\ \hline
  5 & (v) & $\omega_{1}$ or $\omega_{4}$ & $\omega_{1}$ & $\omega_{2}$ & $\omega_{3}$ & $u^{a b c} = u^{d a b} = u^{d c a}$, $u^{a c a} = u^{c a c} = u^{d b c}$ \\ \hline
  6 & (v) & $\omega_{2}$ or $\omega_{5}$ & $\omega_{1}$ & $\omega_{2}$ & $\omega_{3}$ & $u^{a b c} = u^{a c a} = u^{c a c} = u^{d a b}$, $u^{d b c} = u^{d c a}$ \\ \hline
  7 & (v) & any & $\omega_{1}$ & $\omega_{5}$ & \hspace{-0.8em} \begin{tabular}{c} $\omega_{2}$ or \\ $\omega_{4}$ \end{tabular} \hspace{-0.8em} & $u^{a b c} = u^{a c a} = u^{c a c}$, $u^{d a b} = u^{d b c} = u^{d c a}$ \\ \hline
  8 & (v) & \hspace{-0.8em} \begin{tabular}{c} $\omega_{1}$, $\omega_{2}$, \\ $\omega_{4}$ or $\omega_{5}$ \end{tabular} \hspace{-0.8em} & $\omega_{1}$ & $\omega_{5}$ & $\omega_{3}$ & $u^{a b c} = u^{a c a} = u^{c a c}$, $u^{d a b} = u^{d b c} = u^{d c a}$ \\ \hline
  9 & (vi) & $\omega_{0}$ or $\omega_{3}$ & $\omega_{1}$ & $\omega_{2}$ & $\omega_{5}$ & $u^{a b c} = u^{a b d} = u^{b c d}$, $u^{a c a} = u^{c a c} = u^{c a d}$ \\ \hline
  10 & (vi) & $\omega_{1}$ or $\omega_{4}$ & $\omega_{1}$ & $\omega_{2}$ & $\omega_{5}$ & $u^{a b c} = u^{b c d} = u^{c a d}$, $u^{a c a} = u^{c a c} = u^{a b d}$ \\ \hline
  11 & (vi) & $\omega_{2}$ or $\omega_{5}$ & $\omega_{1}$ & $\omega_{2}$ & $\omega_{5}$ & $u^{a b c} = u^{a c a} = u^{c a c} = u^{b c d}$, $u^{a b d} = u^{c a d}$ \\ \hline
  12 & (vi) & \hspace{-0.8em} \begin{tabular}{c} $\omega_{1}$, $\omega_{2}$, \\ $\omega_{4}$ or $\omega_{5}$ \end{tabular} \hspace{-0.8em} & $\omega_{1}$ & $\omega_{5}$ & $\omega_{2}$ & $u^{a b c} = u^{a c a} = u^{c a c}$, $u^{a b d} = u^{b c d} = u^{c a d}$ \\ \hline
  13 & (vi) & any & $\omega_{1}$ & $\omega_{5}$ & \hspace{-0.8em} \begin{tabular}{c} $\omega_{3}$ or \\ $\omega_{4}$ \end{tabular} \hspace{-0.8em} & $u^{a b c} = u^{a c a} = u^{c a c}$, $u^{a b d} = u^{b c d} = u^{c a d}$ \\ \hline
  14 & (vii) & $\omega_{0}$ or $\omega_{3}$ & $\omega_{1}$ & $\omega_{4}$ & $\omega_{3}$ & $u^{a b c} = u^{d b c} = u^{d c b}$, $u^{a c a} = u^{c a c} = u^{c a b}$ \\ \hline
  15 & (viii) & $\omega_{2}$ or $\omega_{5}$ & $\omega_{1}$ & $\omega_{2}$ & $\omega_{3}$ & $u^{a b c} = u^{a c a} = u^{c a c} = u^{b a d}$, $u^{a b d} = u^{b c a}$ \\ \hline
  16 & (ix) & $\omega_{2}$ or $\omega_{5}$ & $\omega_{1}$ & $\omega_{2}$ & $\omega_{3}$ & $u^{a b c} = u^{a c a} = u^{c a c} = u^{d a b}$, $u^{b c a} = u^{d b a}$ \\ \hline
  17 & (ix) & $\omega_{0}$ or $\omega_{3}$ & $\omega_{1}$ & $\omega_{4}$ & $\omega_{4}$ & $u^{a b c} = u^{b c a} = u^{d b a}$, $u^{a c a} = u^{c a c} = u^{d a b}$ \\ \hline
  18 & (ix) & $\omega_{0}$ or $\omega_{3}$ & $\omega_{1}$ & $\omega_{5}$ & $\omega_{3}$ & $u^{a b c} = u^{a c a} = u^{c a c} = u^{d a b}$, $u^{b c a} = u^{d b a}$ \\ \hline
  19 & (ix) & $\omega_{0}$ or $\omega_{3}$ & $\omega_{3}$ & $\omega_{1}$ & $\omega_{1}$ & $u^{a b c} = u^{d a b} = u^{d b a}$, $u^{a c a} = u^{c a c} = u^{b c a}$ \\ \hline
  20 & (x) & $\omega_{1}$ or $\omega_{4}$ & $\omega_{1}$ & $\omega_{4}$ & $\omega_{4}$ & $u^{c b a} = u^{a b d} = u^{a c b}$, $u^{a c a} = u^{c a c} = u^{b a d}$ \\ \hline
  21 & (x) & $\omega_{0}$ or $\omega_{3}$ & $\omega_{1}$ & $\omega_{5}$ & $\omega_{3}$ & $u^{c b a} = u^{a c a} = u^{c a c} = u^{b a d}$, $u^{a b d} = u^{a c b}$ \\ \hline
  22 & (x) & $\omega_{1}$ or $\omega_{4}$ & $\omega_{3}$ & $\omega_{1}$ & \hspace{-0.8em} \begin{tabular}{c} $\omega_{1}$ or \\ $\omega_{4}$ \end{tabular} \hspace{-0.8em} & $u^{c b a} = u^{a b d} = u^{b a d}$, $u^{a c a} = u^{c a c} = u^{a c b}$ \\ \hline
 \end{tabular}
\end{table}
\begin{table}[htbp]
 \centering
 \caption{The list of $\gamma_{1} + \gamma_{2} + \gamma_{3} + \gamma_{4} + \gamma_{5}$ satisfying $u^{a_{i} b_{i} c_{i}} = u^{a_{j} b_{j} c_{j}} = u^{a_{k} b_{k} c_{k}} = u^{a_{l} b_{l} c_{l}} \neq u^{a_{m} b_{m} c_{m}}$ ($\{ i, j, k, l, m \} = \{ 1, 2, 3, 4, 5 \}$). Here, values in the second column come from Lemma \ref{lem:f-connected5}.}
 \label{tab:O6_5_b}
 \begin{tabular}{|c|c|c|c|c|c|l|} \hline
  no.\ & case & $u$ & $b$ & $c$ & $d$ & \multicolumn{1}{c|}{result} \\ \hline \hline
  23 & (vii) & $\omega_{0}$ or $\omega_{3}$ & $\omega_{1}$ & $\omega_{4}$ & $\omega_{5}$ & $u^{a c a} = u^{c a c} = u^{c a b} = u^{d b c} = u^{d c b}$ \\ \hline
  24 & (viii) & $\omega_{0}$ or $\omega_{3}$ & $\omega_{3}$ & $\omega_{1}$ & $\omega_{2}$ & $u^{a c a} = u^{c a c} = u^{a b d} = u^{b a d} = u^{b c a}$ \\ \hline
  25 & (ix) & $\omega_{0}$ or $\omega_{3}$ & $\omega_{3}$ & $\omega_{1}$ & $\omega_{2}$ & $u^{a c a} = u^{c a c} = u^{b c a} = u^{d a b} = u^{d b a}$ \\ \hline
  26 & (x) & $\omega_{1}$ or $\omega_{4}$ & $\omega_{3}$ & $\omega_{1}$ & $\omega_{5}$ & $u^{a c a} = u^{c a c} = u^{a b d} = u^{a c b} = u^{b a d}$ \\ \hline
 \end{tabular}
\end{table}

Since there are no candidates of $\gamma_{1} + \gamma_{2} + \gamma_{3} + \gamma_{4} + \gamma_{5}$ satisfying $|U| = 1$, the fourth case never occurs.

Assume that $q = [p + 3]$ and $v \neq p, q$.
Since we have $v^{p q p} \neq v^{q p q}$, in light of Lemmas \ref{lem:3-cycle_condition} (iii), \ref{lem:f-/g-connectedness} and \ref{lem:f-connected1}, we have the following cases rearranging the roles of $\gamma_{1}, \gamma_{2}, \gamma_{3}, \gamma_{4}$ and $\gamma_{5}$, and that of $\gamma_{6}$ and $\gamma_{7}$ if necessary.
\begin{itemize}
\item
$\{ \gamma_{1}, \gamma_{6} \}$, $\{ \gamma_{2}, \gamma_{7} \}$ and $\{ \gamma_{3}, \gamma_{4}, \gamma_{5} \}$ are respectively $g$-connected.
\item
$\{ \gamma_{1}, \gamma_{6} \}$, $\{ \gamma_{2}, \gamma_{3}, \gamma_{7} \}$ and $\{ \gamma_{4}, \gamma_{5} \}$ are respectively $g$-connected.
\item
$\{ \gamma_{1}, \gamma_{6} \}$ and $\{ \gamma_{2}, \gamma_{3}, \gamma_{4}, \gamma_{5}, \gamma_{7} \}$ are respectively $g$-connected.
\item
$\{ \gamma_{1}, \gamma_{2}, \gamma_{6} \}$ and $\{ \gamma_{3}, \gamma_{4}, \gamma_{5}, \gamma_{7} \}$ are respectively $g$-connected.
\end{itemize}
The first and second cases are nothing less than the reverse of the case (i).
Moreover, the fourth case is nothing less than the reverse of the case (ii).
Therefore, we have no candidates of $\gamma$ in those cases.

In the third case, we have $|U| = 2$ by Lemma \ref{lem:f-/g-connectedness} (iii).
Therefore, we have the cases listed in Tables \ref{tab:O6_5_a} and \ref{tab:O6_5_b}.
Since $\gamma_{6}$ is of type 0 and $g(\gamma_{1} + \gamma_{6}) = 0$, $\gamma_{1}$ is of type 0.
This condition is not satisfied in each case.
We thus have no candidates of $\gamma$.

\section{Proof of Proposition \ref{prop:O6_VIII}}
\label{sec:proof_of_proposition_O6_VIII}

We devote this section to prove Proposition \ref{prop:O6_VIII}.
Let $\gamma_{1}, \gamma_{2}, \dots, \gamma_{7}$ be 3-terms of $(O_{6}, S)$ and $\gamma = \gamma_{1} + \gamma_{2} + \dots + \gamma_{7}$.
Assume that $\gamma$ is a 3-cycle satisfying $l(\gamma) = 7$, $|T_{n}(\gamma)| = 3$, $|T_{n + 1}(\gamma)| = 4$ and $\eta(\pi(\gamma)) \neq 0$.
Then, for the same reason as in the proof of Proposition \ref{prop:O6_VII}, we may assume that
\begin{itemize}
\item[(i)]
$\gamma_{1} + \gamma_{2} + \gamma_{3} =  + \> (n, u; \omega_{0}, \omega_{1}, \omega_{0}) - (n, u; c, \omega_{0}, \omega_{1}) - (n, u; c, \omega_{1}, \omega_{0})$, or
\item[(ii)]
$\gamma_{1} + \gamma_{2} + \gamma_{3} =  + \> (n, u; \omega_{0}, \omega_{1}, \omega_{0}) - (n, u; \omega_{0}, \omega_{1}, c) - (n, u; \omega_{1}, \omega_{0}, c)$
\end{itemize}
up to sign, with some $u, c \in O_{6}$ ($c \neq \omega_{0}, \omega_{1}$).

\subsection{Case (i)}

If $c = \omega_{5}$, we let
\begin{align*}
 & \gamma_{4}^{\prime} + \gamma_{5}^{\prime} + \gamma_{6}^{\prime} = \\
 & + (n + 1, u^{\omega_{0}}; \omega_{1}, \omega_{0}, \omega_{2}) - (n + 1, u^{\omega_{1}}; \omega_{3}, \omega_{5}, \omega_{0}) - (n + 1, u^{\omega_{5}}; \omega_{0}, \omega_{1}, \omega_{0}).
\end{align*}
Then, as we saw in the proof of Proposition \ref{prop:O6_VII}, we have
\[
 \gamma^{\prime}
 = \gamma_{1} + \gamma_{2} + \gamma_{3} + \gamma_{4}^{\prime} + \gamma_{5}^{\prime} + \gamma_{6}^{\prime}
 = - \> \partial (n, u; \omega_{5}, \omega_{0}, \omega_{1}, \omega_{0}).
\]
Since $\gamma$ is a 3-cycle, $\gamma - \gamma^{\prime}$ is a 3-cycle.
Moreover, since $l(\gamma - \gamma^{\prime}) = |T_{n + 1}(\gamma - \gamma^{\prime})| \leq 7$, in light of Propositions \ref{prop:O6_list} and \ref{prop:O6_I}--\ref{prop:O6_IV}, we have $\eta(\pi(\gamma - \gamma^{\prime})) = 0$.
We thus have $\eta(\pi(\gamma)) = \eta(\pi(\gamma^{\prime})) = 0$ contradicting to our assumption.
Therefore, we assume that $c \neq \omega_{5}$.

Consider the following four 3-terms
\begin{align*}
 \delta_{4} & = + (n + 1, u^{\omega_{0}}; c^{\omega_{0}}, \omega_{1}, \omega_{0}), &
 \delta_{5} & = + (n + 1, u^{\omega_{0}}; c^{\omega_{0}}, \omega_{0}, \omega_{2}), \\
 \delta_{6} & = - (n + 1, u^{\omega_{1}}; c^{\omega_{1}}, \omega_{5}, \omega_{0}), &
 \delta_{7} & = - (n + 1, u^{c}; \omega_{0}, \omega_{1}, \omega_{0}).
\end{align*}
We note that if
\begin{equation}
 \gamma_{4} + \gamma_{5} + \gamma_{6} + \gamma_{7} = \delta_{4} + \delta_{5} + \delta_{6} + \delta_{7}, \label{eq:O6_VIII_i}
\end{equation}
then we have $\gamma = - \partial (n, u; c, \omega_{0}, \omega_{1}, \omega_{0})$ contradicting to our assumption.

In the case other than (a) or (b) introduced in the proof of Lemma \ref{lem:O6_f3_usual}, since $u^{\omega_{0}}$, $u^{\omega_{1}}$ and $u^{c}$ are mutually different, reduced $g(\gamma_{1} + \gamma_{2} + \gamma_{3})$ has four 2-terms of index $u^{\omega_{0}}$, and $g(\gamma_{1} + \gamma_{2} + \gamma_{3}) = f(\gamma_{4} + \gamma_{5} + \gamma_{6} + \gamma_{7})$ by Lemma \ref{lem:3-cycle_condition} (i), we may assume that
\begin{align*}
 f(\gamma_{4} + \gamma_{5}) & = f(\delta_{4} + \delta_{5}), &
 \gamma_{6} & = \delta_{6}, &
 \gamma_{7} & = - \> \langle n + 1, u^{c}; \omega_{0}, \omega_{1} \rangle.
\end{align*}
We note that if
\begin{equation}
 \gamma_{4} + \gamma_{5} + \gamma_{6} + \gamma_{7} = \delta_{5} + \delta_{6} + \delta_{7} - (n + 1, u^{c}; \omega_{1}, \omega_{0}, \omega_{1}), \label{eq:O6_VIII_i'}
\end{equation}
we have $g(\gamma_{4} + \gamma_{5} + \gamma_{6} + \gamma_{7}) \neq 0$ contradicting to Lemma \ref{lem:3-cycle_condition} (iii).
Since $f(\gamma_{4} + \gamma_{5}) = f(\delta_{4} + \delta_{5})$, we have the following cases other than $\gamma_{4} + \gamma_{5} = \delta_{4} + \delta_{5}$.
\begin{itemize}
\item[(A)]
$l(\gamma_{4} + \gamma_{5} - \delta_{4} - \delta_{5}) = l(\gamma_{i} - \delta_{j}) = 2$ and $\{ \gamma_{i}, - \delta_{j} \}$ is $f$-connected for some $4 \leq i, j \leq 5$.
\item[(B)]
$\{ \gamma_{4}, \gamma_{5}, - \delta_{4}, - \delta_{5} \}$ is $f$-connected.
\end{itemize}
In the case (A), since $c^{\omega_0} \neq \omega_{0}, \omega_{2}$, we have no candidates of $\gamma_{4} + \gamma_{5}$ by Lemma \ref{lem:f-connected2}.
In the case (B), we have
\[
 \gamma_{4} + \gamma_{5} = + \> (n + 1, u^{\omega_{0}}; c^{\omega_{0}}, \omega_{1}, \omega_{2}) + (n + 1, u^{\omega_{0}}; \omega_{1}, \omega_{0}, \omega_{2})
\]
by Lemma \ref{lem:f-connected4}.
We note that reduced $g(\gamma_{4} + \gamma_{5} + \gamma_{6} + \gamma_{7})$ has at most five negative 2-terms, even though it has six positive 2-terms.
Therefore, we have $g(\gamma_{4} + \gamma_{5} + \gamma_{6} + \gamma_{7}) \neq 0$ contradicting to Lemma \ref{lem:3-cycle_condition} (iii).

In the case (b) introduced in the proof of Lemma \ref{lem:O6_f3_usual}, since we have $g(\gamma_{1} + \gamma_{2} + \gamma_{3}) = f(\gamma_{4} + \gamma_{5} + \gamma_{6} + \gamma_{7})$, we may assume that
\begin{align*}
 f(\gamma_{4} + \gamma_{5}) & = f(\delta_{4} + \delta_{5}), &
 f(\gamma_{6} + \gamma_{7}) & = f(\delta_{6} + \delta_{7})
\end{align*}
as mentioned in the proof of Lemma \ref{lem:O6_f3_usual}.
Therefore, in light of the above argument, we have the following case:
\begin{itemize}
\item[(C)]
$\gamma_{4} + \gamma_{5} = \delta_{4} + \delta_{5}$ or
\item[(D)]
$\gamma_{4} + \gamma_{5} = + \> (n + 1, u^{\omega_{0}}; \omega_{5}, \omega_{1}, \omega_{2}) + (n + 1, u^{\omega_{0}}; \omega_{1}, \omega_{0}, \omega_{2})$,
\end{itemize}
and
\begin{itemize}
\item[(E)]
$\gamma_{6} + \gamma_{7} = \delta_{6} + \delta_{7}$,
\item[(F)]
$l(\gamma_{6} + \gamma_{7} - \delta_{6} - \delta_{7}) = l(\gamma_{i} - \delta_{j}) = 2$ and $\{ \gamma_{i}, - \delta_{j} \}$ is $f$-connected for some $6 \leq i, j \leq 7$, or
\item[(G)]
$\{ \gamma_{6}, \gamma_{7}, - \delta_{6}, - \delta_{7} \}$ is $f$-connected.
\end{itemize}
In the case (G), we have no candidates of $\gamma_{6} + \gamma_{7}$ by Lemma \ref{lem:f-connected4}.
Therefore, it is sufficient to consider the cases (C)-(E), (C)-(F), (D)-(E) and (D)-(F).
The first and second cases are respectively nothing less than the cases (\ref{eq:O6_VIII_i}) and (\ref{eq:O6_VIII_i'}), because we have
\[
 \gamma_{6} + \gamma_{7} = \delta_{6} - (n + 1, u; \omega_{1}, \omega_{0}, \omega_{1})
\]
by Lemma \ref{lem:f-connected2} in the case (F).
In the other cases, it is easy to see that we have $g(\gamma_{4} + \gamma_{5} + \gamma_{6} + \gamma_{7}) \neq 0$ contradicting to Lemma \ref{lem:3-cycle_condition} (iii).

In the case (a) introduced in the proof of Lemma \ref{lem:O6_f3_usual}, since we have $g(\gamma_{1} + \gamma_{2} + \gamma_{3}) = f(\gamma_{4} + \gamma_{5} + \gamma_{6} + \gamma_{7})$, we may assume that
\begin{itemize}
\item[($\alpha$)]
$f(\gamma_{4} + \gamma_{5}) = f(\delta_{4} + \delta_{5} + \delta_{7})$ and $f(\gamma_{6} + \gamma_{7}) = f(\delta_{6})$, or
\item[($\beta$)]
$f(\gamma_{4} + \gamma_{5} + \gamma_{6}) = f(\delta_{4} + \delta_{5} + \delta_{7})$ and $\gamma_{7} = \delta_{6}$
\end{itemize}
as mentioned in the proof of Lemma \ref{lem:O6_f3_usual}.

In the case ($\alpha$), we have the following case by Lemma \ref{lem:f-connected1}:
\begin{itemize}
\item[(H)]
$l(\gamma_{4} + \gamma_{5} - \delta_{4} - \delta_{5} - \delta_{7}) = l(\gamma_{i} - \delta_{j} - \delta_{k}) = 3$ and $\{ \gamma_{i}, - \delta_{j}, - \delta_{k} \}$ is $f$-connected for some $4 \leq i \leq 5$ and $j, k \in \{ 4, 5, 7 \}$ ($j \neq k$), or
\item[(I)]
$\{ \gamma_{4}, \gamma_{5}, - \delta_{4}, - \delta_{5}, - \delta_{7} \}$ is $f$-connected,
\end{itemize}
and
\begin{itemize}
\item[(J)]
$\{ \gamma_{6}, \gamma_{7}, - \delta_{6} \}$ is $f$-connected.
\end{itemize}
In the case (H), in light of Lemma \ref{lem:f-connected3}, we have
\[
 \gamma_{4} + \gamma_{5} = + \> (n + 1, u; \omega_{3}, \omega_{0}, \omega_{2}) - (n + 1, u; \omega_{3}, \omega_{0}, \omega_{1}).
\]
In the case (I), in light of Lemma \ref{lem:f-connected5}, we have
\begin{itemize}
\item
$\gamma_{4} + \gamma_{5} = + \> (n + 1, u; \omega_{3}, \omega_{1}, \omega_{2}) - (n + 1, u; \omega_{0}, \omega_{1}, \omega_{2})$,
\item
$\gamma_{4} + \gamma_{5} = + \> (n + 1, u; \omega_{0}, \omega_{2}, \omega_{1}) - (n + 1, u; \omega_{3}, \omega_{2}, \omega_{1})$, or
\item
$\gamma_{4} + \gamma_{5} = + \> (n + 1, u; \omega_{0}, \omega_{3}, \omega_{1}) - (n + 1, u; \omega_{0}, \omega_{3}, \omega_{2})$.
\end{itemize}
In the case (J), in light of Lemma \ref{lem:f-connected3}, we have
\begin{itemize}
\item
$\gamma_{6} + \gamma_{7} = - \> \langle n + 1, u^{\omega_{1}}; \omega_{0}, \omega_{5} \rangle + (n + 1, u^{\omega_{1}}; \omega_{2}, \omega_{0}, \omega_{5})$ or
\item
$\gamma_{6} + \gamma_{7} = - \> \langle n + 1, u^{\omega_{1}}; \omega_{2}, \omega_{5} \rangle + (n + 1, u^{\omega_{1}}; \omega_{5}, \omega_{2}, \omega_{0})$.
\end{itemize}
Therefore, in the cases (H)-(J) and (I)-(J), we have $g(\gamma_{4} + \gamma_{5} + \gamma_{6} + \gamma_{7}) \neq 0$ contradicting to Lemma \ref{lem:3-cycle_condition} (iii).

In the case ($\beta$), let $N_{ij}^{\varepsilon}$ denote the number of 3-terms among $\gamma_{4}, \gamma_{5}$ and $\gamma_{6}$ whose types are $i$ or $j$ and signs $\varepsilon$.
Then, since reduced $f(\delta_{4} + \delta_{5} + \delta_{7})$ has two positive 2-terms and two negative 2-terms (see the formula (\ref{eq:O6_f3_usual_i_a})), and $f(\gamma_{4} + \gamma_{5} + \gamma_{6}) = f(\delta_{4} + \delta_{5} + \delta_{7})$, we have
\[
 2 (N_{01}^{+} - N_{01}^{-}) + (N_{23}^{+} - N_{23}^{-}) = 0.
\]
Thus, remarking that $N_{01}^{+} + N_{01}^{-} + N_{23}^{+} + N_{23}^{-} = 3$, we have
\[
 (N_{01}^{+}, N_{01}^{-}, N_{23}^{+}, N_{23}^{-}) = (1, 0, 0, 2) \ \text{or} \ (0, 1, 2, 0).
\]
Therefore, we may assume that
\begin{align}
 & \gamma_{4} + \gamma_{5} + \gamma_{6} = \notag \\
 & + \varepsilon (n + 1, u; a_{4}, b_{4}, a_{4}) - \varepsilon (n + 1, u; a_{5}, b_{5}, c_{5}) - \varepsilon (n + 1, u; a_{6}, b_{6}, c_{6}) \label{eq:O6_VIII_beta}
\end{align}
with some $\varepsilon \in \{ + 1, - 1\}$, distinct $a_{4}, b_{4} \in O_{6}$, and mutually distinct $a_{i}, b_{i}, c_{i} \in O_{6}$ ($i = 5, 6$).
Since
\begin{align}
 f(\gamma_{4} + \gamma_{5} + \gamma_{6})
 & = - \> \varepsilon (b_{4}, a_{4}) - \varepsilon (a_{4}, b_{4}) \notag \\
 & \phantom{=} \ + \varepsilon (b_{5}, c_{5}) - \varepsilon (a_{5}, c_{5}) + \varepsilon (a_{5}, b_{5}) \notag \\
 & \phantom{=} \ + \varepsilon (b_{6}, c_{6}) - \varepsilon (a_{6}, c_{6}) + \varepsilon (a_{6}, b_{6}), \label{eq:O6_VIII_beta_f}
\end{align}
$(a_{i}, c_{i}) \neq (b_{i}, c_{i}), (a_{i}, b_{i})$ ($i = 5, 6$), and reduced $f(\gamma_{4} + \gamma_{5} + \gamma_{6})$ has two positive 2-terms and two negative 2-terms, we have the following cases.
\begin{itemize}
\item[(1)]
$(b_{5}, c_{5}) = (b_{4}, a_{4})$ and $(b_{6}, c_{6}) = (a_{4}, b_{4})$.
\item[(2)]
$(b_{5}, c_{5}) = (b_{4}, a_{4})$ and $(a_{6}, b_{6}) = (a_{4}, b_{4})$.
\item[(3)]
$(a_{5}, b_{5}) = (b_{4}, a_{4})$ and $(b_{6}, c_{6}) = (a_{4}, b_{4})$.
\item[(4)]
$(a_{5}, b_{5}) = (b_{4}, a_{4})$ and $(a_{6}, b_{6}) = (a_{4}, b_{4})$.
\item[(5)]
$(b_{5}, c_{5}) = (b_{4}, a_{4})$ or $(a_{4}, b_{4})$, and $(b_{6}, c_{6}) = (a_{5}, c_{5})$.
\item[(6)]
$(b_{5}, c_{5}) = (b_{4}, a_{4})$ or $(a_{4}, b_{4})$, and $(a_{6}, b_{6}) = (a_{5}, c_{5})$.
\item[(7)]
$(b_{5}, c_{5}) = (b_{4}, a_{4})$ or $(a_{4}, b_{4})$, and $(a_{5}, b_{5}) = (a_{6}, c_{6})$.
\item[(8)]
$(a_{5}, b_{5}) = (b_{4}, a_{4})$ or $(a_{4}, b_{4})$, and $(b_{6}, c_{6}) = (a_{5}, c_{5})$.
\item[(9)]
$(a_{5}, b_{5}) = (b_{4}, a_{4})$ or $(a_{4}, b_{4})$, and $(a_{6}, b_{6}) = (a_{5}, c_{5})$.
\item[(10)]
$(a_{5}, b_{5}) = (b_{4}, a_{4})$ or $(a_{4}, b_{4})$, and $(b_{5}, c_{5}) = (a_{6}, c_{6})$.
\item[(11)]
$a_{6} = b_{5}$, $b_{6} = a_{5}$ and $c_{6} = c_{5}$.
\item[(12)]
$a_{6} = a_{5}$, $b_{6} = c_{5}$ and $c_{6} = b_{5}$.
\end{itemize}
In the cases (2), (3), (5), (7) and (9)--(12), we have no candidates of $\gamma_{4} + \gamma_{5} + \gamma_{6}$.

In the case (1), we have
\begin{itemize}
\item
$\gamma_{4} + \gamma_{5} + \gamma_{6} = + \> \langle u; \omega_{1}, \omega_{2} \rangle - (u; \omega_{0}, \omega_{1}, \omega_{2}) - (u; \omega_{3}, \omega_{2}, \omega_{1})$ or
\item
$\gamma_{4} + \gamma_{5} + \gamma_{6} = - \> \langle u; \omega_{1}, \omega_{2} \rangle + (u; \omega_{0}, \omega_{2}, \omega_{1}) + (u; \omega_{3}, \omega_{1}, \omega_{2})$.
\end{itemize}
In each case, we have $g(\gamma_{4} + \gamma_{5} + \gamma_{6} + \gamma_{7}) \neq 0$ contradicting to Lemma \ref{lem:3-cycle_condition} (iii).

In the case (4), we have
\begin{itemize}
\item
$\gamma_{4} + \gamma_{5} + \gamma_{6} = + \> \langle u; \omega_{0}, \omega_{3} \rangle - (u; \omega_{0}, \omega_{3}, \omega_{2}) - (u; \omega_{3}, \omega_{0}, \omega_{1})$ or
\item
$\gamma_{4} + \gamma_{5} + \gamma_{6} = - \> \langle u; \omega_{0}, \omega_{3} \rangle + (u; \omega_{3}, \omega_{0}, \omega_{2}) + (u; \omega_{0}, \omega_{3}, \omega_{1})$.
\end{itemize}
In each case, we have $g(\gamma_{4} + \gamma_{5} + \gamma_{6} + \gamma_{7}) \neq 0$ contradicting to Lemma \ref{lem:3-cycle_condition} (iii).

In the case (6), we have
\begin{itemize}
\item
$\gamma_{4} + \gamma_{5} + \gamma_{6} = + \> \langle u; \omega_{1}, \omega_{3} \rangle - (u; \omega_{0}, \omega_{1}, \omega_{3}) - (u; \omega_{0}, \omega_{3}, \omega_{2})$,
\item
$\gamma_{4} + \gamma_{5} + \gamma_{6} = + \> \langle u; \omega_{0}, \omega_{2} \rangle - (u; \omega_{3}, \omega_{2}, \omega_{0}) - (u; \omega_{3}, \omega_{0}, \omega_{1})$,
\item
$\gamma_{4} + \gamma_{5} + \gamma_{6} = - \> \langle u; \omega_{2}, \omega_{3} \rangle + (u; \omega_{0}, \omega_{2}, \omega_{3}) + (u; \omega_{0}, \omega_{3}, \omega_{1})$, or
\item
$\gamma_{4} + \gamma_{5} + \gamma_{6} = - \> \langle u; \omega_{0}, \omega_{1} \rangle + (u; \omega_{3}, \omega_{1}, \omega_{0}) + (u; \omega_{3}, \omega_{0}, \omega_{2})$.
\end{itemize}
The last case is nothing less than the case (\ref{eq:O6_VIII_i}) or (\ref{eq:O6_VIII_i'}).
In the other cases, we have $g(\gamma_{4} + \gamma_{5} + \gamma_{6} + \gamma_{7}) \neq 0$ contradicting to Lemma \ref{lem:3-cycle_condition} (iii).

In the case (8), we have
\begin{itemize}
\item
$\gamma_{4} + \gamma_{5} + \gamma_{6} = + \> \langle u; \omega_{1}, \omega_{3} \rangle - (u; \omega_{1}, \omega_{3}, \omega_{2}) - (u; \omega_{0}, \omega_{1}, \omega_{2})$,
\item
$\gamma_{4} + \gamma_{5} + \gamma_{6} = + \> \langle u; \omega_{0}, \omega_{2} \rangle - (u; \omega_{2}, \omega_{0}, \omega_{1}) - (u; \omega_{3}, \omega_{2}, \omega_{1})$,
\item
$\gamma_{4} + \gamma_{5} + \gamma_{6} = - \> \langle u; \omega_{0}, \omega_{1} \rangle + (u; \omega_{1}, \omega_{0}, \omega_{2}) + (u; \omega_{3}, \omega_{1}, \omega_{2})$, or
\item
$\gamma_{4} + \gamma_{5} + \gamma_{6} = - \> \langle u; \omega_{2}, \omega_{3} \rangle + (u; \omega_{2}, \omega_{3}, \omega_{1}) + (u; \omega_{0}, \omega_{2}, \omega_{1})$.
\end{itemize}
In each case, we have $g(\gamma_{4} + \gamma_{5} + \gamma_{6} + \gamma_{7}) \neq 0$ contradicting to Lemma \ref{lem:3-cycle_condition} (iii).

\subsection{Case (ii)}

If $c = \omega_{2}$, we let
\begin{align*}
 & \gamma_{4}^{\prime \prime} + \gamma_{5}^{\prime \prime} + \gamma_{6}^{\prime \prime} = \\
 & + (n + 1, u^{\omega_{0}}; \omega_{1}, \omega_{0}, \omega_{2}) - (n + 1, u^{\omega_{1}}; \omega_{5}, \omega_{0}, \omega_{2}) - (n + 1, u^{\omega_{2}}; \omega_{1}, \omega_{3}, \omega_{1}).
\end{align*}
Then, as we saw in the proof of Proposition \ref{prop:O6_VII}, we have
\[
 \gamma_{1} + \gamma_{2} + \gamma_{3} + \gamma_{4}^{\prime \prime} + \gamma_{5}^{\prime \prime} + \gamma_{6}^{\prime \prime}
 = \partial (n, u; \omega_{0}, \omega_{1}, \omega_{0}, \omega_{2}).
\]
Therefore, for the same reason as in the case (i), we have $\eta(\pi(\gamma)) = 0$ contradicting to our assumption.
Thus, we assume that $c \neq \omega_{2}$.

Consider the following four 3-terms
\begin{align*}
 \lambda_{4} & = + (n + 1, u^{\omega_{0}}; \omega_{1}, \omega_{0}, c), &
 \lambda_{5} & = + (n + 1, u^{\omega_{0}}; \omega_{0}, \omega_{2}, c), \\
 \lambda_{6} & = - (n + 1, u^{\omega_{1}}; \omega_{5}, \omega_{0}, c), &
 \lambda_{7} & = - (n + 1, u^{c}; \omega_{0}^{c}, \omega_{1}^{c}, \omega_{0}^{c}).
\end{align*}
We note that if
\begin{equation}
 \gamma_{4} + \gamma_{5} + \gamma_{6} + \gamma_{7} = \lambda_{4} + \lambda_{5} + \lambda_{6} + \lambda_{7}, \label{eq:O6_VIII_ii}
\end{equation}
then we have $\gamma = \partial (n, u; \omega_{0}, \omega_{1}, \omega_{0}, c)$ contradicting to our assumption.

In the case other than (a) or (b) introduced in the proof of Lemma \ref{lem:O6_f3_usual}, for the same reason as in the case (i), we may assume that
\begin{align*}
 f(\gamma_{4} + \gamma_{5}) & = f(\lambda_{4} + \lambda_{5}), &
 \gamma_{6} & = \lambda_{6}, &
 \gamma_{7} & = - \> \langle n + 1, u^{c}; \omega_{0}^{c}, \omega_{1}^{c} \rangle.
\end{align*}
We note that if
\begin{equation}
 \gamma_{4} + \gamma_{5} + \gamma_{6} + \gamma_{7} = \lambda_{5} + \lambda_{6} + \lambda_{7} - (n + 1, u^{c}; \omega_{1}^{c}, \omega_{0}^{c}, \omega_{1}^{c}), \label{eq:O6_VIII_ii'}
\end{equation}
we have $g(\gamma_{4} + \gamma_{5} + \gamma_{6} + \gamma_{7}) \neq 0$ contradicting to Lemma \ref{lem:3-cycle_condition} (iii).
Since $f(\gamma_{4} + \gamma_{5}) = f(\lambda_{4} + \lambda_{5})$, we have the cases (A) and (B) replacing $\delta_{i}$ with $\lambda_{i}$ other than $\gamma_{4} + \gamma_{5} = \lambda_{4} + \lambda_{5}$.
For the same reason as in the case (i), we have no candidates of $\gamma_{4} + \gamma_{5}$ in the case (A), and
\[
 \gamma_{4} + \gamma_{5} = + \> (n + 1, u^{\omega_{0}}; \omega_{1}, \omega_{0}, \omega_{2}) + (n + 1, u^{\omega_{0}}; \omega_{1}, \omega_{2}, c)
\]
in the case (B) which yields $g(\gamma_{4} + \gamma_{5} + \gamma_{6} + \gamma_{7}) \neq 0$ contradicting to Lemma \ref{lem:3-cycle_condition} (iii).

In the case (b) introduced in the proof of Lemma \ref{lem:O6_f3_usual}, for the same reason as in the case (i), we may assume that
\begin{align*}
 f(\gamma_{4} + \gamma_{5}) & = f(\lambda_{4} + \lambda_{5}), &
 f(\gamma_{6} + \gamma_{7}) & = f(\lambda_{6} + \lambda_{7}),
\end{align*}
as mentioned in the proof of Lemma \ref{lem:O6_f3_usual}.
Therefore, in light of the above argument, we have
\begin{itemize}
\item[(C')]
$\gamma_{4} + \gamma_{5} = \lambda_{4} + \lambda_{5}$ or
\item[(D')]
$\gamma_{4} + \gamma_{5} = + \> (n + 1, u^{\omega_{0}}; \omega_{1}, \omega_{0}, \omega_{2}) + (n + 1, u^{\omega_{0}}; \omega_{1}, \omega_{2}, \omega_{4})$, and
\end{itemize}
(E), (F) or (G) replacing $\delta_{i}$ with $\lambda_{i}$.
For the same reason as in the case (i), we have
\[
 \gamma_{6} + \gamma_{7} = \lambda_{6} - (n + 1, u; \omega_{1}, \omega_{2}, \omega_{1})
\]
in the case (F), and no candidates of $\gamma_{6} + \gamma_{7}$ in the case (G).
Therefore, it is sufficient to consider the cases (C')-(E), (C')-(F), (D')-(E) and (D')-(F).
The first and second cases are nothing less than the cases (\ref{eq:O6_VIII_ii}) and (\ref{eq:O6_VIII_ii'}), respectively.
It is easy to see that we have $g(\gamma_{4} + \gamma_{5} + \gamma_{6} + \gamma_{7}) \neq 0$ contradicting to Lemma \ref{lem:3-cycle_condition} (iii) in the other cases.

In the case (a) introduced in the proof of Lemma \ref{lem:O6_f3_usual}, for the same reason as in the case (i), we may assume ($\alpha$) or ($\beta$) replacing $\delta_{i}$ with $\lambda_{i}$.

In the case ($\alpha$), we have the case (H) or (I) replacing $\delta_{i}$ with $\lambda_{i}$ by Lemma \ref{lem:f-connected1}.
In the case (H), since mutually distinct four elements of $O_{6}$ are appeared in any two of $\lambda_{4}, \lambda_{5}$ and $\lambda_{6}$ as entries of their colors, we have no candidates of $\gamma_{4} + \gamma_{5}$ by Lemma \ref{lem:f-connected3}.
In the case (I), since mutually distinct five elements of $O_{6}$ are appeared in $\lambda_{4}, \lambda_{5}$ and $\lambda_{7}$ as entries of their colors, we have no candidates of $\gamma_{4} + \gamma_{5}$ by Lemma \ref{lem:f-connected5}.

In the case ($\beta$), since reduced $f(\lambda_{4} + \lambda_{5} + \lambda_{7})$ has three positive 2-terms and three negative 2-terms (see the formula (\ref{eq:O6_f3_usual_ii_a})), we have
\[
 (N_{01}^{+}, N_{01}^{-}, N_{23}^{+}, N_{23}^{-}) = (1, 0, 0, 2) \ \text{or} \ (0, 1, 2, 0)
\]
in a similar way as in the case (i).
Therefore, we may assume the formula (\ref{eq:O6_VIII_beta}).
Since we have the formula (\ref{eq:O6_VIII_beta_f}), $(a_{i}, c_{i}) \neq (b_{i}, c_{i}), (a_{i}, b_{i})$ ($i = 5, 6$), and reduced $f(\gamma_{4} + \gamma_{5} + \gamma_{6})$ has three positive 2-terms and three negative 2-terms, we essentially have the following cases.
\begin{itemize}
\item[(13)]
$(b_{5}, c_{5}) = (b_{4}, a_{4})$ or $(a_{4}, b_{4})$.
\item[(14)]
$(a_{5}, b_{5}) = (b_{4}, a_{4})$ or $(a_{4}, b_{4})$.
\item[(15)]
$(b_{5}, c_{5}) = (a_{6}, c_{6})$.
\item[(16)]
$(a_{5}, b_{5}) = (a_{6}, c_{6})$.
\end{itemize}
In the cases (13) and (14), we have no candidates of $\gamma_{4} + \gamma_{5} + \gamma_{6}$.
In the case (15), we have
\[
 \gamma_{4} + \gamma_{5} + \gamma_{6} = - \> \langle u; \omega_{0}, \omega_{5} \rangle + (u; \omega_{1}, \omega_{0}, \omega_{3}) + (u; \omega_{0}, \omega_{2}, \omega_{3})
\]
which is nothing less than the case (\ref{eq:O6_VIII_ii}) or (\ref{eq:O6_VIII_ii'}).
In the case (16), we have
\[
 \gamma_{4} + \gamma_{5} + \gamma_{6} = - \> \langle u; \omega_{0}, \omega_{5} \rangle + (u; \omega_{1}, \omega_{2}, \omega_{3}) + (u; \omega_{1}, \omega_{0}, \omega_{2})
\]
which yields $g(\gamma_{4} + \gamma_{5} + \gamma_{6} + \gamma_{7}) \neq 0$ contradicting to Lemma \ref{lem:3-cycle_condition} (iii).

\section{Upper bounds of the lengths of the 3-cocycles}
\label{sec:upper_bounds_of_zeta_and_eta}

The aim of this section is to prove Theorems \ref{thm:R7} and \ref{thm:O6} giving upper bounds of the lengths of the Mochizuki 3-cocycle $\zeta$ of the 7-dihedral quandle $R_{7}$ and the 3-cocycle $\eta$ of the octahedral quandle $O_{6}$.
We also see that those proofs naturally yield the claims in Corollaries \ref{cor:2-twist-spun_5_2} and \ref{cor:4-twist-spun_3_1}.
To achieve the upper bounds, we start with reviewing that we have an inequality between the length of a 3-cocycle of a quandle and the triple point number of a surface knot, which was given by Satoh in \cite{Sat2016}.
We will use several basic terminologies and facts on surface knots without explanations.
We refer the reader \cite{CS1998, Kam2017} for more details about surface knots.

Let $F$ be an oriented surface knot and $D$ a diagram of $F$.
Consider to assign elements of a quandle $X$ to the sheets of $D$.
This assignment is said to be an \emph{$X$-coloring} of $D$ if it satisfies the condition depicted in the left-hand side of Figure \ref{fig:coloring} around each double point of $D$.
In the figure, the arrow perpendicular to the upper sheet denotes the normal orientation of the sheet, and $a, b$ and $a^{b}$ the elements of $X$ assigned to the corresponding sheets.
Associated with a 3-cocycle $\theta$ of $X$, we define the \emph{weight} of an $X$-coloring of $D$ by summing up the \emph{local weights} $w$ depicted in the middle and right-hand side of Figure \ref{fig:coloring} over all triple points of $D$.
It is known that the multi-set consisting of weights of all $X$-colorings of $D$ associated with $\theta$ is invariant under Roseman moves (see \cite{Kam2017} for example).
We call this multi-set a \emph{cocycle invariant} of $F$ associated with $\theta$.
\begin{figure}[htbp]
 \centering
 \includegraphics[scale=0.25]{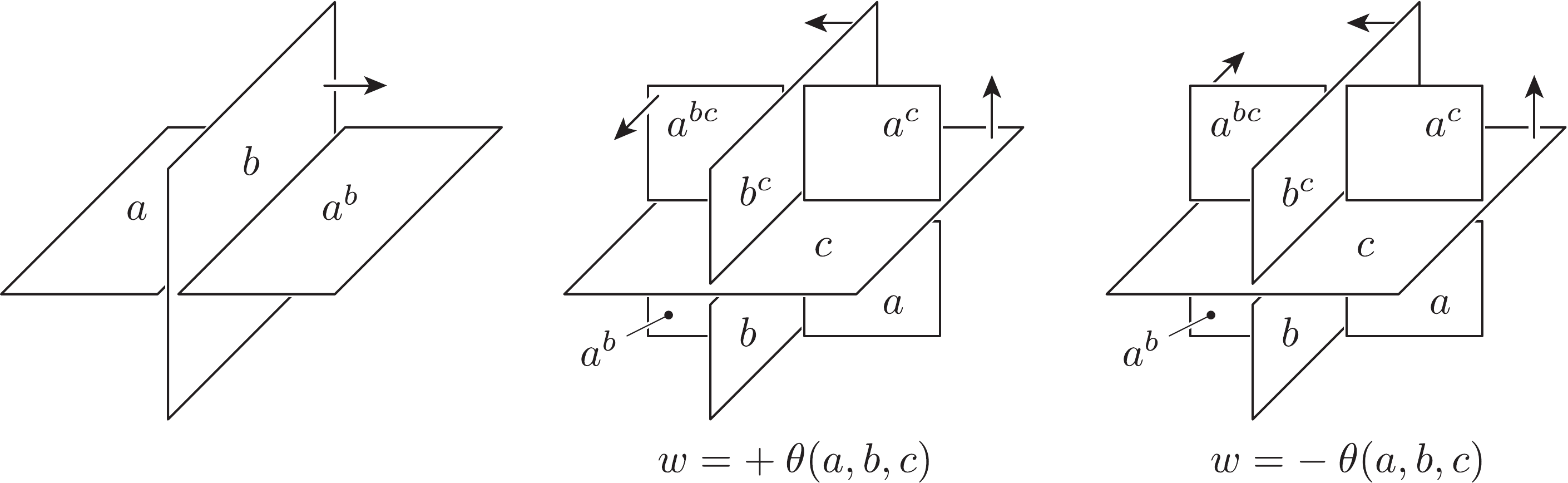}
 \caption{The $X$-coloring condition (left) and local weights derived from positive and negative triple points (middle and right).}
 \label{fig:coloring}
\end{figure}

Let $t(F)$ denote the triple point number of $F$, that is, the minimal number of triple points among all diagrams of $F$.
Then, Satoh showed in \cite{Sat2016} the following theorem.

\begin{theorem}[Theorem 2.2 of \cite{Sat2016}]
\label{thm:upper_bound}
Let $X$ be a quandle, $\theta$ a 3-cocycle of $X$, and $F$ an oriented surface knot.
If the cocycle invariant of $F$ associated with $\theta$ contains a non-zero element, then we have $l(\theta) \leq t(F)$.
\end{theorem}

We omit the proof.
In light of this theorem and Theorem \ref{thm:lower_bound_R7} or \ref{thm:lower_bound_O6}, if there is a surface knot whose triple point number is at most eight and cocycle invariant associated with $\zeta$ or $\eta$ contains a non-zero element, then we have Theorem \ref{thm:R7} or \ref{thm:O6}, respectively.
Thus, our last task is to give us such surface knots concretely.
To accomplish it, we start with reviewing the following methods to obtain surface knots.

Let $\mathbb{R}^{3}_{+}$ denote the upper half-space $\{ (x, y, z) \in \mathbb{R}^{3} \mid z \geq 0 \}$.
Consider a knotted arc $k$ which is smoothly and properly embedded in $\mathbb{R}^{3}_{+}$.
We let $K$ denote the classical knot which is obtained from $k$ by gluing its two ends together.
Rotate $\mathbb{R}^{3}_{+}$ 360 degrees in $\mathbb{R}^{4}$ along $\partial \mathbb{R}^{3}_{+}$.
Then, the locus of $k$ yields a $S^{2}$-knot, known as the \emph{spun} $K$ \cite{Art1926}.
Furthermore, let $B$ be a 3-ball in the interior of $\mathbb{R}^{3}_{+}$ which completely contains the knotted part of $k$.
We assume that $k$ intersects with $\partial B$ only at two antipodal points.
Rotate $B$ $360 n$ degrees ($n \geq 1$) along the line passing through the antipodal points, while rotating $\mathbb{R}^{3}_{+}$ 360 degrees.
Then the locus of $k$ also yields a $S^{2}$-knot, known as the \emph{$n$-twist-spun} $K$ \cite{Zee1965}.
Let $\Pi$ be an orthogonal projection from $\mathbb{R}^{3}_{+}$ to $\mathbb{R}^{2}_{+} = \{ (x, y) \in \mathbb{R}^{2} \mid y \geq 0 \}$ given by $\Pi(x, y, z) = (x, z)$.
Rotating $\mathbb{R}^{2}_{+}$ 360 degrees in $\mathbb{R}^{3}$ along $\partial \mathbb{R}^{2}_{+}$ simultaneously with $\mathbb{R}^{3}_{+}$, we obtain the locus of $\Pi(k)$ in $\mathbb{R}^{3}$.
Obviously, this locus of $\Pi(k)$ with height information yields a diagram of the $S^{2}$-knot in nearly every case.
In the remaining, we call the one-parameter family of $\Pi(k)$ with height information a \emph{motion picture} of the $S^{2}$-knot.

Satoh and Shima essentially shown in \cite{SS2004} that Figure \ref{fig:motion_picture_of_3_1} depicts a motion picture of the 4-twist-spun trefoil.
It is easy to see (as mentioned in Lemma 6.1 of \cite{SS2004}) that the diagram which Figure \ref{fig:motion_picture_of_3_1} yields has eight triple points, each of which respectively appears during (a), (b), (d), (e), (g), (h), (j) or (k).
Therefore, the triple point number of the 4-twist-spun trefoil is at most eight.
Similarly, in light of Theorem 1 of Yashiro's paper \cite{Yas2005}, Figure \ref{fig:motion_picture_of_5_2} depicts a motion picture of the 2-twist-spun $5_{2}$-knot.
It is easy to see that the diagram which Figure \ref{fig:motion_picture_of_5_2} yields has eight triple points, two of which respectively appear during (a), (b), (d) or (e).
Therefore, the triple point number of the 2-twist-spun $5_{2}$-knot is also at most eight.
\begin{figure}[htbp]
 \centering
 \includegraphics[scale=0.2]{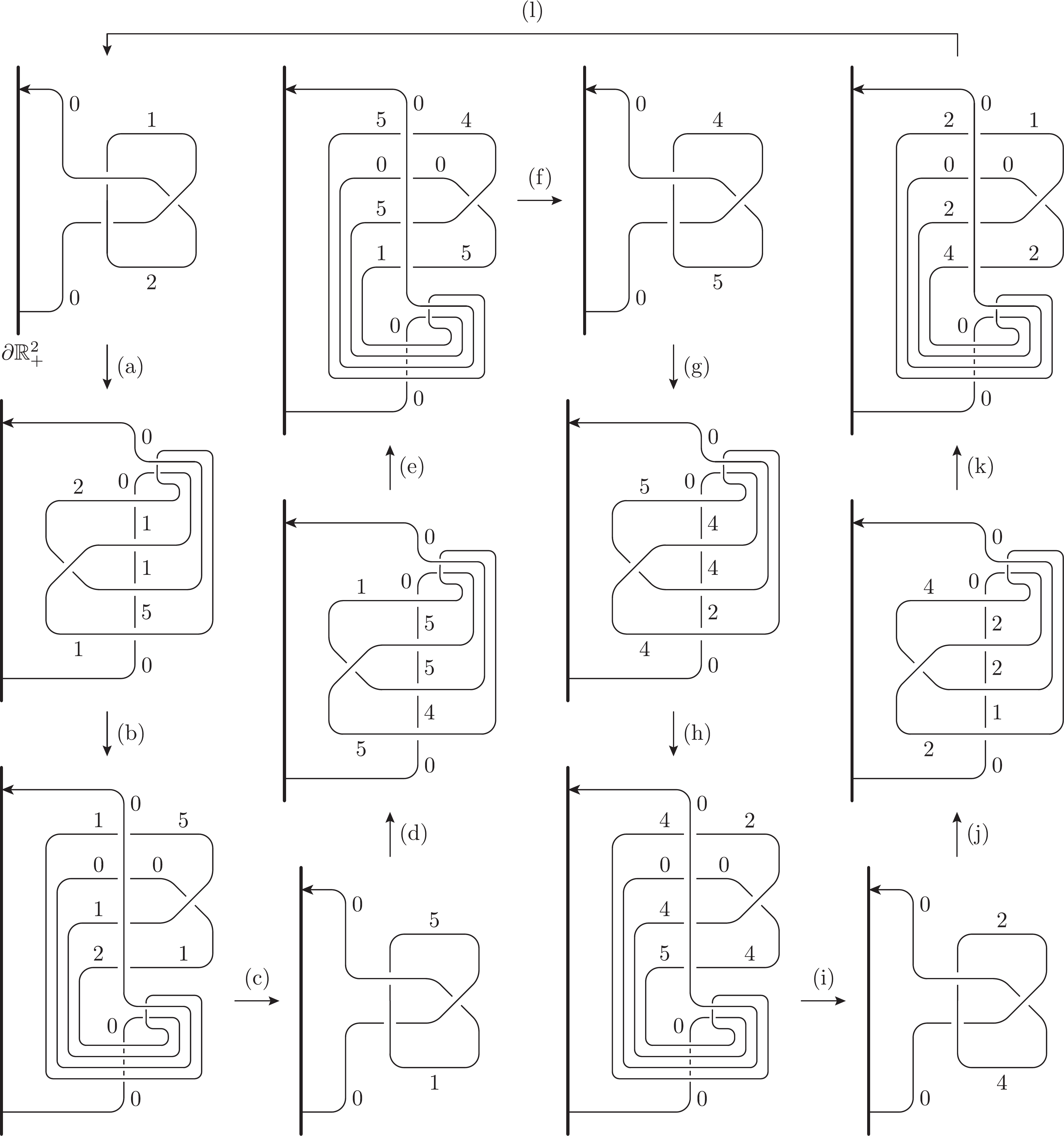}
 \caption{A motion picture of the 4-twist-spun trefoil.}
 \label{fig:motion_picture_of_3_1}
\end{figure}
\begin{figure}[htbp]
 \centering
 \includegraphics[scale=0.2]{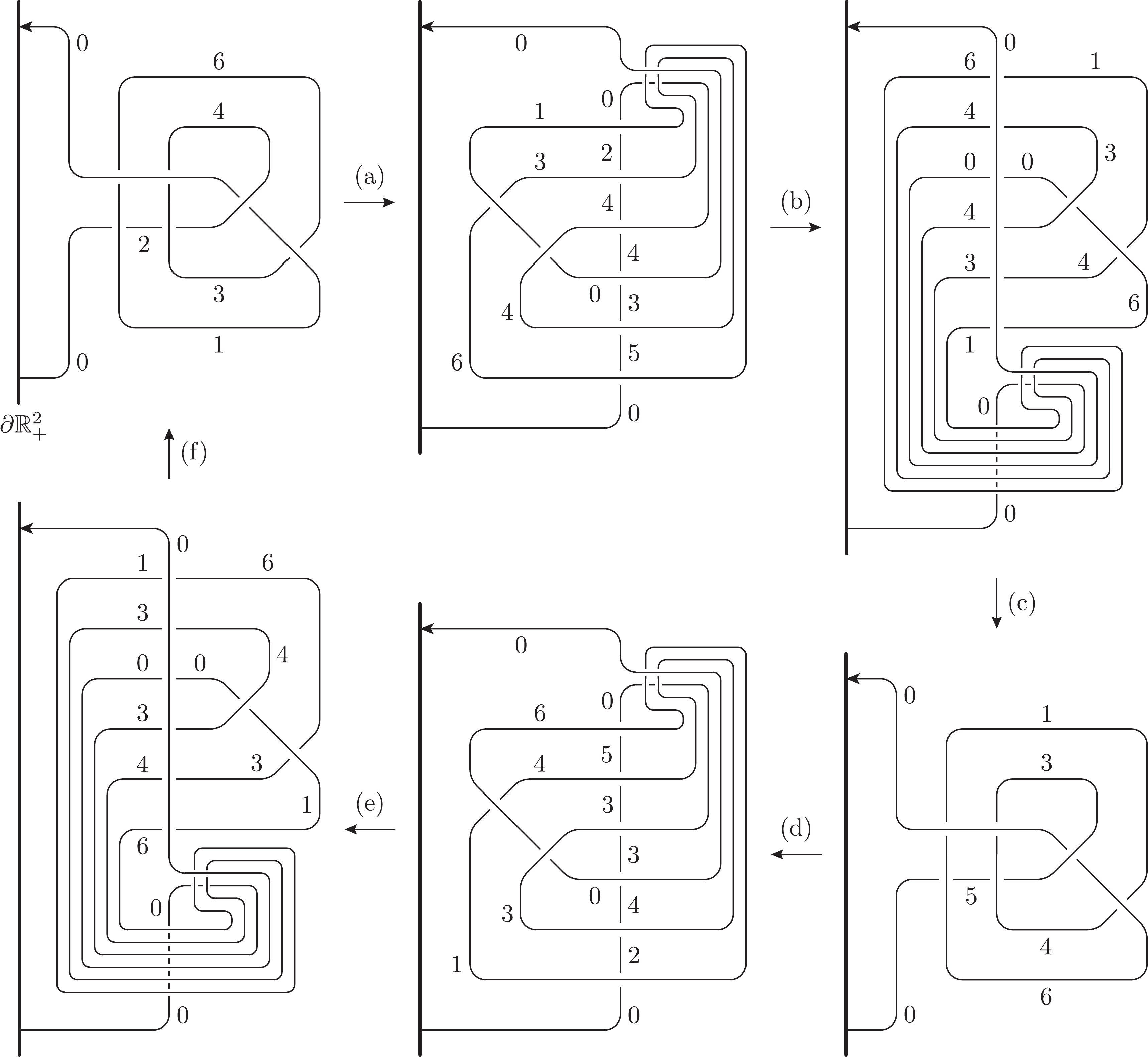}
 \caption{A motion picture of the 2-twist-spun $5_{2}$.}
 \label{fig:motion_picture_of_5_2}
\end{figure}

We are now ready to prove Theorems \ref{thm:R7} and \ref{thm:O6}.

\begin{proof}[Proof of Theorem \ref{thm:R7}]
Let $D$ denote the diagram of the 2-twist-spun $5_{2}$-knot which Figure \ref{fig:motion_picture_of_5_2} yields.
Assign elements of $R_{7}$ to sheets of $D$ as depicted in Figure \ref{fig:motion_picture_of_5_2}.
Then, this assignment is uniquely extended to an $R_{7}$-coloring of $D$.
Furthermore, it is routine to check that the weight of this $R_{7}$-coloring associated with $\zeta$ is
\begin{align*}
 & + \zeta(0, 6, 1) - \zeta(5, 1, 4) - \zeta(1, 6, 0) + \zeta(6, 3, 0) \\
 & \quad + \zeta(0, 1, 6) - \zeta(2, 6, 3) - \zeta(6, 1, 0) + \zeta(1, 4, 0)) = 6.
\end{align*}
Thus, the cocycle invariant of the 2-twist-spun $5_{2}$-knot associated with $\zeta$ contains a non-zero element.
Therefore, in light of Theorem \ref{thm:upper_bound}, we have
\begin{equation}
 l(\zeta) \leq t(\text{2-twist-spun $5_{2}$-knot}) \leq 8. \label{eq:5_2}
\end{equation}
On the other hand, in light of Theorem \ref{thm:lower_bound_R7}, we have
\[
 8 \leq l(\zeta, \mathbb{Z} \times R_{7}) \leq l(\zeta).
\]
We thus have $l(\zeta) = 8$.
\end{proof}

\begin{proof}[Proof of Theorem \ref{thm:O6}]
Let $D$ denote the diagram of the 4-twist-spun trefoil which Figure \ref{fig:motion_picture_of_3_1} yields.
Assign elements of $O_{6}$ to sheets of $D$ as depicted in Figure \ref{fig:motion_picture_of_3_1}.
Then, this assignment is uniquely extended to an $O_{6}$-coloring of $D$.
Furthermore, it is routine to check that the weight of this $O_{6}$-coloring associated with $\eta$ is
\begin{align*}
 & - \eta(0, 2, 1) + \eta(1, 5, 0) - \eta(0, 1, 5) + \eta(5, 4, 0) \\
 & \quad - \eta(0, 5, 4) + \eta(4, 2, 0) - \eta(0, 4, 2) + \eta(2, 1, 0)) = 1.
\end{align*}
Thus, the cocycle invariant of the 4-twist-spun trefoil associated with $\eta$ contains a non-zero element.
Therefore, in light of Theorem \ref{thm:upper_bound}, we have
\begin{equation}
 l(\eta) \leq t(\text{4-twist-spun trefoil}) \leq 8. \label{eq:3_1}
\end{equation}
On the other hand, in light of Theorem \ref{thm:lower_bound_O6}, we have
\[
 8 \leq l(\eta, \mathbb{Z} \times O_{6}) \leq l(\eta).
\]
We thus have $l(\eta) = 8$.
\end{proof}

In light of Theorem \ref{thm:R7} or \ref{thm:O6} and inequality (\ref{eq:5_2}) or (\ref{eq:3_1}), we immediately have Corollary \ref{cor:2-twist-spun_5_2} or \ref{cor:4-twist-spun_3_1}, respectively.

\section*{Acknowledgments}
This work was supported by JSPS KAKENHI Grant Number JP25K07014.

\appendix

\section{}
\label{app:counter_example}

The aim of this appendix is to introduce a concrete $R_{3}$-coloring of a diagram of a surface knot, whose existence might contradict to the claim of Theorem 1.1 of Yashiro's paper \cite{Yas2016}.
In another paper \cite{Yas2018}, Yashiro introduced an invariant $\mu$ of surface knots, and showed that the $2k$-twist-spun trefoil has the triple point number $4k$ for each $k \geq 1$ (Corollary 1.2 of \cite{Yas2018}) utilizing this invariant $\mu$.
Although invariance (well-definedness) of $\mu$ is claimed as Lemma 2.5 of \cite{Yas2018}, it seems that the argument in the proof of the lemma depends on exactness of the proof of Theorem 1.1 of \cite{Yas2016}.
Therefore, the existence of this $R_{3}$-coloring might also contradict to invariance of $\mu$.
We will use several terminologies and facts on surface knots and quandle colorings without explanations.
We refer the reader \cite{CS1998, Kam2017} for more details.

Let $\mathbb{R}^{3}_{+}$ denote the upper half-space again and $K$ be an oriented trefoil located in the box $\{ (x, y, z) \in \mathbb{R}^{3}_{+} \mid -1 \leq x \leq 1, \, -1 \leq y \leq 1, \, 2 \leq z \leq 4 \}$.
We assume that the left-hand side of Figure \ref{fig:diagram_of_trefoil_and_part_of_diagram} depicts the diagram of $K$ derived from the projection along the $x$-axis.
Let $F$ be the $T^{2}$-knot which is the locus of $K$ under spinning of $\mathbb{R}^{3}_{+}$ along $\partial \mathbb{R}^{3}_{+}$ in $\mathbb{R}^{4}$.
Here, $T^{2}$ denotes the torus as usual.
Moreover, allocate a 2-sphere $S = \{ (x, y, z, w) \in \mathbb{R}^{4} \mid x = -2, \, y^{2} + (z - 3)^{2} + w^{2} = 2^{2} \}$ in the ambient space of $F$.
Then, we have a $T^{2}$-knot $F^{\prime}$ connecting $S$ and $F$ by a tube so that the right-hand side of Figure \ref{fig:diagram_of_trefoil_and_part_of_diagram} depicts a part of the diagram $D^{\prime}$ of $F^{\prime}$ derived from the projection along the $x$-axis.
Obviously, $F^{\prime}$ is ambient isotopic to $F$.
Assigning elements of $R_{3}$ to the sheets of $D^{\prime}$ as depicted in the right-hand side of Figure \ref{fig:diagram_of_trefoil_and_part_of_diagram}, we have an $R_{3}$-coloring $\mathrm{Col}^{\prime}$ of $D^{\prime}$.
\begin{figure}[htbp]
 \centering
 \includegraphics[scale=0.25]{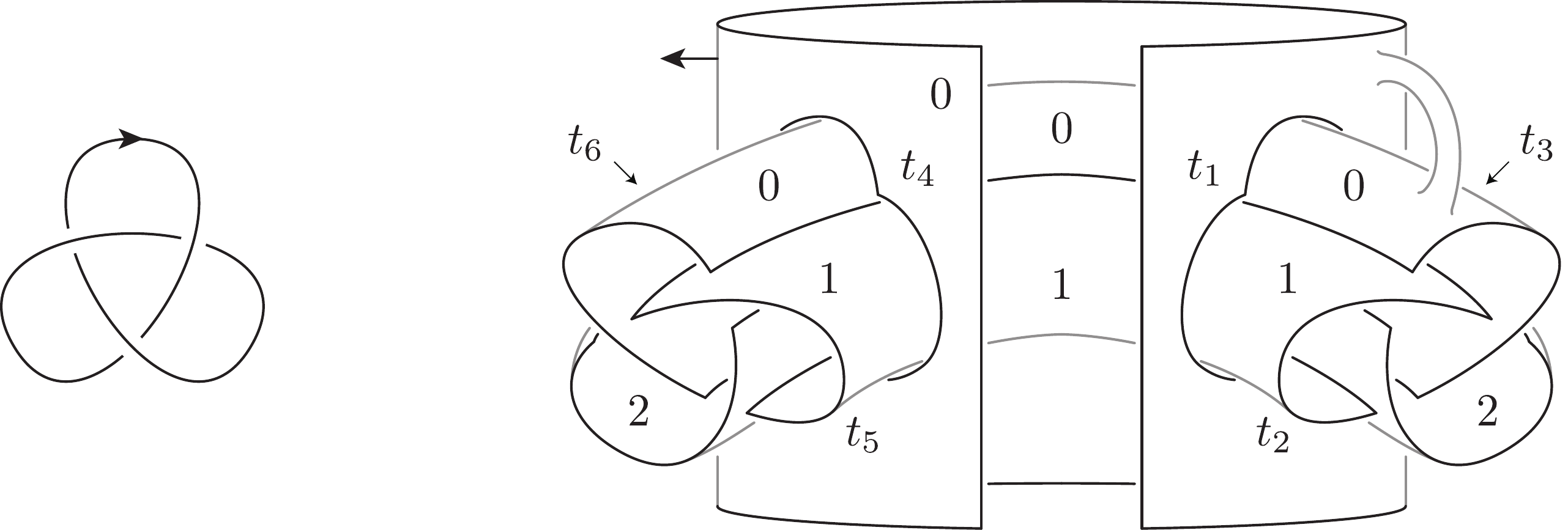}
 \caption{The diagram of $K$ (left) and a part of $D^{\prime}$ (right)}
 \label{fig:diagram_of_trefoil_and_part_of_diagram}
\end{figure}

Let us discuss why the existence of $\mathrm{Col}^{\prime}$ might contradict to the claim of Theorem 1.1 of \cite{Yas2016}.
We will use terminologies and notations introduced in \cite{Yas2016} without explanations.
Since the left-hand side of Figure \ref{fig:lower_decker_set_and_rcc_of_D_prime} depicts the double decker diagram of $F^{\prime}$ derived from the projection along the $x$-axis, we have the rectangular-cell-complex $K_{D^{\prime}}$ induced from $D^{\prime}$ depicted in the right-hand side of Figure \ref{fig:lower_decker_set_and_rcc_of_D_prime}.
In the figure, $T^{2}$ is unfolded into a rectangle in a usual way, both of the crossing $p_{i}$ and the rectangle $\sigma_{i}$ correspond to the triple point $t_{i}$ of $D^{\prime}$, and both of the lower decker set and $K_{D^{\prime}}$ are colored by $R_{3}$ related to $\mathrm{Col}^{\prime}$.
Consider the following fundamental chains of $K_{D^{\prime}}$:
\begin{align*}
 c_{1} & = + \> \sigma_{1} + \sigma_{2} + \sigma_{3}, &
 c_{2} & = - \> \sigma_{4} - \sigma_{5} - \sigma_{6}.
\end{align*}
Since
\[
 \mathrm{Col}^{\prime}_{\sharp} (c_{i})
 = (-1)^{i+1} ((2, 2, 1) + (2, 0, 2) + (2, 1, 0))
 = (-1)^{i+1} ((2, 0, 2) + (2, 1, 0)),
\]
it is routine to see that
\begin{align*}
 \partial \circ \mathrm{Col}^{\prime}_{\sharp} (c_{i}) & = 0, &
 \zeta_{3}(\mathrm{Col}^{\prime}_{\sharp} (c_{i})) & = (-1)^{i}.
\end{align*}
Here, $\mathrm{Col}^{\prime}_{\sharp}$ denotes the homomorphism $C^{\prime}_{2}(K_{D^{\prime}}) \to C^{Q}_{3}(R_{3})$ induced from $\mathrm{Col}^{\prime}$.
Therefore, $c_{1}$ and $c_{2}$ are pseudo-cycles in $K_{D^{\prime}}$ associated with $\mathrm{Col}^{\prime}$.\footnote{Although the condition (i) for a fundamental chain $c$ to be a pseudo-cycle associated with a coloring $\mathrm{Col}$ is given by $\mathrm{Col}_{\sharp} \circ \partial (c) = 0$ in Definition 4.2 of \cite{Yas2016}, it seems that this is a typo for $\partial \circ \mathrm{Col}_{\sharp} (c) = 0$ (see also Definition 2.1 of \cite{Yas2018}). Since $\partial (c_{i}) = 0$, our $c_{1}$ and $c_{2}$ are pseudo-cycles in $K_{D^{\prime}}$ associated with $\mathrm{Col}^{\prime}$ even if the condition (i) is correct.}
Since the minimal subcomplexes of $K_{D^{\prime}}$ respectively containing $\{ \sigma_{1}, \sigma_{2}, \sigma_{3} \}$ and $\{ \sigma_{4}, \sigma_{5}, \sigma_{6} \}$ only share a 0-cell, the maximal number of pseudo-cycles in $K_{D^{\prime}}$ associated with $\mathrm{Col}^{\prime}$ is at least two.\footnote{Although the definition of the maximal number of pseudo-cycles seems to be unclear in \cite{Yas2016}, it seems that we should consider ``separated'' pseudo-cycles to count the number (see also Definition 2.2 of \cite{Yas2018}). We note that $+ \> \sigma_{2} + \sigma_{3}$, $- \> \sigma_{5} - \sigma_{6}$, $+ \> \sigma_{1} + \sigma_{2} + \sigma_{3} - \sigma_{4}$, $+ \> \sigma_{1} - \sigma_{4} - \sigma_{5} - \sigma_{6}$, $+ \> \sigma_{2} + \sigma_{3} - \sigma_{4}$ and $+ \> \sigma_{1} - \sigma_{5} - \sigma_{6}$ are also pseudo-cycles in $K_{D^{\prime}}$ associated with $\mathrm{Col}^{\prime}$.}
On the other hand, let $D$ be the diagram of $F$ derived from the projection along the $x$-axis.
Since the left-hand side of Figure \ref{fig:lower_decker_set_and_rcc_of_D} depicts the double decker diagram of $F$ derived from the projection along the $x$-axis, we have the rectangular-cell-complex $K_{D}$ induced from $D$ depicted in the right-hand side of Figure \ref{fig:lower_decker_set_and_rcc_of_D}.
Since each 2-cell of $K_{D}$ is a bubble, for any quandle $X$ and $X$-coloring $\mathrm{Col}$ of $D$, the induced homomorphism $\mathrm{Col}_{\sharp}$ is the zero-map.
Therefore, the maximal number of pseudo-cycles in $K_{D}$ is always zero.
Since $F^{\prime}$ is ambient isotopic to $F$, it might contradict to the claim of Theorem 1.1 of \cite{Yas2016}: the maximal numbers of pseudo-cycles in $K_{D}$ and $K_{D^{\prime}}$, which are respectively associated with $\mathrm{Col}^{\prime}$ and an $R_{3}$-coloring of $D$ corresponding to $\mathrm{Col}^{\prime}$, are the same.
\begin{figure}[htbp]
 \centering
 \includegraphics[scale=0.25]{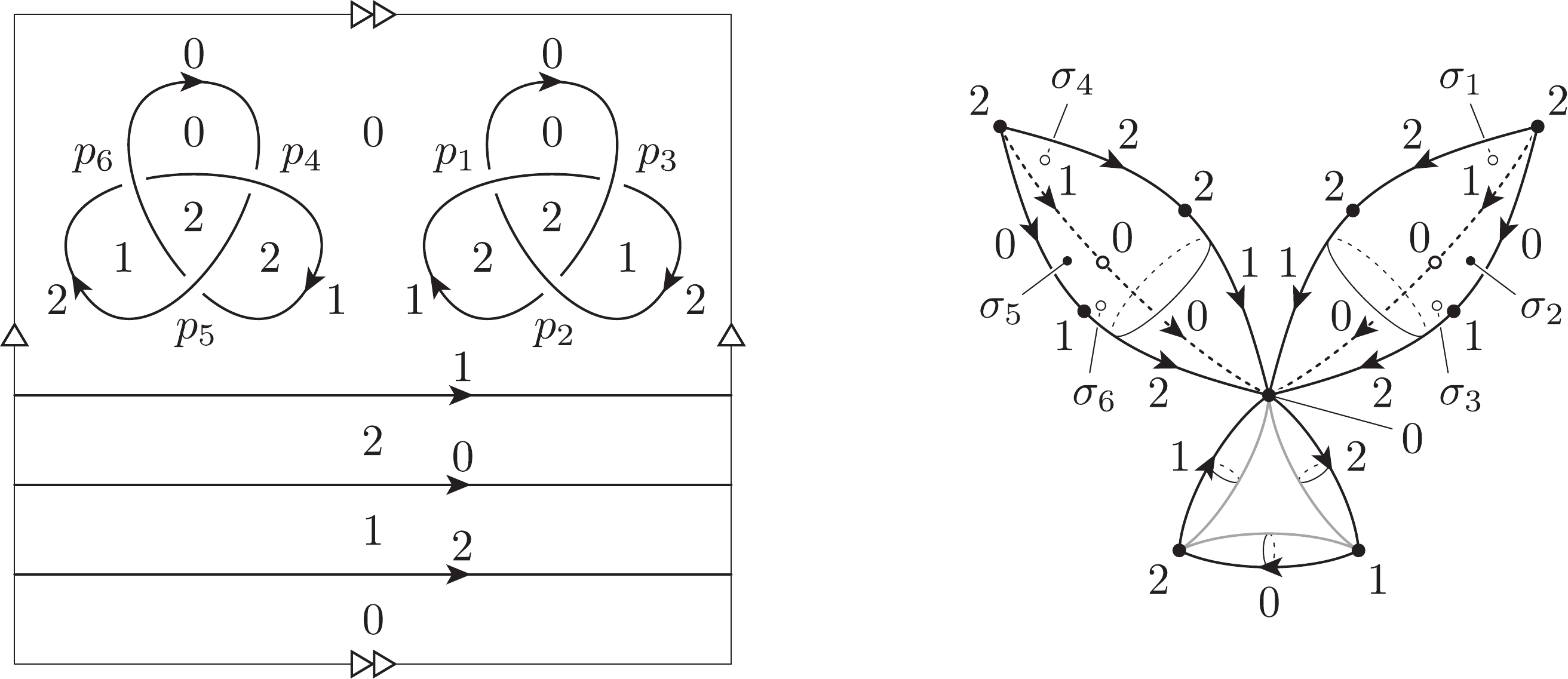}
 \caption{The lower decker diagram of $F^{\prime}$ derived from the projection along the $x$-axis (left) and the rectangular-cell-complex $K_{D^{\prime}}$ induced from $D^{\prime}$ (right)}
 \label{fig:lower_decker_set_and_rcc_of_D_prime}
\end{figure}
\begin{figure}[htbp]
 \centering
 \includegraphics[scale=0.25]{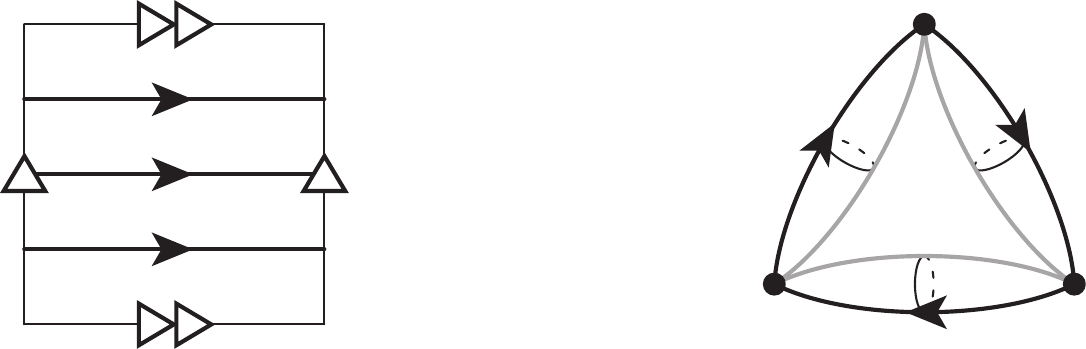}
 \caption{The lower decker diagram of $F$ derived from the projection along the $x$-axis (left) and the rectangular-cell-complex $K_{D}$ induced from $D$ (right)}
 \label{fig:lower_decker_set_and_rcc_of_D}
\end{figure}

\bibliographystyle{amsplain}

\end{document}